\documentclass[12pt, a4paper, leqno]{article}
\usepackage[margin=2.5cm]{geometry}
\usepackage[english]{babel}

\usepackage[runin]{abstract}
\usepackage{titling}

\usepackage{amsmath}
\usepackage{amsthm}
\usepackage{amsfonts}
\usepackage{amssymb}
\usepackage{stmaryrd}
\usepackage{bbm}
\usepackage{mathtools}
\usepackage{clrscode}
\usepackage{enumitem}
\usepackage{mdwlist}
\usepackage{diagrams}

\abslabeldelim{.}

\newcommand{\keywordsname}{Key words}
\makeatletter
\newcommand{\keywords}[1]{%
\begin{@bstr@ctlist}
\hspace*{\abstitleskip}{\abstractnamefont\keywordsname\@bslabeldelim}\abstracttextfont\
#1%
\par\end{@bstr@ctlist}
}
\makeatother

\newcommand{\subjclassname}{Mathematics subject classification}
\makeatletter
\newcommand{\subjclass}[2][2010]{%
\begin{@bstr@ctlist}
\hspace*{\abstitleskip}{\abstractnamefont\subjclassname\ (#1)\@bslabeldelim}\abstracttextfont\
#2%
\par\end{@bstr@ctlist}
}
\makeatother

\makeatletter
\def\and{
	\end{tabular}%
	and%
	\begin{tabular}[t]{c}}%
\makeatother

\makeatletter
\def\thanks#1{
\protected@xdef\@thanks{\@thanks
\protect\footnotetext[\the\c@footnote]{#1}}%
}
\makeatother

\makeatletter
\let\addresses\@empty      
\newcommand{\address}[2][]{\g@addto@macro\addresses{\address{#1}{#2}}}
\newcommand{\curraddr}[2][]{\g@addto@macro\addresses{\curraddr{#1}{#2}}}
\newcommand{\email}[2][]{\g@addto@macro\addresses{\email{#1}{#2}}}
\newcommand{\urladdr}[2][]{\g@addto@macro\addresses{\urladdr{#1}{#2}}}
\def\enddoc@text{
  \ifx\@empty\addresses \else\@setaddresses\fi}
\AtEndDocument{\enddoc@text}
\def\emailaddrname{e-mail}
\def\@setaddresses{\par
  \nobreak \begingroup
  \interlinepenalty\@M
  \def\address##1##2{\begingroup%
    \par\addvspace\bigskipamount
    \@ifnotempty{##1}{(\ignorespaces##1\unskip) }%
    {\noindent\ignorespaces##2}\par\endgroup}%
  \def\email##1##2{\begingroup
    \@ifnotempty{##2}{\nobreak\noindent\emailaddrname
      \@ifnotempty{##1}{, \ignorespaces##1\unskip}\/:\space
      \ttfamily##2\par}\endgroup}%
  \addresses
  \endgroup
}

\makeatother

\makeatletter
\def\cstar#1{\expandafter\@cstar\csname c@#1\endcsname}
\def\@cstar#1{\ifcase#1\or $\ast$\or $\ast\ast$\or $\ast\ast\ast$\fi}
\AddEnumerateCounter{\cstar}{\@cstar}{$\ast\ast\ast$}
\makeatother

\newlist{conditions}{enumerate}{1}
\newlist{altconditions}{enumerate}{1}
\newlist{starconditions}{enumerate}{1}
\setlist[conditions]{label=\normalfont(\alph*),ref=\normalfont\alph*}
\setlist[altconditions]{label=\normalfont(\alph*$'$),ref=\normalfont\alph*$'$}
\setlist[starconditions]{label=\normalfont(\cstar*),ref=\normalfont\cstar*}
\newcounter{favoritecondition}

\newcommand{\rank}{\func{rank}}
\newcommand{\SC}{\mathcal{S}}
\newcommand{\Z}{\mathbb{Z}}
\newcommand{\R}{\mathbb{R}}
\newcommand{\X}{\mathcal{X}}
\newcommand{\C}{\mathcal{C}}
\newcommand{\RC}{\mathcal{R}}
\newcommand{\F}{\mathbb{F}}
\newcommand{\HB}{\mathbb{H}}
\newcommand{\CB}{\mathbb{C}}
\newcommand{\SB}{\mathbb{S}}
\newcommand{\T}{\mathbb{T}}
\newcommand{\K}{\mathbb{K}}
\newcommand{\Hom}{\func{Hom}}
\newcommand{\FC}{\mathcal{F}}
\newcommand{\TC}{\mathcal{T}}
\newcommand{\Y}{\mathcal{Y}}
\newcommand{\G}{\mathbb{G}}
\newcommand{\VB}{\func{VB}}
\newcommand{\A}{\mathcal{A}}
\newcommand{\PC}{\mathcal{P}}
\newcommand{\PB}{\mathbb{P}}
\newcommand{\codim}{\func{codim}}
\newcommand{\one}{\mathbbm{1}}
\newcommand{\cupproduct}{\mathbin{\smile}}
\newcommand{\Halg}{H_{\mathrm{alg}}}
\newcommand{\Hstr}{H_{\mathrm{str}}}
\newcommand{\VBRstr}{\VB_{\R\mhyphen\mathrm{str}}}
\newcommand{\V}{\mathbb{V}}
\newcommand{\cl}{\func{cl}}
\newcommand{\Spec}{\func{Spec}}
\newcommand{\HCalg}{H_{\CB\mhyphen\mathrm{alg}}}
\newcommand{\HCstr}{H_{\CB\mhyphen\mathrm{str}}}
\newcommand{\HCalgeven}{H_{\CB\mhyphen\mathrm{alg}}^{\mathrm{even}}}
\newcommand{\HCstreven}{H_{\CB\mhyphen\mathrm{str}}^{\mathrm{even}}}
\newcommand{\Heven}{H^{\mathrm{even}}}
\newcommand{\VBCstr}{\VB_{\CB\mhyphen\mathrm{str}}}
\newcommand{\ch}{\func{ch}}
\newcommand{\Q}{\mathbb{Q}}
\newcommand{\KCstr}{K_{\CB\mhyphen\mathrm{str}}}
\newcommand{\KCalg}{K_{\CB\mhyphen\mathrm{alg}}}
\newcommand{\U}{\mathbb{U}}
\newcommand{\E}{\mathbb{E}}

\newtheorem{theorem}{Theorem}[section]
\newtheorem{corollary}[theorem]{Corollary}
\newtheorem{proposition}[theorem]{Proposition}
\newtheorem{lemma}[theorem]{Lemma}
\theoremstyle{definition}
\newtheorem*{acknowledgements}{Acknowledgements}
\newtheorem{definition}[theorem]{Definition}
\newtheorem{example}[theorem]{Example}
\newtheorem{notation}[theorem]{Notation}
\newtheorem{remark}[theorem]{Remark}
\theoremstyle{remark}
\newtheorem*{assertion}{\indent Assertion}

\DeclarePairedDelimiter\abs{\lvert}{\rvert}%
\DeclarePairedDelimiter\norm{\lVert}{\rVert}%

\makeatletter
\let\oldabs\abs
\def\abs{\@ifstar{\oldabs}{\oldabs*}}
\let\oldnorm\norm
\def\norm{\@ifstar{\oldnorm}{\oldnorm*}}
\makeatother

\mathchardef\mhyphen="2D

\title{\bf Stratified-algebraic vector bundles}
\date{}
\author{Wojciech Kucharz\thanks{The first author was partially supported by
NCN (Poland) grant 2011/01/B/ST1/01289.} \and Krzysztof
Kurdyka\thanks{The second author was partially supported by ANR (France)
grant STAAVF.}}

\address{Wojciech Kucharz\\Institute of Mathematics\\Faculty of Mathematics and Computer
Science\\Jagiellonian University\\ul. \L{}ojasiewicza 6\\30-348
Krak\'ow\\Poland}
\email{Wojciech.Kucharz@im.uj.edu.pl}

\address{Krzysztof Kurdyka\\Laboratoire de Math\'ematiques\\UMR 5175 du
CNRS\\Universit\'e de Savoie\\Campus Scientifique\\73 376 Le
Bourget-du-Lac Cedex\newline France}
\email{kurdyka@univ-savoie.fr}

\usepackage[pdftex, pdfauthor={\theauthor}, pdftitle={\thetitle}]{hyperref}

\begin{document}
\maketitle
\thispagestyle{empty}

\begin{abstract}
We investigate stratified-algebraic vector bundles on a real algebraic
variety $X$. A~stratification of $X$ is a finite collection of pairwise
disjoint, Zariski locally closed subvarieties whose union is $X$. A
topological vector bundle $\xi$ on $X$ is called a stratified-algebraic
vector bundle if, roughly speaking, there exists a stratification $\SC$
of $X$ such that the restriction of $\xi$ to each stratum $S$ in $\SC$
is an algebraic vector bundle on $S$. In particular, every algebraic
vector bundle on $X$ is stratified-algebraic. It turns out that
stratified-algebraic vector bundles have many surprising properties,
which distinguish them from algebraic and topological vector bundles.
\end{abstract}

\keywords{Real algebraic variety, stratification, stratified-algebraic
vector bundle, stratified-regular map.}

\subjclass{14P25, 14P99, 14F25, 19A49.}

\tableofcontents

\section{Introduction and main results}\label{sec-1}
In real algebraic geometry, the role of algebraic, semi-algebraic and
Nash vector bundles is firmly established. Vector bundles of a new type,
called stratified-algebraic vector bundles, are introduced and
investigated in this paper. Stratified-algebraic vector bundles form an
intermediate category between algebraic and semi-algebraic vector
bundles. They have many desirable features of algebraic vector bundles,
but are more flexible. Some of their properties and applications are
quite unexpected.

For background material on real algebraic geometry we refer to
\cite{bib9}. The term \emph{real algebraic variety} designates a locally
ringed space isomorphic to an algebraic subset of $\R^N$, for some $N$,
endowed with the Zariski topology and the sheaf of real-valued regular
functions (such an object is called an affine real algebraic variety in
\cite{bib9}). The class of real algebraic varieties is identical with
the class of quasi-projective real varieties, cf.
\cite[Proposition~3.2.10, Theorem~3.4.4]{bib9}. Morphisms of real
algebraic varieties are called \emph{regular maps}. Each real algebraic
variety carries also the Euclidean topology, which is induced by the
usual metric on $\R$. Unless explicitly stated otherwise, all
topological notions relating to real algebraic varieties refer to the
Euclidean topology.

Let $X$ be a real algebraic variety. By a \emph{stratification} of $X$
we mean a finite collection $\X$ of pairwise disjoint, Zariski locally
closed subvarieties whose union is $X$. Each subvariety in $\X$ is
called a \emph{stratum} of $\X$; a stratum can be empty. The
stratification $\X$ is said to be \emph{nonsingular} if each stratum in
it is a nonsingular subvariety. A stratification $\X'$ of $X$ is said to
be a \emph{refinement} of $\X$ if each stratum of $\X'$ is contained in
some stratum of $\X$. There is a nonsingular stratification of $X$ which
is a refinement of $\X$. If $\X_1$ and $\X_2$ are stratifications of
$X$, then the collection $\{ S_1 \cap S_2 \mid S_1 \in \X_1, S_2 \in
\X_2 \}$ is a stratification of $X$ that is a refinement of $\X_i$ for
$i=1,2$. These facts will be frequently tacitly used. The stratification
$\{X\}$ of $X$, consisting of only one stratum, is said to be
\emph{trivial}.

Let $\F$ stand for $\R$, $\CB$ or $\HB$ (the quaternions). All $\F$-vector
spaces will be left $\F$-vector spaces. When convenient, $\F$ will be
identified with $\R^{d(\F)}$, where ${d(\F) = \dim_{\R}\F}$.

For any topological $\F$-vector bundle $\xi$ on $X$, denote by $E(\xi)$
its total space and by $\pi(\xi) \colon E(\xi) \to X$ the bundle
projection. The fiber of $\xi$ over a point $x$ in $X$ is ${E(\xi)_x
\coloneqq \pi(\xi)^{-1}(x)}$. Denote by $\varepsilon_X^n(\F)$ the
standard trivial $\F$-vector bundle on $X$ with total space $X \times
\F^n$, where $X \times \F^n$ is regarded as a real algebraic variety.
By an \emph{algebraic $\F$-vector bundle} on $X$ we mean an algebraic
$\F$-vector subbundle of $\varepsilon_X^n(\F)$ for some $n$ (cf.
\cite[Chapters~12 and 13]{bib9} for various characterizations of
algebraic $\F$-vector bundles). In particular, if $\xi$ is an algebraic
$\F$-vector subbundle of $\varepsilon_X^n(\F)$, then $E(\xi)$ is a
Zariski closed subvariety of $X \times \F^n$, $\pi(\xi)$ is the
restriction of the canonical projection $X \times \F^n \to X$, and
$E(\xi)_x$ is an $\F$-vector subspace of $\{x\}\times\F^n$ for each
point $x$ in $X$.

Given a stratification $\X$ of $X$, we now introduce the crucial notion
for this paper.

\begin{definition}\label{def-1-1}
An \emph{$\X$-algebraic $\F$-vector bundle} on $X$ is a topological
$\F$-vector subbundle $\xi$ of $\varepsilon_X^n(\F)$, for some $n$, such
that the restriction $\xi|_S$ of $\xi$ to each stratum $S$ of $\X$ is an
algebraic $\F$-vector subbundle of $\varepsilon_S^n(\F)$. If $\xi$ and
$\eta$ are $\X$-algebraic $\F$-vector bundles on $X$, an \emph{$\X$-algebraic
morphism} $\varphi \colon \xi \to \eta$ is a morphism of topological
$\F$-vector bundles which induces a morphism of algebraic $\F$-vector
bundles $\varphi_S \colon \xi|_S \to \eta|_S$ for each stratum $S$ in
$\X$.
\end{definition}

The conditions imposed on $\varphi$ mean that $\varphi \colon E(\xi) \to
E(\eta)$ is a continuous map, ${\pi(\xi) = \pi(\eta) \circ \varphi}$, the
restriction $\varphi_x \colon E(\xi)_x \to E(\eta)_x$ of $\varphi$ is an
$\F$-linear transformation for each point $x$ in $X$, and the
restriction $\varphi_S \colon \pi(\xi)^{-1}(S) \to \pi(\eta)^{-1}(S)$
of $\varphi$ is a regular map of real algebraic varieties for each
stratum $S$ in $\X$.

If $\xi$ is as in Definition~\ref{def-1-1}, we also say that $\xi$ is an
\emph{$\X$-algebraic $\F$-vector subbundle of $\varepsilon_X^n(\F)$}.

One readily checks that $\X$-algebraic $\F$-vector bundles on $X$
(together with $\X$-algebraic morphisms) form a category. An
$\X$-algebraic morphism of $\X$-algebraic $\F$-vector bundles is an
isomorphism if and only if it is a bijective map. In particular, the
category of algebraic $\F$-vector bundles on $X$ coincides with the
category of $\{X\}$-algebraic $\F$-vector bundles on $X$.

\begin{definition}\label{def-1-2}
A \emph{stratified-algebraic $\F$-vector bundle} on $X$ is an
$\SC$-algebraic $\F$-vector bundle for some stratification $\SC$ of $X$.
If $\xi$ and $\eta$ are stratified-algebraic $\F$-vector bundles on $X$,
a \emph{stratified-algebraic morphism} $\varphi \colon \xi \to \eta$ is
an $\SC$-algebraic morphism for some stratification $\SC$ of $X$ such
that both $\xi$ and $\eta$ are $\SC$-algebraic $\F$-vector bundles.
\end{definition}

A stratified-algebraic $\F$-vector subbundle of $\xi_X^n(\F)$ is defined
in an obvious way.

Stratified-algebraic $\F$-vector bundles on $X$ (together with
stratified-algebraic morphisms) form a category. A stratified-algebraic
morphism of stratified-algebraic $\F$-vector bundles is an isomorphism
if and only if it is a bijective map.

Theory of $\X$-algebraic and stratified-algebraic $\F$-vector bundles is
developed in the subsequent sections. In the present section, we only
announce seven rather surprising results.

Denote by $\SB^d$ the unit $d$-sphere,
\begin{equation*}
\SB^d = \{ (x_0, \ldots, x_d) \in \R^{d+1} \mid x_0^2 + \cdots + x_d^2 =
1 \}.
\end{equation*}
\begin{theorem}\label{th-1-3}
Let $X$ be a compact real algebraic variety homotopically equivalent to
$\SB^d$. Then each topological $\F$-vector bundle on $X$ is isomorphic
to a stratified-algebraic $\F$-vector bundle.
\end{theorem}

Theorem~\ref{th-1-3} makes it possible to demonstrate an essential
difference between algebraic and stratified-algebraic $\F$-vector
bundles.

\begin{example}\label{ex-1-4}
It is well known that each topological $\F$-vector bundle on $\SB^d$ is
isomorphic to an algebraic $\F$-vector bundle, cf.
\cite[Theorem~11.1]{bib44} and \cite[Proposition~12.1.12; pp.~325, 326,
352]{bib9}. This is no longer true if $\SB^d$ is replaced by a
nonsingular real algebraic variety diffeomorphic to $\SB^d$. Indeed, for
every positive integer $k$, there exists a nonsingular real algebraic
variety $\Sigma^{4k}$ that is diffeomorphic to $\SB^{4k}$ and each
algebraic $\F$-vector bundle on $\Sigma^{4k}$ is topologically stably
trivial, cf. \cite[Theorem~9.1]{bib8}. However, on $\SB^{4k}$, and hence
on $\Sigma^{4k}$, there are topological $\F$-vector bundles that are not
stably trivial, cf. \cite{bib28}. On the other hand, according to
Theorem~\ref{th-1-3}, each topological $\F$-vector bundle on
$\Sigma^{4k}$ is isomorphic to a stratified-algebraic $\F$-vector
bundle.
\end{example}

Stratified-algebraic vector bundles on a compact real algebraic variety
are in some sense stable with respect to stratifications. This is made
precise in the following result.

\begin{theorem}\label{th-1-5}
Any compact real algebraic variety $X$ admits a nonsingular stratification $\SC$ such
that each stratified-algebraic $\F$-vector bundle on $X$ is isomorphic
(in the category of stratified-algebraic $\F$-vector bundles on $X$) to
an $\SC$-algebraic $\F$-vector bundle.
\end{theorem}

It is remarkable that the stratification $\SC$ in Theorem~\ref{th-1-5}
is suitable for all stratified-algebraic $\F$-vector bundles on $X$. In
general one cannot take as $\SC$ the trivial stratification $\{ X \}$
of $X$, even if $X$ is a ``simple'' nonsingular real algebraic variety,
cf. Example~\ref{ex-1-4}.

A \emph{multiblowup} of a real algebraic variety $X$ is a regular map
$\pi \colon X' \to X$ which is the composition of a finite number of
blowups with nonsingular centers. If no additional restrictions on the
centers of blowups are imposed, $\pi$ need not be a birational map (for
example, a blowup of $X$ with center of dimension $\dim X$ is not a
birational map).

\begin{theorem}\label{th-1-6}
For any compact real algebraic variety $X$, there exists a birational
multiblowup $\pi \colon X' \to X$, with $X'$ a nonsingular variety, such
that for each stratified-algebraic $\F$-vector bundle $\xi$ on $X$, the
pullback $\F$-vector bundle $\pi^*\xi$ on $X'$ is isomorphic (in the
category of stratified-algebraic $\F$-vector bundles on $X'$) to an
algebraic $\F$-vector bundle on $X'$.
\end{theorem}

The multiblowup $\pi \colon X' \to X$ in Theorem~\ref{th-1-6} is chosen
in a universal way, that is, it does not depend on $\xi$.

Any topological $\F$-vector bundle $\xi$ can be regarded as an
$\R$-vector bundle, which is indicated by $\xi_{\R}$. If $\xi$ is
$\X$-algebraic, then so is $\xi_{\R}$.

\begin{theorem}\label{th-1-7}
Let $X$ be a compact real algebraic variety. A topological $\F$-vector
bundle $\xi$ on $X$ is isomorphic to a stratified-algebraic $\F$-vector
bundle if and only if the topological $\R$-vector bundle $\xi_{\R}$ is
isomorphic to a stratified-algebraic $\R$-vector bundle.
\end{theorem}

This result is unexpected since it may happen that a topological
$\F$-vector bundle $\xi$ (with $\F=\CB$ or $\F=\HB$) is not isomorphic
to any algebraic $\F$-vector bundle, whereas the $\R$-vector bundle
$\xi_{\R}$ is isomorphic to an algebraic $\R$-vector bundle, cf. the
$\CB$-line bundle $\lambda^{\CB}$ in Example~\ref{ex-1-11}.

Theorems~\ref{th-1-5}, \ref{th-1-6} and \ref{th-1-7} are equivalent to
certain approximation results, cf. Theorems~\ref{th-4-8}, \ref{th-4-9}
and \ref{th-6-6}.

It is possible to give, in some cases, a simple geometric criterion for
a topological vector bundle to be isomorphic to a stratified-algebraic
vector bundle.

Let $X$ be a compact nonsingular real algebraic variety. A smooth (of
class~$\C^{\infty}$) \mbox{$\F$-vector} bundle $\xi$ on $X$ is said to be
\emph{adapted} if there exists a smooth section ${u \colon X \to \xi}$
transverse to the zero section and such that its zero locus $Z(u)
\coloneqq \{ {x \in X} \mid {u(x)=0} \}$, which is a compact smooth
submanifold of $X$, is smoothly isotopic to a nonsingular Zariski
locally closed subvariety $Z$ of $X$. In that case, $Z$ is a closed
subset of $X$ in the Euclidean topology, but need not be Zariski closed.
If $\rank \xi = \dim X$, then $\xi$ is adapted since the zero locus of
any smooth section of $\xi$ that is transverse to the zero section is a
finite set.

\begin{theorem}\label{th-1-8}
Let $X$ be a compact nonsingular real algebraic variety. If a smooth
$\F$-line bundle $\xi$ on $X$ is adapted, then it is topologically
isomorphic to a stratified-algebraic $\F$-line bundle.
\end{theorem}

Theorem~\ref{th-1-8} with $\F = \R$ is not interesting since then a
stronger result, asserting that $\xi$ is isomorphic to an algebraic
$\R$-line bundle, is known, cf. \cite[Theorem~12.4.6]{bib9}. However,
$\xi$ need not be isomorphic to any algebraic $\F$-line bundle for
$\F=\CB$ or $\F=\HB$, cf. the $\F$-line bundle $\lambda^{\F}$ in
Example~\ref{ex-1-11}.

Assuming that $\F=\R$ or $\F=\CB$, the $r$th exterior power of any
$\F$-vector bundle of rank $r$ is denoted by $\det \xi$. Thus, $\det
\xi$ is an $\F$-line bundle.

\begin{theorem}\label{th-1-9}
Let $X$ be a compact nonsingular real algebraic variety and let $\xi$ be
a smooth $\F$-vector bundle of rank $2$ on $X$, where $\F = \R$ or
$\F=\CB$. If both $\F$-vector bundles $\xi$ and $\det \xi$ are adapted,
then $\xi$ is topologically isomorphic to a stratified-algebraic
$\F$-vector bundle.
\end{theorem}

It is not a serious restriction that the vector bundle $\xi$ in the last
two theorems is smooth. In fact, any topological $\F$-vector bundle on a
smooth manifold is topologically isomorphic to a smooth $\F$-vector
bundle, cf. \cite{bib27}. Actually, suitably refined versions of
Theorems~\ref{th-1-8} and~\ref{th-1-9} hold true even if the variety $X$
is possibly singular, cf. Theorems~\ref{th-6-9} and~\ref{th-6-10}.

We select one more result for this section.

\begin{theorem}\label{th-1-10}
Let $X = X_1 \times \cdots \times X_n$, where each $X_i$ is a
compact real algebraic variety homotopically equivalent to the unit
$d_i$-sphere for $1 \leq i \leq n$. If $\F=\CB$ or $\F=\HB$, then each
topological $\F$-vector bundle on $X$ is isomorphic to a
stratified-algebraic $\F$-vector bundle. If $\xi$ is a topological
$\R$-vector bundle on $X$, then the direct sum $\xi \oplus \xi$ is
isomorphic to a stratified-algebraic $\R$-vector bundle.
\end{theorem}

We do not know if in Theorem~\ref{th-1-10} each topological $\R$-vector
bundle on $X$ is isomorphic to a stratified-algebraic $\R$-vector
bundle. This remains an open problem even for $\R$-vector bundles on the
standard $n$-torus
\begin{equation*}
\T^n = \SB^1 \times \cdots \times \SB^1 \quad \textrm{(the $n$-fold
product)}.
\end{equation*}

It is worthwhile to contrast the results above with the behavior of
algebraic $\F$-vector bundles on $\T^n$. As usual, the $k$th Chern class
of a $\CB$-vector bundle $\xi$ will be denoted by $c_k(\xi)$. Any
$\F$-vector bundle $\zeta$ can be regarded as a $\K$-vector bundle,
denoted $\zeta_{\K}$, where $\K \subseteq \F$ and $\K$ stands for $\R$,
$\CB$ or $\HB$. In particular, $\zeta_{\F}=\zeta$. If $\zeta$ is an
$\HB$-vector bundle, then $\zeta_{\R} = (\zeta_{\CB})_{\R}$.

\begin{example}\label{ex-1-11}
Every algebraic $\CB$-vector bundle on $\T^n$ is algebraically stably
trivial, cf. \cite{bib11} or \cite[Corollary~12.6.6]{bib9}. In
particular, every algebraic $\CB$-line bundle on $\T^n$ is algebraically
trivial. Consequently, for each algebraic $\F$-vector bundle $\zeta$ on
$\T^n$, where $\F=\CB$ or $\F=\HB$, one has $c_k(\zeta_{\CB}) = 0$ for
every $k \geq 1$. Furthermore, for each algebraic $\R$-vector bundle
$\eta$ on $\T^n$, the direct sum $\eta \oplus \eta$ is algebraically
stably trivial (it suffices to consider the complexification $\CB
\otimes \eta$ of $\eta$ and make use of the equality $(\CB \otimes
\eta)_{\R} = \eta \oplus \eta)$. On the other hand, if $1 \leq n\leq 3$, then
each topological $\R$-vector bundle on $\T^n$ is isomorphic to an
algebraic $\R$-vector bundle \cite{bib13}.

Assume now that $\F=\CB$ or $\F=\HB$. We choose a smooth $\F$-line
bundle $\theta^{\F}$ on $\T^{d(\F)}$ such that $c_1(\theta^{\CB}) \neq
0$ and $c_2((\theta^{\HB})_{\CB}) \neq 0$. If $n \geq d(\F)$ and $p_{\F}
\colon \T^n = \T^{d(\F)} \times \T^{n - d(\F)} \to \T^{d(\F)}$ is the
canonical projection, then the smooth $\F$-line bundle
\begin{equation*}
\lambda^{\F} \coloneqq p_{\F}^* \theta^{\F}
\end{equation*}
on $\T^n$ is not topologically isomorphic to any algebraic $\F$-line
bundle. This assertion holds since $c_1(\lambda^{\CB}) \neq 0$ and
$c_2((\lambda^{\HB})_{\CB}) \neq 0$. Obviously, $\theta^{\F}$ is
adapted, and hence $\lambda^{\F}$ is adapted too. Furthermore, the
$\R$-vector bundle $(\lambda^{\CB})_{\R}$ on $\T^n$ is isomorphic to an
algebraic $\R$-vector bundle since the $\R$-vector bundle
$(\theta^{\CB})_{\R}$ on $\T^2$ is isomorphic to an algebraic
$\R$-vector bundle, $p_{\CB}$ is a regular map, and
\begin{equation*}
(\lambda^{\CB})_{\R} = p_{\CB}^* ( (\theta^{\CB})_{\R} ).
\end{equation*}
Finally, if $\xi = (\lambda^{\HB})_{\R}$, then the $\R$-vector bundle
$\xi \oplus \xi$ on $\T^n$ is not isomorphic to any algebraic
$\R$-vector bundle. Indeed, supposing otherwise, the complexification
$\CB \otimes (\xi \oplus \xi)$ of $\xi \oplus \xi$ would be isomorphic
to an algebraic $\CB$-vector bundle, and hence $c_2(\CB \otimes (\xi
\oplus \xi)) =0$. However, one has $c_1( (\lambda^{\HB} \oplus
\lambda^{\HB})_{\CB} ) = 0$, which implies
\begin{equation*}
c_2( \CB \otimes (\xi \oplus \xi) ) = c_2( \CB \otimes (\lambda^{\HB}
\oplus \lambda^{\HB})_{\R} ) = 2 c_2( (\lambda^{\HB} \oplus
\lambda^{\HB})_{\CB} ) = 4 c_2 ( (\lambda^{\HB})_{\CB} ) \neq 0,
\end{equation*}
cf. \cite[Corollary~15.5]{bib38}.
\end{example}

There are other significant differences between algebraic and
stratified-algebraic vector bundles. The interested reader can identify
them by consulting papers devoted to algebraic vector bundles on real
algebraic varieties \cite{bib5, bib6, bib7, bib8, bib9, bib10, bib13,
bib14}. It should be mentioned that there are topological $\F$-line
bundles which are not isomorphic to stratified-algebraic $\F$-line
bundles. For instance, such $\F$-line bundles exist on some nonsingular
real algebraic varieties diffeomorphic to $\T^n$, cf.
Example~\ref{ex-7-10}.

The paper is organized as follows. In Section~\ref{sec-2}, we define in
a natural way stratified-regular maps. Some homotopical properties of
such maps imply Theorem~\ref{th-1-3}. In Section~\ref{sec-3}, we study
relationships between vector bundles of various types and appropriate
finitely generated projective modules. As a model serve classical
results due to Serre \cite{bib41} and Swan \cite{bib43}. Subsequently,
we prove Theorems~\ref{th-1-5} and~\ref{th-1-6} by applying \cite{bib44}
and some results from K-theory. Section~\ref{sec-4} is devoted to maps
with values in a Grassmannian (actually, a multi-Grassmannian). The main
topic is the approximation of continuous maps by stratified-regular maps.
Approximation results of this kind are essential for the rest of the
paper. As the starting point for these results the extension theorem of
Koll\'ar and Nowak \cite{bib31} is indispensable. In
Section~\ref{sec-5}, we show how certain properties of a vector bundle
can be deduced from the behavior of its restrictions to Zariski closed
subvarieties of the base space. This is useful in proofs by induction in
Section~\ref{sec-6}. In particular, Theorems~\ref{th-1-7}, \ref{th-1-8}
and~\ref{th-1-9} are proved in Section~\ref{sec-6}. Algebraic and
$\CB$-algebraic cohomology classes have many applications in real
algebraic geometry, cf. \cite{bib1, bib2, bib5, bib6, bib8, bib9, bib10,
bib12, bib13, bib14, bib16, bib17, bib18, bib19, bib32, bib34}. We
introduce stratified-algebraic and stratified-$\CB$-algebraic cohomology
classes in Section~\ref{sec-7} and Section~\ref{sec-8}, respectively.
They prove to be very useful in our investigation of
stratified-algebraic vector bundles. In particular,
stratified-$\CB$-algebraic cohomology classes play a key role in the
proof of Theorem~\ref{th-1-10} given in Section~\ref{sec-8}.

\begin{acknowledgements}
We greatly benefited from studying \cite{bib23} and \cite{bib31}, as
well as earlier versions of these papers. Our discussions with
G.~Fichou, J.~Huisman, F.~Mangolte and J.-Ph.~Monnier about their paper
\cite{bib23} provided additional inspiration.
\end{acknowledgements}

\section{Stratified-regular maps}\label{sec-2}

Throughout this section, $X$ and $Y$ denote real algebraic varieties,
and $\X$ denotes a stratification of $X$.

To begin with, we introduce maps that will be crucial for the
investigation of $\X$-algebraic and stratified-algebraic vector bundles.

\begin{definition}\label{def-2-1}
A map $f \colon X \to Y$ is said to be \emph{$\X$-regular} if it is
continuous and its restriction to each stratum in $\X$ is a regular map.
Furthermore, $f$ is said to be \emph{stratified-regular} if it is
$\SC$-regular for some stratification $\SC$ of $X$.
\end{definition}

In particular, $f$ is $\{X\}$-regular if and only if it is a regular
map. Following \cite{bib31,bib32,bib34}, we say that $f \colon X \to Y$
is a \emph{continuous rational} map if $f$ is continuous and its
restriction to some Zariski open and dense subvariety of $X$ is a
regular map.

By a \emph{filtration} of $X$ we mean a finite sequence $\FC = (X_{-1},
X_0, \ldots, X_m)$ of Zariski closed subvarieties satisfying
\begin{equation*}
\varnothing = X_{-1} \subseteq X_0 \subseteq \ldots \subseteq X_m = X.
\end{equation*}
The collection $\overline{\FC} \coloneqq \{ X_i \setminus X_{i-1} \mid 0 \leq
i \leq m \}$ is a stratification of $X$.

\begin{proposition}\label{prop-2-2}
For a map $f \colon X \to Y$, the following conditions are equivalent:
\begin{conditions}%
\item\label{prop-2-2-a} The map $f$ is stratified-regular.

\item\label{prop-2-2-b} There exists a filtration $\FC$ of $X$ such that
$f$ is $\overline{\FC}$-regular and each stratum in $\overline{\FC}$ is
nonsingular and equidimensional.%
\setcounter{favoritecondition}{\value{conditionsi}}%
\suspend{conditions}%
\begin{altconditions}[start=\value{favoritecondition}]
\item\label{prop-2-2-b1}%
There exists a filtration $\FC'$ of $X$ such that $f$ is $\overline{\FC}'$-regular.
\end{altconditions}
\resume{conditions}%
\item\label{prop-2-2-c} For each Zariski closed subvariety $Z$ of $X$,
the restriction $f|_Z \colon Z \to Y$ is a continuous rational map.
\setcounter{favoritecondition}{\value{conditionsi}}
\end{conditions}
\end{proposition}

\begin{proof}
Suppose that condition (\ref{prop-2-2-c}) holds. Since the map $f$ is
continuous rational, there exists a Zariski open and dense subvariety
$X^0$ of $X$ such that the restriction of $f$ to $X^0$ is a regular map.
We can choose such an $X^0$ disjoint from the singular locus of $X$. Let
$Z$ be the union of $X \setminus X^0$ and all the irreducible components
of $X$ of dimension strictly less than $\dim X$. Then $Z$ is a Zariski
closed subvariety of $X$ with $\dim Z < \dim X$, and $X \setminus Z$ is
nonsingular of pure dimension. Since the map $f|_Z$ is continuous
rational, the construction above can be repeated with $X$ replaced by
$Z$. By continuing this process, we conclude that (\ref{prop-2-2-b})
holds.

If $V$ is an irreducible Zariski closed subvariety of $X$, then there
exists a stratum $S$ in $\X$ such that the intersection $S \cap V$ is
nonempty and Zariski open in $V$ (hence Zariski dense in $V$). It
follows that (\ref{prop-2-2-a}) implies (\ref{prop-2-2-c}).

It is obvious that (\ref{prop-2-2-b}) implies (\ref{prop-2-2-b1}), and
(\ref{prop-2-2-b1}) implies (\ref{prop-2-2-a}).
\end{proof}

\begin{remark}\label{rem-2-3}
Proposition~\ref{prop-2-2} (with $Y=\R$) appears in the paper of
Koll\'ar and Nowak \cite{bib31} as a comment following Definition~8. In
\cite{bib31}, the attention is focused on functions satisfying
condition~(\ref{prop-2-2-c}), called continuous hereditarily rational
functions. By \cite[Proposition~7]{bib31}, if the variety $X$ is
nonsingular, then condition (\ref{prop-2-2-c}) is equivalent to
\begin{altconditions}[start=\value{favoritecondition}]
\item\label{rem-2-3-c'}\textit{The map $f$ is continuous rational.}
\end{altconditions}
If $X$ is singular, then (\ref{rem-2-3-c'}) need not imply (\ref{prop-2-2-c}),
cf. \cite[Examples~2 and~3]{bib31}. Continuous rational maps defined on
nonsingular real algebraic varieties are investigated in \cite{bib32,
bib34}. According to \cite[Theorem~9]{bib31} and
\cite[Corolaire~4.40]{bib23}, stratified-regular maps coincide with
``\emph{applications r\'egulues}'' between real algebraic varieties,
which are introduced in a more general framework in
\cite{bib23}. Notions introduced and methods developed in \cite{bib35,
bib36, bib37, bib39} are important in the study of the geometry of
``\emph{fonctions r\'egulues}'', cf. \cite{bib23}.
\end{remark}

The proof of Theorem~\ref{th-1-3} requires some knowledge of homotopy
classes represented by stratified-regular maps with values in the unit
$d$-sphere $\SB^d$, cf. Theorem~\ref{th-2-5}. We first need the
following technical result.

\begin{lemma}\label{lem-2-4}
Assume that the variety $X$ is compact of dimension $d$. Let $A$ be the
union of the singular locus of $X$ and all the irreducible components of
$X$ of dimension at most $d -1$. Let $\varphi \colon X \to \R^d$ be a
continuous map which is constant in a neighborhood of $A$ and smooth on
$X \setminus A$. Assume that $0$ in $\R^d$ is a regular value of
$\varphi|_{X \setminus A}$, and the inverse image $V \coloneqq
\varphi^{-1}(0)$ is disjoint from $A$. Then there exists a regular map
$\psi \colon X \to \R^d$ such that $\psi$ is equal to $\varphi$ on $A$,
$\psi^{-1}(0)=V$, and the differentials $d\psi_x$ and $d\varphi_x$ are
equal for every point $x$ in $V$.
\end{lemma}

\begin{proof}
We may assume that $X$ is a Zariski closed subvariety of $\R^n$ for some
$n$. Then we can find a smooth map $f \colon \R^n \to \R^d$ which is an
extension of $\varphi$ and is constant on a neighborhood $U$ of $A$ in
$\R^n$. Thus $U \subseteq f^{-1}(c)$ for some point $c=(c_1, \ldots,
c_d)$ in $\R^d$. Since $V$ is a finite set, there exists a polynomial
map $g \colon \R^n \to \R^d$ such that $g(x)=f(x)$ and $dg_x = df_x$ for
every point $x$ in $V$, and $A \subseteq g^{-1}(c)$. Indeed, $g$ can be
constructed as follows. First we choose a polynomial map $\eta \colon
\R^n \to \R^d$ with $\eta(x)=f(x)$ and $d\eta_x=df_x$ for every $x$ in
$V$. If $\alpha \colon \R^n \to \R$ is a polynomial function
satisfying $V \subseteq \alpha^{-1}(0)$ and $A \subseteq
\alpha^{-1}(1)$, then the map $g \coloneqq \eta - \alpha^2(\eta-c)$ satisfies
all the requirements.

For any subset $Z$ of $\R^n$, we denote by $I(Z)$ the ideal of all
polynomial functions on $\R^n$ vanishing on $Z$. Let $p_1, \ldots, p_r$
be generators of the ideal $I(V)^2I(A)$.

\begin{assertion}
There exist smooth maps $\lambda_1 \colon \R^n \to \R^d, \ldots,
\lambda_r \colon \R^n \to \R^d$ such that $f=g+p_1 \lambda_1 + \cdots +
p_r\lambda_r$.
\end{assertion}

The Assertion can be proved as follows. For any point $x$ in $\R^n$,
denote by $\C^{\infty}_x$ the ring of germs at $x$ of smooth functions
on $\R^n$. Let $f = (f_1, \ldots, f_d)$, $g = (g_1, \ldots, g_d)$ and
$h_i = f_i - g_i$ for $1 \leq i \leq d$. For every point $x$ in $\R^n
\setminus A$, the germ $(h_i)_x$ of $h_i$ at $x$ belongs to the ideal
$I(V)^2I(A)\C^{\infty}_x$. Obviously, $g_i - c_i$ belongs to $I(A)$. If
$x$ is a point in $A$, then $f_i - c_i$ vanishes in a
neighborhood of $x$, and hence $(h_i)_x = ( (f_i)_x - c_i) - ( (g_i)_x -
c_i)$ belongs to $I(V)^2I(A)\C^{\infty}_x$. Making use of partition of
unity, we obtain smooth functions $\lambda_{i1}, \ldots, \lambda_{ir}$
on $\R^n$ for which $h_i = p_1\lambda_{i1} + \cdots +
p_r\lambda_{ir}$. Thus, the Assertion holds with $\lambda_j =
(\lambda_{1j}, \ldots \lambda_{dj})$ for $1 \leq j \leq r$.

Now, let $q_1 \colon \R^n \to \R^d, \ldots, q_r \colon \R^n \to \R^d$ be
polynomial maps and let $\psi \colon X \to \R^d$ be the restriction of
the map $g + p_1q_1 + \cdots p_rq_r$. Then $\psi = \varphi$ on $A$, $V
\subseteq \psi^{-1}(0)$, and $d\psi_x = d\varphi_x$ for every point $x$
in $V$. According to the Stone--Weierstrass theorem, given $\varepsilon
> 0$, we can find polynomial maps $q_j$ such that
\begin{equation*}
\norm{\lambda_j(x) - q_j(x)} + \sum_{k=1}^{n} \norm{ \frac{\partial
\lambda_j}{\partial x_k}(x) - \frac{\partial q_j}{\partial x_k}(x)} <
\varepsilon
\end{equation*}
for all $x$ in $X$ and $1 \leq j \leq r$. If $\varepsilon$ is
sufficiently small, then $\psi^{-1}(0) = V$, and hence $\psi$ satisfies
all the requirements.
\end{proof}

\begin{theorem}\label{th-2-5}
If the variety $X$ is compact of dimension $d$, then each continuous map
from $X$ into $\SB^d$ is homotopic to a stratified-regular map.
\end{theorem}

\begin{proof}
Let $A$ be the union of the singular locus of $X$ and all the
irreducible components of $X$ of dimension at most $d-1$. Since $(X,A)$
is a polyhedral pair \cite[Corollary~9.3.7]{bib9}, the restriction of
$f$ to some neighborhood of $A$ is null homotopic. Hence, according to
the homotopy extension theorem \cite[p.~90, Theorem~1.4]{bib27},
the map $f$ can be deformed without affecting its homotopy class so that
$f$ is constant in a compact neighborhood $L$ of $A$. Furthermore, we
may assume that $f$ is smooth on $X \setminus A$. By Sard's theorem,
there exists a regular value $y$ in $\SB^d$ of the smooth map $f|_{X
\setminus A}$ such that both points $y$ and $-y$ are in $\SB^d \setminus
f(L)$. In particular, the set $f^{-1}(y)$ is finite. We choose a compact
neighborhood $K$ of $f^{-1}(y)$ in $X$ which is disjoint from $L \cup
f^{-1}(-y)$ and such that each point in $K$ is a regular point of $f$.
Let $p \colon \SB^d \setminus \{-y\} \to \R^d$ be the stereographic
projection (in particular, $p(y)=0$). Let $\kappa \colon X \to \R$ be a
continuous function, smooth on $X \setminus A$, with $\kappa = 1$ on $K
\cup L$ and $\kappa = 0$ in a neighborhood of $f^{-1}(-y)$. The map $\xi
= (\xi_1, \ldots, \xi_d) \colon X \to \R^d$, defined by
\begin{equation*}
\xi(x) =
\begin{cases}
\kappa(x)p(f(x)) & \textrm{for } x \textrm{ in } X \setminus
f^{-1}(-y)\\
0 & \textrm{for } x \textrm{ in } f^{-1}(-y),
\end{cases}
\end{equation*}
is continuous, smooth on $X \setminus A$, $\xi^{-1}(0)\cap L =
\varnothing$, $f^{-1}(y) \subseteq \xi^{-1}(0)$, and each point in $K$
is a regular point of $\xi$. Let $\lambda \colon X \to \R$ be a
continuous function, smooth on $X \setminus A$, with $\lambda^{-1}(0) =
K \cup L$ (such a function exists, cf. for instance
\cite[Theorem~14.1]{bib21}). The map $\eta \colon X \times \R^d \to
\R^d$, defined by
\begin{equation*}
\eta( x, (s_1,\ldots s_d) ) = ( \xi_1(x) + s_1\lambda(x)^2, \ldots,
\xi_d(x) + s_d\lambda(x)^2 ),
\end{equation*}
is continuous, smooth on $(X \setminus A) \times \R^d$, and $0$ in
$\R^d$ is a regular value of the restriction of $\eta$ to $(X \setminus
A) \times \R^d$. According to the parametric transversality theorem, we
can choose a point $(s_1, \ldots, s_d)$ in $\R^d$ such that if $\varphi
\colon X \to \R^d$ is defined by $\varphi(x) = \eta( x, (s_1, \ldots,
s_d) )$ for all $x$ in $X$, then $0$ in $\R^d$ is a regular value of the
restriction of $\varphi$ to $X \setminus A$. By construction, $\varphi$
is constant in a neighborhood of $A$, and the set $V \coloneqq
\varphi^{-1}(0)$ is disjoint from $A$. Furthermore, $V$ is a finite set
containing $f^{-1}(y)$. Let $\psi \colon X \to \R^d$ be a regular map as
in Lemma~\ref{lem-2-4}. We choose a regular function $\alpha \colon X
\to \R$ with $\alpha^{-1}(0) = W$ and $\alpha^{-1}(1) = f^{-1}(y)\cup A$,
where $W \coloneqq V \setminus f^{-1}(y)$. For example, we can take
$\alpha = \alpha_1^2 / (\alpha_1^2 + \alpha_2^2)$, where $\alpha_1$ and
$\alpha_2$ are regular functions on $X$ satisfying $\alpha_1^{-1}(0) =
W$ and $\alpha_2^{-1}(0) = f^{-1}(y) \cup A$. By the \L{}ojasiewicz
inequality \cite[Corollary~2.6.7]{bib9}, there exist a neighborhood $U$
of $W$ in $X$, a positive real number $c$, and a positive integer $k$
for which
\begin{equation*}
\norm{\psi(x)} \geq c\alpha(x)^{2k} \quad \textrm{for all $x$ in $U$}.
\end{equation*}
Let $\bar{\psi}(x) = (1 / \alpha(x)^{2k+1}) \psi(x)$ for $x$ in $X
\setminus W$. The map $g \colon X \to \SB^d$, defined by
\begin{equation*}
g(x)=
\begin{cases}
p^{-1}(\bar{\psi}(x)) & \textrm{for $x$ in $X \setminus W$}\\
-y & \textrm{for $x$ in $W$},
\end{cases}
\end{equation*}
is continuous. Since $p$ is a biregular isomorphism, the restriction
$g|_{X \setminus W}$ is a regular map, and hence $g$ is a
stratified-regular map. It suffices to prove that $f$ is homotopic to
$g$.

The map $\bar{\psi} \colon X \setminus W \to \R^d$ has the following
properties: $\psi^{-1}(0) = f^{-1}(y) = V \setminus W$ and
$d\bar{\psi}_x = d\varphi_x$ for every point $x$ in $f^{-1}(y)$. If $U_1$
is a sufficiently small open neighborhood of $f^{-1}(y)$ in $X$, then
\begin{equation*}
(1-t) \varphi(x) + t\bar{\psi}(x) \neq 0 \quad \textrm{for all $(x,t)$
in $(U_1 \setminus f^{-1}(y)) \times [0,1]$}.
\end{equation*}
We may assume that $U_1 \subseteq K$, and hence $\varphi(x) = p(f(x))$
for all $x$ in $U_1$. Let $U_2$ be an open neighborhood of $f^{-1}(y)$
whose closure $\overline{U_2}$ is contained in $U_1$. Choose a
continuous function $\mu \colon X \to [0,1]$, smooth on $X \setminus A$,
with $\mu = 1$ on $\overline{U_2}$ and $\mu = 0$ in a neighborhood of $X
\setminus U_1$. Then the map $F \colon X \times [0,1] \to \SB^p$,
defined by
\begin{equation*}
F(x,t)=
\begin{cases}
p^{-1} ( (1 - t\mu(x)) \varphi(x) + t\mu(x) \bar{\psi}(x) ) &
\textrm{for $(x,t)$ in $U_1 \times [0,1]$} \\
f(x) & \textrm{for $(x,t)$ in $(X \setminus U_1) \times [0,1]$},
\end{cases}
\end{equation*}
is continuous. Furthermore, if $F_t \colon X \to \SB^p$ is defined by
$F_t(x) = F(x,t)$, then $F_t^{-1}(y) = f^{-1}(y) = g^{-1}(y)$ for every
$t$ in $[0,1]$, and $F_0 = f$. It remains to prove that the map $\bar{f}
\coloneqq F_1$ is homotopic to $g$. This can be done as follows. Note
that $\bar{f}=g$ on $U_2$. If $q \colon \SB^d \setminus \{y\} \to \R^d$
is the stereographic projection, then
\begin{equation*}
G(x,t)=
\begin{cases}
q^{-1} ( (1-t)q( \bar{f}(x) ) + tq(g(x)) ) & \textrm{for $(x,t)$ in $(X
\setminus f^{-1}(y)) \times [0,1]$}\\
\bar{f}(x) & \textrm{for $(x,t)$ in $U_2 \times [0,1]$}
\end{cases}
\end{equation*}
is a homotopy between $\bar{f}$ and $g$.
\end{proof}

A special case of Theorem~\ref{th-2-5}, with $X$ nonsingular, is
contained in \cite[Theorem~1.2]{bib32}.

Recall that each regular map from
$\T^n$ into $\SB^n$ is null homotopic, cf. \cite{bib11} or \cite{bib12}.
In particular, Theorem~\ref{th-2-5} shows that stratified-regular maps are more flexible than regular
maps. Other results illustrating this point can be found in
\cite{bib32, bib34}.

For any continuous map $f \colon X \to Y$ and any topological
$\F$-vector subbundle $\theta$ of $\varepsilon_Y^n(\F)$, the pullback
$f^*\theta$ will be regarded as a topological $\F$-vector subbundle of
$\varepsilon_X^n(\F)$.

\begin{proposition}\label{prop-2-6}
If the map $f$ is $\X$-regular, and the $\F$-vector bundle $\theta$ is
algebraic, then $f^*\theta$ is $\X$-algebraic. Similarly, if $f$ is
stratified-regular and $\theta$ is stratified-algebraic, then
$f^*\theta$ is stratified-algebraic.
\end{proposition}

\begin{proof}
The first assertion is obvious. For the second assertion, let $\TC$ be a
stratification of $X$ such that the map $f$ is $\TC$-regular, and let
$\Y$ be a stratification of $Y$ such that the $\F$-vector bundle
$\theta$ is $\Y$-algebraic. Then $\SC \coloneqq \{ T \cap f^{-1}(P) \mid
T \in \TC, P \in \Y \}$ is a stratification of $X$, and $f^*\theta$ is
$\SC$-algebraic.
\end{proof}

\begin{proof}[Proof of Theorem~\ref{th-1-3}.]
Let $h \colon X \to \SB^d$ be a homotopy equivalence. By
Theorem~\ref{th-2-5}, $h$ is homotopic to a stratified-regular map $f
\colon X \to \SB^d$. According to Proposition~\ref{prop-2-6}, if
$\theta$ is an algebraic $\F$-vector bundle on $\SB^d$, then $f^*\theta$
is a stratified-algebraic $\F$-vector bundle on $X$. The proof is
complete since every topological $\F$-vector bundle on $\SB^d$ is
isomorphic to an algebraic $\F$-vector bundle (cf.
Example~\ref{ex-1-4}).
\end{proof}

\section[Basic properties of stratified-algebraic vector
bundles]{\texorpdfstring{Basic properties of stratified-algebraic\\ vector
bundles}{Basic properties of stratified-algebraic vector bundles}}\label{sec-3}

Throughout this section, $X$ denotes a real algebraic variety, and $\X$
is a stratification of $X$. All modules that appear below are left
modules.

Vector bundles are often investigated by means of maps into
Grassmannians, cf. \cite{bib3, bib5, bib9, bib10, bib25, bib27, bib28,
bib29}. As in \cite{bib9, bib10}, the Grassmannian $\G_k(\F^n)$ of
$k$-dimensional $\F$-vector subspaces of $\F^n$ will be regarded as a
real algebraic variety. The tautological $\F$-vector bundle on
$\G_k(\F^n)$ will be denoted by $\gamma_k(\F^n)$. If $\xi$ is a
topological (resp. algebraic) $\F$-vector subbundle of $\varepsilon_X^n(\F)$ of rank $k$,
then the map $f \colon X \to \G_k(\F^n)$ defined by
\begin{equation*}
E(\xi)_x = \{ x \} \times f(x) \quad \textrm{for all $x$ in $X$}
\end{equation*}
is continuous (resp. regular).

Let $K$ be a finite nonempty collection of nonnegative integers and $n$ an
integer such that $n \geq k$ for every $k$ in $K$. We denote by
$\G_K(\F^n)$ the disjoint union of all $\G_k(\F^n)$ for $k$ in~$K$. The
tautological $\F$-vector bundle $\gamma_K(\F^n)$ on $\G_K(\F^n)$ is the
bundle whose restriction to $\G_k(\F^n)$ is $\gamma_k(\F^n)$ for
each $k$ in $K$. In particular, $\G_K(\F^n)$ is a real algebraic
variety, and $\gamma_K(\F^n)$ is an algebraic $\F$-vector subbundle of
the standard trivial $\F$-vector bundle on $\G_K(\F^n)$ with fiber
$\F^n$. We call $\G_K(\F^n)$ the $(K,n)$-\emph{multi-Grassmannian}. It
is explained below why we need this notion.

Any algebraic $\F$-vector bundle on $X$ has constant rank on each
Zariski connected component of $X$. This observation can be partially
generalized as follows.

\begin{proposition}\label{prop-3-1}
Assume that the variety $X$ is nonsingular. Then any stratified-algebraic
$\F$-vector bundle on $X$ has constant rank on each irreducible
component of $X$.
\end{proposition}

\begin{proof}
We may assume that $X$ is irreducible. It suffices to note that in each
stratification of $X$, one can find a stratum which is nonempty and
Zariski open in $X$.
\end{proof}

However, we encounter a different phenomenon for vector bundles on
singular varieties.

\begin{example}\label{ex-3-2}
The real algebraic curve
\begin{equation*}
C \coloneqq \{ (x,y) \in \R^2 \mid y^2 = x^2(x-1) \}
\end{equation*}
is irreducible, and hence Zariski connected. It has two connected
components in the Euclidean topology, $S_1 = \{(0,0)\}$ and $S_2 = C
\setminus S_1$. The collection $\C = \{S_1,S_2\}$ is a stratification of
$C$. Let $\xi$ be the topological $\F$-vector subbundle of
$\varepsilon_C^2(\F)$ such that $E(\xi|_{S_1}) = \{(0,0)\} \times (\F
\times \{0\})$ and $\xi|_{S_2} = \varepsilon_{S_2}^2(\F)$. Then $\xi$ is
$\C$-algebraic and it does not have constant rank. Furthermore, if $K =
\{1,2\}$ and $f \colon C \to \G_K(\F^2)$ is a map such that $f(S_1)
\subseteq \G_1(\F^2)$ and $f(S_2) \subseteq \G_2(\F^2)$, then $f$ is
$\C$-regular
with $\xi = f^*\gamma_K(\F^2)$.
\end{example}

\begin{proposition}\label{prop-3-3}
Any $\X$-algebraic $\F$-vector bundle $\xi$ on $X$ is of the form $\xi =
f^* \gamma_K(\F^n)$ for some multi-Grassmannian $\G_K(\F^n)$ and
$\X$-regular map $f \colon X \to \G_K(\F^n)$.
\end{proposition}

\begin{proof}
Let $\xi$ be an $\X$-algebraic $\F$-vector subbundle of
$\varepsilon_X^n(\F)$. For each stratum $S$ in $\X$, the function
\begin{equation*}
S \to \Z, \quad x \to \dim_{\F} E(\xi)_x
\end{equation*}
is locally constant in the Zariski topology, the $\F$-vector bundle
$\xi|_S$ being algebraic. Hence, the set $K \coloneqq \{ \dim_{\F}
E(\xi)_x \mid x \in X \}$ is finite. The map $f \colon X \to \G_K(\F^n)$
defined by
\begin{equation*}
E(\xi)_x = \{x\} \times f(x) \quad \textrm{for all $x$ in $X$}
\end{equation*}
is continuous and $\xi = f^*\gamma_K(\F^n)$. Furthermore, the map $f|_S$
is regular. Thus $f$ is $\X$-regular, as required.
\end{proof}

\begin{proposition}\label{prop-3-4}
Any stratified-algebraic $\F$-vector bundle $\xi$ on $X$ is of the
form
\begin{equation*}
\xi = f^* \gamma_K(\F^n)
\end{equation*}
for some multi-Grassmannian $\G_K(\F^n)$ and
stratified-regular map ${f \colon X \to \G_K(\F^n)}$.
\end{proposition}

\begin{proof}
It suffices to apply Proposition~\ref{prop-3-3}.
\end{proof}

\begin{proposition}\label{prop-3-5}
For a topological $\F$-vector bundle $\xi$ on $X$, the following
conditions are equivalent:
\begin{conditions}%
\item\label{prop-3-5-a} The bundle $\xi$ is stratified-algebraic.

\item\label{prop-3-5-b} There exists a filtration $\FC$ of $X$ such
that $\xi$ is $\overline{\FC}$-algebraic and each stratum in
$\overline{\FC}$ is nonsingular and equidimensional.%
\setcounter{favoritecondition}{\value{conditionsi}}%
\end{conditions}%
\begin{altconditions}[start=\value{favoritecondition}]%
\item\label{prop-3-5-b'} There exists a filtration $\FC'$ of $X$ such
that $\xi$ is $\overline{\FC}'$-algebraic.
\end{altconditions}
\end{proposition}

\begin{proof}
By Proposition~\ref{prop-3-4}, if the $\F$-vector bundle $\xi$ is
stratified-algebraic, then 
\begin{equation*}
\xi=f^*\gamma_K(\F^n)
\end{equation*}
for some
multi-Grassmannian $\G_K(\F^n)$ and stratified-regular map $f \colon X
\to \G_K(\F^n)$. According to Proposition~\ref{prop-2-2}, there exists a
filtration $\FC$ of $X$ such that $f$ is $\overline{\FC}$-regular and
each stratum in $\overline{\FC}$ is nonsingular and equidimensional. It
follows that $\xi$ is $\overline{\FC}$-algebraic. Consequently,
(\ref{prop-3-5-a}) implies (\ref{prop-3-5-b}). It is obvious that
(\ref{prop-3-5-b}) implies (\ref{prop-3-5-b'}), and (\ref{prop-3-5-b'})
implies (\ref{prop-3-5-a}).
\end{proof}

Recall that $d(\F) = \dim_{\R}\F$. As a consequence of
Theorem~\ref{th-2-5}, we obtain the following result on vector bundles
on low-dimensional varieties.

\begin{corollary}\label{cor-3-6}
Assume that the variety $X$ is compact and $\dim X \leq d(\F)$. Then any
topological $\F$-vector bundle of constant rank on $X$ is topologically
isomorphic to a stratified-algebraic $\F$-vector bundle.
\end{corollary}

\begin{proof}
Let $\xi$ be a topological $\F$-vector bundle of rank $r \geq 1$ on $X$.
Since $\dim X \leq d(\F)$, the bundle $\xi$ splits off a trivial vector
bundle of rank $r-1$. Moreover, if $\dim X < d(\F)$, then $\xi$ is
topologically trivial. These are well known topological facts, cf.
\cite[p.~99]{bib28}. Hence we may assume without loss of generality that
$r=1$ and $\dim X = d(\F)$. Then there exists a continuous map $f \colon
X \to \G_1(\F^2)$ such that the pullback $\F$-line bundle
$f^*\gamma_1(\F^2)$ is isomorphic to $\xi$. Recall that $\G_1(\F^2)$ is
biregularly isomorphic to the unit $d(\F)$-sphere. Consequently,
according to Theorem~\ref{th-2-5}, the map $f$ can be assumed to be
stratified-regular. Thus, it suffices to apply
Proposition~\ref{prop-2-6}.
\end{proof}

Subsequent results require some preparation. For any real algebraic
variety $Y$, denote by $\C(X,Y)$ the set of all continuous maps
from $X$ into $Y$. There are the following inclusions:
\begin{equation*}
\RC(X,Y) \subseteq \RC_{\X}(X,Y) \subseteq \RC^0(X,Y) \subseteq \C(X,Y),
\end{equation*}
where $\RC(X,Y)$ (resp. $\RC_{\X}(X,Y)$, $\RC^0(X,Y)$) is the set of all
regular (resp. $\X$-regular, stratified-regular) maps. Each of the sets
$\RC(X,\F)$, $\RC_{\X}(X,\F)$ and $\RC^0(X,\F)$ is a subring of the ring
$\C(X,\F)$. We next discuss various aspects of the Serre--Swan
construction \cite{bib41, bib43}, relating vector bundles and finitely
generated projective modules.

If $\xi$ is a topological $\F$-vector bundle on $X$, then the set
$\Gamma(\xi)$ of all (global) continuous sections of $\xi$ is a
$\C(X,\F)$-module. If $\varphi \colon \xi \to \eta$ is a morphism of
topological $\F$-vector bundles on $X$, then
\begin{equation*}
\Gamma(\varphi) \colon \Gamma(\xi) \to \Gamma(\eta), \quad
\Gamma(\varphi)(\sigma) = \varphi \circ \sigma
\end{equation*}
is a homomorphism of $\C(X,\F)$-modules. Since $X$ is homotopically
equivalent to a compact subset of $X$ (cf.
\cite[Corollary~9.3.7]{bib9}), it follows form \cite{bib43} that
$\Gamma$ is an equivalence of the category of topological $\F$-vector
bundles on $X$ with the category of finitely generated projective
$\C(X,\F)$-modules. We give below suitable counterparts of this result
for $\X$-algebraic and stratified-algebraic $\F$-vector bundles on $X$.

If $\xi$ is an $\X$-algebraic $\F$-vector bundle on $X$, an
\emph{$\X$-algebraic section} $u \colon X \to \xi$ is a continuous
section whose restriction $u|_S \colon S \to \xi|_S$ to each stratum $S$
in $\X$ is an algebraic section. In other words, $u \colon X \to E(\xi)$
is a continuous map such that $\pi(\xi) \circ u$ is the identity map of
$X$, and the restriction $u|_S \colon S \to \pi(\xi)^{-1}(S)$ is a
regular map of real algebraic varieties for each stratum $S$ in $\X$. The
set $\Gamma_{\X}(\xi)$ of all (global) $\X$-algebraic sections of $\xi$
is an $\RC_{\X}(X,\F)$-module. For any $\X$-algebraic morphism $\varphi
\colon \xi \to \eta$ of $\X$-algebraic $\F$-vector bundles on X,
\begin{equation*}
\Gamma_{\X}(\varphi) \colon \Gamma_{\X}(\xi) \to \Gamma_{\X}(\eta),
\quad \Gamma_{\X}(\varphi)(u)=\varphi \circ u
\end{equation*}
is a homomorphism of $\RC_{\X}(X,\F)$-modules.

The $\RC(X,\F)$-module $\Gamma_{\X}(\varepsilon_X^n(\F))$ is canonically
isomorphic to the direct sum $\RC_{\X}(X,\F)^n$ of $n$ copies of
$\RC_{\X}(X,\F)$. If $\xi$ is an $\X$-algebraic $\F$-vector subbundle of
$\varepsilon_X^n(\F)$, then $\Gamma_{\X}(\xi)$ will be regarded as a
submodule of $\Gamma_{\X}(\varepsilon_X^n(\F))$.

For any topological $\F$-vector subbundle $\xi$ of
$\varepsilon_X^n(\F)$, let $\xi^{\perp}$ denote its orthogonal
complement with respect to the standard inner product $\F^n \times \F^n
\to \F$. Thus, $\xi^{\perp}$ is a topological $\F$-vector subbundle of
$\varepsilon_X^n(\F)$ and $\xi \oplus \xi^{\perp} =
\varepsilon_X^n(\F)$. The orthogonal projection $\varepsilon_X^n(\F) \to
\xi$ is a topological morphism of $\F$-vector bundles.

\begin{proposition}\label{prop-3-7}
If $\xi$ is an $\X$-algebraic $\F$-vector subbundle of
$\varepsilon_X^n(\F)$, then $\xi^{\perp}$ also is an $\X$-algebraic
$\F$-vector subbundle of $\varepsilon_X^n(\F)$, and the orthogonal
projection $\varepsilon_X^n(\F) \to \xi$ is an $\X$-algebraic morphism.
In particular, $\xi$ is generated by $n$ (global) $\X$-algebraic
sections.
\end{proposition}

\begin{proof}
Let $S$ be a stratum in $\X$. One readily checks that $(\xi|_S)^\perp =
(\xi^{\perp})|_S$ is an algebraic $\F$-vector subbundle of
$\varepsilon_S^n(\F) = \varepsilon_X^n(\F)|_S$, and the orthogonal
projection $\varepsilon_S^n(\F) \to \xi|_S$ is an algebraic morphism.
The last assertion in the proposition follows immediately.
\end{proof}

\begin{proposition}\label{prop-3-8}
If $\xi$ is an $\X$-algebraic $\F$-vector bundle on $X$, then the
$\RC_{\X}(X,\F)$-module $\Gamma_{\X}(\xi)$ is finitely generated and
projective. Furthermore, $\Gamma_{\X}$ is an equivalence of the category
of $\X$-algebraic $\F$-vector bundles on $X$ with the category of
finitely generated projective $\RC_{\X}(X,\F)$-modules.
\end{proposition}

\begin{proof}
According to Proposition~\ref{prop-3-7}, if $\xi$ is an $\X$-algebraic
$\F$-vector subbundle of $\varepsilon_X^n(\F)$, then
\begin{equation*}
\Gamma_{\X}(\xi) \oplus \Gamma_{\X}(\xi^{\perp}) =
\Gamma_{\X}(\varepsilon_X^n(\F)).
\end{equation*}
Hence, $\Gamma_{\X}(\xi)$ is a finitely generated projective
$\RC_{\X}(X,\F)$-module. Furthermore, $\Gamma_{\X}$ is an equivalence of
categories since, in view of Proposition~\ref{prop-3-7}, the proof given
in \cite[pp.~30, 31]{bib3} that $\Gamma$ is an equivalence of categories
in the topological framework can easily be adapted to $\Gamma_{\X}$.
\end{proof}

If $\xi$ is an algebraic $\F$-vector bundle on $X$, then the set
$\Gamma_{\mathrm{alg}}(\xi)$ of all (global) algebraic sections of $\xi$
is an $\RC(X,\F)$-module. For any morphism $\varphi \colon \xi \to \eta$
of algebraic $\F$-vector bundles on $X$,
\begin{equation*}
\Gamma_{\mathrm{alg}}(\varphi) \colon \Gamma_{\mathrm{alg}} (\xi) \to
\Gamma_{\mathrm{alg}}(\eta), \quad \Gamma_{\mathrm{alg}} (\varphi)
(\sigma) = \varphi \circ \sigma
\end{equation*}
is a homomorphism of $\RC(X, \F)$-modules. It is well known that
$\Gamma_{\mathrm{alg}}$ is an equivalence of the category of algebraic
$\F$-vector bundles on $X$ with the category of finitely generated
projective $\RC(X,\F)$-modules, cf. \cite[Proposition~12.1.12]{bib9}.
This result is a special case of Proposition~\ref{prop-3-8} since
algebraic $\F$-vector bundles on $X$ coincide with $\{X\}$-algebraic
$\F$-vector bundles, and $\RC(X,\F) = \RC_{\{X\}}(X,\F)$.

There is also a version of Proposition~\ref{prop-3-8} for
stratified-algebraic vector bundles. If $\xi$ is a stratified-algebraic
$\F$-vector bundle on $X$, a \emph{stratified-algebraic section} $u
\colon X \to \xi$ is an $\SC$-algebraic section for some stratification
$\SC$ of $X$ such that $\xi$ is an $\SC$-algebraic $\F$-vector bundle. The
set $\Gamma_{\mathrm{str}}(\xi)$ of all (global) stratified-algebraic
sections of $\xi$ is an $\RC^0(X, \F)$-module. For any
stratified-algebraic morphism $\varphi \colon \xi \to \eta$ of
stratified-algebraic $\F$-vector bundles on $X$,
\begin{equation*}
\Gamma_{\mathrm{str}}(\varphi) \colon \Gamma_{\mathrm{str}}(\xi) \to
\Gamma_{\mathrm{str}}(\eta), \quad \Gamma_{\mathrm{str}} (\varphi)
(\sigma) = \varphi \circ \sigma
\end{equation*}
is a homomorphism of $\RC^0(X,\F)$-modules.

\begin{proposition}\label{prop-3-9}
If $\xi$ is a stratified-algebraic $\F$-vector bundle on $X$, then the
$\RC^0(X,\F)$-module $\Gamma_{\mathrm{str}}(\xi)$ is finitely generated
and projective. Furthermore, $\Gamma_{\mathrm{str}}$ is an equivalence of
the category of stratified-algebraic $\F$-vector bundles on $X$ with the
category of finitely generated projective $\RC^0(X,\F)$-modules.
\end{proposition}

\begin{proof}
One proceeds as in the proof of Proposition~\ref{prop-3-8}.
\end{proof}

We identify the direct sum $\varepsilon_X^n(\F) \oplus
\varepsilon_X^m(\F)$ with $\varepsilon_X^{n+m}(\F)$. Consequently, if
$\xi \subseteq \varepsilon_X^n(\F)$ and $\eta \subseteq
\varepsilon_X^m(\F)$ are $\F$-vector subbundles, then $\xi \oplus \eta
\subseteq \varepsilon_X^{n+m}(\F)$ is an $\F$-vector subbundle.

It is convenient to bring into play the sets of isomorphism classes of
vector bundles of types considered above. Denote by
$\VB_{\F\mathrm{\mhyphen alg}}(X)$, $\VB_{\F\mhyphen\X}(X)$, $\VB_{\F\mathrm{\mhyphen str}}(X)$
and $\VB_{\F}(X)$ the sets of isomorphism classes (in the appropriate
category) of algebraic, $\X$-algebraic, stratified-algebraic and
topological $\F$-vector bundles on $X$. Each of these sets of
isomorphism classes is a commutative monoid with operation induced by
direct sum of $\F$-vector bundles. There are obvious canonical
homomorphisms of monoids
\begin{equation*}
\VB_{\F\mathrm{\mhyphen alg}}(X) \to \VB_{\F \mhyphen \X}(X) \to \VB_{\F\mathrm{\mhyphen str}}(X)
\to \VB_{\F}(X). 
\end{equation*}
For example, if $\xi$ is an $\X$-algebraic $\F$-vector bundle on $X$,
then $\VB_{\F \mhyphen \X}(X) \to \VB_{\F\mathrm{\mhyphen str}}(X)$ sends the isomorphism
class of $\xi$ in the category of $\X$-algebraic $\F$-vector bundles on
$X$ to the isomorphism class of $\xi$ in the category of
stratified-algebraic $\F$-vector bundles on $X$. Any composition of
these homomorphisms will also be called a canonical homomorphism. Note
that $\VB_{\F\mathrm{\mhyphen alg}}(x) = \VB_{\F \mhyphen \{X\}}(X)$.

For any ring $R$ (associative with $1$), the set $P(R)$ of isomorphism
classes of finitely generated projective $R$-modules is a commutative
monoid, with operation induced by direct sum of $R$-modules. If $R$ is a
subring of a ring $R'$, then there is a canonical homomorphism
\begin{equation*}
P(R) \to P(R'),
\end{equation*}
induced by the correspondence which assigns to an $R$-module $M$ the
$R'$-module $R' \otimes_R M$.

There are canonical homomorphisms of monoids
\begin{align*}
&\Gamma_{\mathrm{alg}} \colon \VB_{\F\mathrm{\mhyphen alg}}(X) \to
P(\RC(X,\F)),
&&\Gamma_{\X} \colon \VB_{\F \mhyphen \X}(X) \to
P(\RC_{\X}(X,\F)), \\
&\Gamma_{\mathrm{str}} \colon \VB_{\F\mathrm{\mhyphen str}}(X) \to
P(\RC^0(X,\F)), 
&&\Gamma \colon \VB_{\F}(X) \to P(\C(X,\F)),
\end{align*}
induced by the global section functor in the appropriate category of
$\F$-vector bundles on $X$. For example, $\Gamma_{\X} \colon
\VB_{\F \mhyphen \X}(X) \to P(\RC_{\X}(X,\F))$ sends the isomorphism class of an
$\X$-algebraic $\F$-vector bundle $\xi$ on $X$ to the isomorphism class
of the $\RC_{\X}(X,\F)$-module $\Gamma_{\X}(\xi)$ (cf.
Proposition~\ref{prop-3-8}).

\begin{theorem}\label{th-3-10}
The diagram
\begin{diagram}
\VB_{\F\mathrm{\mhyphen alg}}(X) & \rTo & \VB_{\F \mhyphen \X}(X) & \rTo &
\VB_{\F\mathrm{\mhyphen str}}(X) & \rTo & \VB_{\F}(X) \\
\dTo^{\Gamma_{\mathrm{alg}}} && \dTo^{\Gamma_{\X}} &&
\dTo^{\Gamma_{\mathrm{str}}} && \dTo^{\Gamma} \\
P(\RC(X,\F)) & \rTo & P(\RC_{\X}(X,\F)) & \rTo & P(\RC^0(X,\F)) & \rTo & P(\C(X,\F))  
\end{diagram}
is commutative, and the vertical maps are all bijective. Furthermore, if
the variety $X$ is compact, then the horizontal maps are all injective.
\end{theorem}

\begin{proof}
Let $\xi$ be an $\X$-algebraic $\F$-vector subbundle of
$\varepsilon_X^n(\F)$. Tensoring
\begin{equation*}
\Gamma_{\X}(\xi) \oplus \Gamma_{\X} (\xi^{\perp}) = \Gamma_{\X}
(\varepsilon_X^n(\F))
\end{equation*}
with $\RC^0(X,\F)$ (over the ring $\RC_{\X}(X,\F)$) and identifying
$\RC^0(X,\F) \otimes \Gamma_{\X}(\varepsilon_X^n(\F))$ with
$\Gamma_{\mathrm{str}}(\varepsilon_X^n(\F))$, one readily checks that
the $\RC^0(X,\F)$-modules $\RC^0(X,\F) \otimes \Gamma_{\X}(\xi)$ and
$\Gamma_{\mathrm{str}}(\xi)$ are isomorphic. This implies that the
middle square in the diagram is commutative. Similar arguments show
that the other two squares also are commutative. According to
Propositions~\ref{prop-3-8} and~\ref{prop-3-9}, the maps
$\Gamma_{\mathrm{alg}}$, $\Gamma_{\X}$ and $\Gamma_{\mathrm{str}}$ are
bijective. By \cite{bib41}, the map $\Gamma$ is bijective (since $X$ is
homotopically equivalent to a compact subset of $X$, cf.
\cite[Corollary~9.3.7]{bib9}).

Suppose that the variety $X$ is compact. It suffices to prove that each
map in the bottom row of the diagram is injective. This follows from
Swan's theorem \cite[Theorem~2.2]{bib44}. Indeed, $\C(X,\F)$ is a
topological ring with topology induced by the $\sup$ norm. Each subring of
$\C(X,\F)$ is a topological ring with the subspace topology. By the
Weierstrass approximation theorem, $\RC(X,\F)$ is dense in $\C(X,\F)$.
Consequently, Swan's theorem is applicable.
\end{proof}

Denote by $K_{\F\mathrm{\mhyphen alg}}(X)$, $K_{\F \mhyphen \X}(X)$,
$K_{\F\mathrm{\mhyphen str}}(X)$ and $K_{\F}(X)$ the Grothendieck group
of the commutative monoids $\VB_{\F\mathrm{\mhyphen alg}}(X)$,
$\VB_{\F \mhyphen \X}(X)$, $\VB_{\F\mathrm{\mhyphen str}}(X)$ and $\VB_{\F}(X)$.
Note that $K_{\F\mathrm{\mhyphen alg}}(X) = K_{\F \mhyphen \{X\}}(X)$.

As usual, for any ring $R$, the Grothendieck group of the commutative
monoid $P(R)$ will be denoted by $K_0(R)$.

\begin{corollary}\label{cor-3-11}
The commutative diagram in Theorem~\ref{th-3-10} gives rise to a
commutative diagram
\begin{diagram}
K_{\F\mathrm{\mhyphen alg}}(X) & \rTo & K_{\F \mhyphen \X}(X) & \rTo &
K_{\F\mathrm{\mhyphen str}}(X) & \rTo & K_{\F}(X) \\
\dTo^{\Gamma_{\mathrm{alg}}} && \dTo^{\Gamma_{\X}} &&
\dTo^{\Gamma_{\mathrm{str}}} && \dTo^{\Gamma} \\
K_0(\RC(X,\F)) & \rTo & K_0(\RC_{\X}(X,\F)) & \rTo & K_0(\RC^0(X,\F)) & \rTo & K_0(\C(X,\F))  
\end{diagram}
in which the vertical homomorphisms are all isomorphisms. Furthermore,
if the variety $X$ is compact, then the horizontal homomorphisms are all
monomorphisms.
\end{corollary}

\begin{proof}
It suffices to make use of Theorem~\ref{th-3-10}.
\end{proof}

\begin{corollary}\label{cor-3-12}
Assume that the variety $X$ is compact. For a stratified-algebraic
$\F$-vector bundle $\xi$ on $X$, the following conditions are equivalent:
\begin{conditions}
\item $\xi$ is isomorphic in the category of stratified-algebraic
$\F$-vector bundles on $X$ to an $\X$-algebraic $\F$-vector bundle on
$X$.

\item $\xi$ is isomorphic in the category of topological $\F$-vector
bundles on $X$ to an $\X$-algebraic $\F$-vector bundle on $X$.

\item The class of $\xi$ in $K_{\F\mathrm{\mhyphen str}}(X)$ belongs to
the image of the canonical homomorphism $K_{\F \mhyphen \X}(X) \to
K_{\F\mathrm{\mhyphen str}}(X)$.

\item The class of $\xi$ in $K_{\F}(X)$ belongs to the image of the
canonical homomorphism $K_{\F \mhyphen \X}(X) \to K_{\F}(X)$.
\end{conditions}
\end{corollary}

\begin{proof}
In view of Theorem~\ref{th-3-10} and Corollary~\ref{cor-3-11}, it
suffices to apply Swan's theorem \cite[Theorem~2.2]{bib44}.
\end{proof}

In the same way, we obtain the next two corollaries, which are recorded
for the sake of completeness.

\begin{corollary}\label{cor-3-13}
Assume that the variety $X$ is compact. For a topological $\F$-vector
bundle $\xi$ on $X$, the following conditions are equivalent:
\begin{conditions}
\item $\xi$ is isomorphic in the category of topological $\F$-vector
bundles on $X$ to an $\X$-algebraic $\F$-vector bundle on $X$.

\item The class of $\xi$ in $K_{\F}(X)$ belongs to the image of the
canonical homomorphism $K_{\F \mhyphen \X}(X) \to K_{\F}(X)$.
\end{conditions}
\end{corollary}

\begin{corollary}\label{cor-3-14}
Assume that the variety $X$ is compact. For a topological $\F$-vector
bundle $\xi$ on $X$, the following conditions are equivalent:
\begin{conditions}
\item $\xi$ is isomorphic in the category of topological $\F$-vector
bundles on $X$ to a stratified-algebraic $\F$-vector bundle on $X$.

\item The class of $\xi$ in $K_{\F}(X)$ belongs to the image of the
canonical homomorphism $K_{\F\mathrm{\mhyphen str}}(X) \to K_{\F}(X)$.
\end{conditions}
\end{corollary}

Now we are in a position to prove two results announced in
Section~\ref{sec-1}.

\begin{proof}[Proof of Theorem~\ref{th-1-5}]
Recall that the group $K_{\F}(X)$ is finitely generated (cf.
\cite[Exercise~III.7.5]{bib29} or the spectral sequence in \cite{bib4,
bib21}). According to Corollary~\ref{cor-3-11}, the group
$K_{\F\mathrm{\mhyphen str}}(X)$ also is finitely generated. Hence there
exist stratified-algebraic $\F$-vector bundles $\xi_1, \ldots, \xi_r$ on
$X$ whose classes in $K_{\F\mathrm{\mhyphen str}}(X)$ generate this
group. We can find a stratification $\SC$ of $X$ such that each $\xi_i$
is $\SC$-algebraic for $1 \leq i \leq r$. Consequently, the canonical
homomorphism $K_{\F \mhyphen \SC}(X) \to K_{\F\mathrm{\mhyphen str}}(X)$ is
surjective. The proof is complete in view of Corollary~\ref{cor-3-12}.
\end{proof}

\begin{proof}[Proof of Theorem~\ref{th-1-6}]
As in the proof of Theorem~\ref{th-1-5}, we obtain stratified-algebraic
$\F$-vector bundles $\xi_1, \ldots, \xi_r$ on $X$ whose classes in
$K_{\F\mathrm{\mhyphen str}}(X)$ generate this group. According to
Proposition~\ref{prop-3-4}, each $\xi_i$ is of the form $\xi_i = f_i^*
\gamma_{K_i} (\F^{n_i})$ for some multi-Grassmannian
$\G_{K_i}(\F^{n_i})$ and stratified-regular map $f_i \colon X \to
\G_{K_i}(\F^{n_i})$. Set
\begin{equation*}
f \coloneqq (f_1, \ldots, f_r) \colon X \to G \coloneqq
\G_{K_1}(\F^{n_1}) \times \cdots \times \G_{K_r}(\F^{n_r}).
\end{equation*}
By Hironaka's theorem on resolution of singularities \cite{bib25,
bib30}, there exists a birational multiblowup $\rho \colon \tilde{X} \to
X$ such that the variety $\tilde{X}$ is nonsingular. The composite map
$f \circ \rho \colon \tilde{X} \to G$ is stratified-regular, and hence,
in view of Proposition~\ref{prop-2-2}, it is continuous rational. Now,
Hironaka's theorem on resolution of points of indeterminacy
\cite{bib26, bib30} implies the existence of a birational multiblowup
$\sigma \colon X' \to \tilde{X}$ such that the composite map $f \circ
\rho \circ \sigma \colon X' \to G$ is regular. The variety $X'$ is
nonsingular and $\pi \coloneqq \rho \circ \sigma \colon X' \to X$ is a
birational multiblowup. Since the composite map $f_i \circ \pi$ is
regular, the pullback
\begin{equation*}
\pi^*\xi_i = \pi^*(f_i^* \gamma_{K_i}(\F^{n_i})) = (f_i \circ \pi)^*
\gamma_{K_i} (\F^{n_i})
\end{equation*}
is an algebraic $\F$-vector bundle on $X'$ for $1 \leq i \leq r$. Since
the classes of $\xi_1, \ldots, \xi_r$ generate the group
$K_{\F\mathrm{\mhyphen str}}(X)$, it follows that the image of the
homomorphism 
\begin{equation*}
\pi^* \colon K_{\F \mathrm{\mhyphen str}} (X) \to K_{\F
\mathrm{\mhyphen str}}(X')
\end{equation*}
induced by $\pi$ is contained in the image
of the canonical homomorphism
\begin{equation*}
K_{\F \mathrm{\mhyphen alg}} (X') = K_{\F \mhyphen \{X'\}} (X') \to K_{\F
\mathrm{\mhyphen str}} (X').
\end{equation*}
Hence, in view of Corollary~\ref{cor-3-11} (with $X=X'$), for any
stratified-algebraic $\F$-vector bundle $\xi$ on $X$, the pullback
$\pi^*\xi$ is isomorphic in the category of stratified-algebraic
$\F$-vector bundles on $X'$ to an algebraic $\F$-vector bundle on $X'$.
This proves Theorem~\ref{th-1-6}.
\end{proof}

We conclude this section by showing that the categories of vector bundles
considered here are closed under standard operations. This is made
precise in Proposition~\ref{prop-3-15} and its proof. In fact it was
already done for the direct sum $\oplus$ when the monoids
$\VB_{\F \mhyphen \X}(X)$ and $\VB_{\F\mathrm{\mhyphen str}}(X)$ were introduced.

\begin{proposition}\label{prop-3-15}
Let $\xi$ and $\eta$ be $\X$-algebraic $\F$-vector bundles on $X$. Then
the $\F$-vector bundles $\xi \oplus \eta$, $\Hom(\xi, \eta)$ and
$\xi^{\vee}$ (dual bundle) are $\X$-algebraic. If $\F = \R$ or $\F =
\CB$, then the $\F$-vector bundles $\xi \otimes \eta$ and $\bigwedge^k\xi$
($k$th exterior power) are $\X$-algebraic. Furthermore, the analogous
statements hold true for stratified-algebraic $\F$-vector bundles.
\end{proposition}

\begin{proof}
Suppose that $\xi$ is an $\X$-algebraic subbundle of
$\varepsilon_X^n(\F)$ and $\eta$ is an $\X$-algebraic subbundle of
$\varepsilon_X^m(\F)$. We already know that $\xi \oplus \eta$ is
regarded as a subbundle of 
\begin{equation*}
\varepsilon_X^n(\F) \oplus
\varepsilon_X^m(\F) \cong \varepsilon_X^{n+m}(\F).
\end{equation*}
Since $\xi \oplus
\xi^{\perp} = \varepsilon_X^n(\F)$ and $\eta \oplus \eta^{\perp} =
\varepsilon_X^m(\F)$, (cf. Proposition~\ref{prop-3-7}), $\Hom(\xi,\eta)$
can be regarded as a subbundle of $\Hom(\varepsilon_X^n(\F),
\varepsilon_X^m(\F)) \cong \varepsilon_X^{nm}(\F)$. By dualizing the
orthogonal projection $\varepsilon_X^n(\F) \to \xi$, we regard
$\xi^{\vee}$ as a subbundle of $\varepsilon_X^n(\F)^{\vee} \cong
\varepsilon_X^n(\F)$. If $\F=\R$ or $\F=\CB$, then
$\xi \otimes \eta$ can be regarded as a subbundle of
$\varepsilon_X^n(\F) \otimes \varepsilon_X^m(\F) \cong
\varepsilon_X^{nm}(\F)$, and $\bigwedge^k\xi$ can be regarded as a
subbundle of  $\bigwedge^k\varepsilon_X^n(\F) \cong
\varepsilon_X^q(\F)$, where $q = \binom{n}{k}$ is the binomial
coefficient. After these identifications, each of the vector bundles
under consideration becomes an $\X$-algebraic subbundle of
$\varepsilon_X^p(\F)$ for an appropriate $p$. The same argument works
for stratified-algebraic vector bundles.
\end{proof}

\section{Approximation by stratified-regular maps}\label{sec-4}
As in Section~\ref{sec-3}, for any real algebraic variety $X$ with
stratification $\X$, and any real algebraic variety $Y$, we have the
following inclusions:
\begin{equation*}
\RC(X,Y) \subseteq \RC_{\X}(X,Y) \subseteq \RC^0(X,Y) \subseteq \C(X,Y).
\end{equation*}
A challenging problem is to find a useful description of the closure of
each of the sets $\RC(X,Y)$, $\RC_{\X}(X,Y)$ and $\RC^0(X,Y)$ in the space
$\C(X,Y)$, endowed with the compact-open topology. In other words, the
problem is to find a characterization of these maps in $\C(X,Y)$ that
can be approximated by either regular or $\X$-regular or
stratified-regular maps. Approximation by regular maps is investigated
in \cite{bib9, bib10, bib12, bib13, bib16, bib18}. As demonstrated in
\cite{bib32, bib34}, stratified-regular maps are much more flexible
than regular maps. In the present section, we prove approximation
theorems for maps with values in multi-Grassmannians. It is important
for applications to obtain results in which approximating maps satisfy
some extra conditions.

We first recall a key extension result due to Koll\'ar and Nowak
\cite[Theorem~9, Proposition~10]{bib31}.

\begin{theorem}[\cite{bib31}]\label{th-4-1}
Let $X$ be a real algebraic variety. Let $A$ and $B$ be Zariski closed
subvarieties of $X$ with $B \subseteq A$. For any stratified-regular
function $f \colon A \to \R$ whose restriction $f|_{A \setminus B}$ is a
regular function, there exists a stratified-regular function $F \colon X
\to \R$ such that $F|_A = f$ and $F|_{X \setminus B}$ is a regular function.
\end{theorem}

It should be mentioned that in \cite{bib31}, Theorem~\ref{th-4-1} is stated in
terms of hereditarily rational functions (cf. Remark~\ref{rem-2-3}).

The following terminology will be convenient. We say that a topological
$\F$-vector bundle on a real algebraic variety $X$ \emph{admits an
algebraic} (resp. \emph{$\X$-algebraic}, \emph{stratified-algebraic})
structure if it is topologically isomorphic to an algebraic (resp.
$\X$-algebraic, stratified-algebraic) $\F$-vector bundle on $X$. If $Z$
is a Zariski closed subvariety of $X$, we say that a multiblowup $\pi
\colon X' \to X$ of $X$ is \emph{over} $Z$ if the restriction $\pi_Z
\colon X' \setminus \pi^{-1}(Z) \to X \setminus Z$ of $\pi$ is a
biregular isomorphism.

\begin{notation}\label{not-4-2}
In the remainder of this section, $X$ denotes a compact real algebraic
variety, and $A$ is a Zariski closed subvariety of $X$. Moreover,
$\G_K(\F^n)$ is the $(K,n)$-multi-Grassmannian and $\gamma_K(\F^n)$ is
the tautological $\F$-vector bundle on $\G_K(\F^n)$, for some fixed
$(K,n)$ (cf. Section~\ref{sec-2}).
\end{notation}

Our basic approximation result is the following.

\begin{lemma}\label{lem-4-3}
Let $B$ be a Zariski closed subvariety of $X$ with $B \subseteq A$. Let $f
\colon X \to \R$ be a continuous function such that $f|_A$ is a
stratified-regular function and $f|_{A \setminus B}$ is a regular
function. For every $\varepsilon > 0$, there exists a stratified-regular
function $g \colon X \to \R$ such that $g|_A = f|_A$, the restriction
$g|_{X \setminus B}$ is a regular function, and
\begin{equation*}
\abs{g(x)-f(x)} < \varepsilon \quad \textrm{for all $x$ in $X$.}
\end{equation*}
\end{lemma}

\begin{proof}
According to Theorem~\ref{th-4-1}, there exists a stratified-regular
function $h \colon X \to \R$ such that $h|_A = f|_A$ and the restriction
$h|_{X \setminus B}$ is a regular function. The function $\varphi
\coloneqq f-h$ is continuous and $\varphi|_A=0$. Hence, by
\cite[Lemma~2.1]{bib15}, one can find a regular function $\psi \colon X
\to \R$ satisfying $\psi|_A=0$ and
\begin{equation*}
\abs{\varphi(x) - \psi(x)} < \varepsilon \quad \textrm{for all $x$ in
$X$.}
\end{equation*}
The function $g \coloneqq h + \psi$ has the required properties.
\end{proof}

For any stratification $\X$ of $X$ and any Zariski closed subvariety $B$
of $X$, we set
\begin{equation*}
\X \cap B \coloneqq \{ S \cap B \mid S \in \X \} \quad \textrm{and}
\quad \X \setminus B \coloneqq \{ S \setminus B \mid S \in \X \}.
\end{equation*}
Then
\begin{equation*}
\X_B \coloneqq (\X \cap B) \cup \{ \X \setminus B \}
\end{equation*}
is a stratification of $X$. Obviously $\X_B \cap A$ is a
stratification of $A$.

\begin{theorem}\label{th-4-4}
Let $\X$ be a stratification of $X$ and let $B$ be a Zariski closed
subvariety of~$X$. Let $f \colon X \to \G_K(\F^n)$
be a continuous map whose restriction $f|_A$ is $(\X_B \cap A)$-regular.
If $B \subseteq A$ and $A \setminus B$ is a stratum in $\X_B$, then
the following conditions are equivalent:
\begin{conditions}
\item\label{th-4-4-a} Each neighborhood of $f$ in $\C(X, \G_K(\F^n))$ contains an
$\X_B$-regular map 
\begin{equation*}
g \colon X \to \G_K(\F^n)
\end{equation*}
with $g|_A=f|_A$.

\item\label{th-4-4-b} The map $f$ is homotopic to an $\X_B$-regular map $h \colon X \to
\G_K(\F^n)$ with $h|_A = f|_A$.

\item\label{th-4-4-c} The pullback $\F$-vector bundle $f^*\gamma_K(\F^n)$ on $X$ admits
an $\X_B$-algebraic structure.
\end{conditions}
\end{theorem}

\begin{proof}
It is a standard fact that (\ref{th-4-4-a}) implies (\ref{th-4-4-b}). If
(\ref{th-4-4-b}) holds, then the $\X_B$-algebraic $\F$-vector bundle
$h^*\gamma_K(\F^n)$ is topologically isomorphic to $f^*\gamma_K(\F^n)$,
and hence (\ref{th-4-4-c}) holds. It suffices to prove that
(\ref{th-4-4-c}) implies (\ref{th-4-4-a}).

Suppose that (\ref{th-4-4-c}) holds. In order to simplify notation, we
set $G = \G_K(\F^n)$, $\gamma=\gamma_K(F^n)$, $\A=\X_B\cap A$, and
$\varepsilon_X^l = \varepsilon_X^l(\F^n)$ for any nonnegative integer
$l$. Condition (\ref{th-4-4-c}) implies the existence of a topological
isomorphism $\varphi \colon \xi \to f^*\gamma$, where $\xi$ is an
$\X_B$-algebraic $\F$-vector subbundle of $\varepsilon_X^k$ for some
$k$. Regarding $f^*\gamma$ as an $\F$-vector subbundle of
$\varepsilon_X^n$, we get a continuous section $w \colon X \to \Hom(\xi,
\varepsilon_X^n)$, defined by $w(x)(e)=\varphi(e)$ for all points $x$ in
$X$ and all vectors $e$ in the fiber $E(\xi)_x$. In particular, $w(x)
\colon E(\xi)_x \to \{x\} \times \F^n$ is an injective $\F$-linear
transformation satisfying
\begin{equation*}
w(x)(E(\xi)_x) = \{x\} \times f(x).
\end{equation*}
Making use of the equalities $\varepsilon_X^k = \xi \oplus \xi^{\perp}$
and $\varepsilon_X^n = f^*\gamma \oplus f^* \gamma^{\perp}$, we find a
continuous section $\bar{w} \colon X \to \Hom(\varepsilon_X^k,
\varepsilon_X^n)$ satisfying $\bar{w}(x) (E(\varepsilon_X^k)_x)
\subseteq E(f^*\gamma)_x$ and $\bar{w}(x)|_{E(\xi)_x} = w(x)$ for all
$x$ in $X$. Since $X$ is compact and the $\F$-vector bundle
$\Hom(\varepsilon_X^k, \varepsilon_X^n)$ can be identified with
$\varepsilon_X^{kn}$, the Weierstrass approximation theorem implies the
existence of an algebraic section $v \colon X \to \Hom(\varepsilon_X^k,
\varepsilon_X^n)$ arbitrarily close to $\bar{w}$ in the compact-open
topology. If $\rho \colon \varepsilon_X^n \to f^*\gamma$ is the
orthogonal projection and $\iota \colon f^*\gamma \hookrightarrow
\varepsilon_X^n$ is the inclusion morphism, then $\bar{v} \coloneqq
\iota \circ \rho \circ v \colon X \to \Hom(\varepsilon_X^k,
\varepsilon_X^n)$ is a continuous section close to $\bar{w}$. If $v$ is
sufficiently close to $\bar{w}$, then the $\F$-linear transformation
$\bar{v}(x) \colon E(\xi)_x \to \{x\} \times \F^n$ is injective and
\begin{equation*}
\bar{v}(x)(E(\xi)_x) = \{x\} \times f(x)
\end{equation*}
for all $x$ in $X$. Since the map $f|_A$ is $\A$-regular, the
$\F$-vector bundle $(f|_A)^*\gamma = (f^*\gamma)|_A$ on $A$ is
$\A$-algebraic. Hence, in view of Proposition~\ref{prop-3-7}, the
restriction $\rho_A \colon \varepsilon_X^n|_A \to (f^*\gamma)|_A$ of
$\rho$ is an $\A$-algebraic morphism. Consequently, the restriction
$\bar{v}|_A \colon A \to \Hom(\varepsilon_X^k, \varepsilon_X^n)$
of $\bar{v}$ is an $\A$-algebraic section. By
Lemma \ref{lem-4-3}, there exists a continuous section 
\begin{equation*}
u \colon X \to
\Hom(\varepsilon_X^k, \varepsilon_X^n) \cong \varepsilon_X^{kn}
\end{equation*}
such that $u|_A = \bar{v}|_A$, the restriction $u|_{X \setminus B}$ is an
algebraic section, and $u$ is close to $\bar{v}$. If $u$ is sufficiently
close to $\bar{v}$, then the $\F$-linear transformation $u(x) \colon
E(\xi)_x \to \{x\} \times \F^n$ is injective for all $x$ in $X$. Now, the
map $g \colon X \to G$, defined by
\begin{equation*}
u(x)(E(\xi)_x) = \{x\} \times g(x)
\end{equation*}
for all $x$ in $X$, is continuous and close to $f$, and $g|_A = f|_A$.
If $T$ is a stratum in $\X_B$, then $u|_T$ is an algebraic section and
$\xi|_T$ is an algebraic $\F$-vector subbundle of $\varepsilon_X^k|_T$,
and hence the map $g|_T$ is regular (cf.
\cite[Proposition~3.4.7]{bib9}). Consequently, $g$ is an $\X_B$-regular
map. Thus, (\ref{th-4-4-c}) implies (\ref{th-4-4-a}).
\end{proof}

It is worthwhile to state several consequences of Theorem~\ref{th-4-4}.

\begin{corollary}\label{cor-4-5}
Let $f \colon X \to \G_K(\F^n)$ be a continuous map whose restriction
$f|_A$ is regular. Then the following conditions are equivalent:
\begin{conditions}
\item Each neighborhood of $f$ in $\C(X, \G_K(\F^n))$ contains a regular
map $g \colon X \to \G_K (\F^n)$ with $g|_A=f|_A$.

\item The map $f$ is homotopic to a regular map $h \colon X \to
\G_K(\F^n)$ with $h|_A=f|_A$.

\item The pullback $\F$-vector bundle $f^*\gamma_K(\F^n)$ on $X$ admits
an algebraic structure.
\end{conditions}
\end{corollary}

\begin{proof}
If $\X=\{X\}$ and $B = \varnothing$, then $\X_B = \{X\}$ and $\X_B\cap A
= \{A\}$. Hence, it suffices to apply Theorem~\ref{th-4-4}.
\end{proof}

Corollary~\ref{cor-4-5} with $A = \varnothing$ is well known, cf.
\cite{bib10} or \cite[Theorem~13.3.1]{bib9}.

\begin{corollary}\label{cor-4-6}
Let $\X$ be a stratification of $X$ such that $\A \coloneqq \{ S \in \X
\mid S \subseteq A\}$ is a stratification of $A$. Let $f \colon X \to
\G_K(\F^n)$ be a continuous map whose restriction $f|_A$ is
$\A$-regular. Then the following conditions are equivalent:
\begin{conditions}
\item Each neighborhood of $f$ in $\C(X, \G_K(\F^n))$ contains an
$\X$-regular map $g \colon X \to \G_K(\F^n)$ with $g|_A = f|_A$.

\item The map $f$ is homotopic to an $\X$-regular map $h \colon X \to
\G_K(\F^n)$ with $h|_A=f|_A$.

\item The pullback $\F$-vector bundle $f^*\gamma_K(\F^n)$ on $X$ admits
an $\X$-algebraic structure.
\end{conditions}
\end{corollary}

\begin{proof}
If $B=A$, then $\X_B=\X$ and $\X_B\cap A = \A$, and hence it suffices to
apply Theorem~\ref{th-4-4}.
\end{proof}

\begin{corollary}\label{cor-4-7}
Let $\X$ be a stratification of $X$. For a continuous map $f \colon X
\to \G_K(\F^n)$, the following conditions are equivalent:
\begin{conditions}
\item The map $f$ can be approximated by $\X$-regular maps.

\item The map $f$ is homotopic to an $\X$-regular map.

\item The pullback $\F$-vector bundle $f^*\gamma_K(\F^n)$ on $X$ admits
an $\X$-algebraic structure.
\end{conditions}
\end{corollary}

\begin{proof}
This is Corollary~\ref{cor-4-6} with $A = \varnothing$.
\end{proof}

Corollary~\ref{cor-4-7} makes it possible to interpret
Theorems~\ref{th-1-5} and~\ref{th-1-6} as approximation results.

\begin{theorem}\label{th-4-8}
There exists a stratification $\SC$ of $X$ such that any
stratified-regular map from $X$ into $\G_K(\F^n)$ can be approximated by
$\SC$-regular maps.
\end{theorem}

\begin{proof}
Let $\SC$ be a stratification of $X$ as in Theorem~\ref{th-1-5}. By
Proposition~\ref{prop-2-6}, if 
\begin{equation*}
f \colon X \to \G_K(\F^n)
\end{equation*}
is a
stratified-regular map, then the $\F$-vector bundle $f^*\gamma_K(\F^n)$
on $X$ is stratified-algebraic. It follows that $f^*\gamma_K(\F^n)$
admits an $\SC$-algebraic structure. According to
Corollary~\ref{cor-4-7}, $f$ can be approximated by $\SC$-regular maps.
\end{proof}

\begin{theorem}\label{th-4-9}
There exists a birational multiblowup $\pi \colon X' \to X$, with $X'$ a
nonsingular variety, such that for any stratified-regular map $f \colon
X \to \G_K(\F^n)$, the composite map $f \circ \pi \colon X' \to
\G_K(\F^n)$ can be approximated by regular maps.
\end{theorem}

\begin{proof}
Let $\pi \colon X' \to X$ be a birational multiblowup as in
Theorem~\ref{th-1-6}. By Proposition~\ref{prop-2-6}, if $f \colon X \to
\G_K(\F^n)$ is a stratified-regular map, the $\F$-vector bundle
$f^*\gamma_K(\F^n)$ on $X$ is stratified-algebraic. It follows that the
$\F$-vector bundle
\begin{equation*}
\pi^* ( f^* \gamma_K (\F^n) ) = (f \circ \pi)^* \gamma_K (\F^n)
\end{equation*}
on $X'$ admits an algebraic-structure. According to
Corollary~\ref{cor-4-7} (with $X=X'$ and $\X=\{X'\}$), the map $f \circ
\pi \colon X' \to \G_K(\F^n)$ can be approximated by regular maps.
\end{proof}

It follows from Corollary~\ref{cor-4-7} and Propositions~\ref{prop-2-6}
and~\ref{prop-3-4} that Theorem~\ref{th-4-9} is equivalent to
Theorem~\ref{th-1-6}.

\begin{theorem}\label{th-4-10}
Let $f \colon X \to \G_K(\F^n)$ be a continuous map whose restriction
$f|_A$ is stratified-regular. Then the following conditions are
equivalent:
\begin{conditions}
\item\label{th-4-10-a} Each neighborhood of $f$ in $\C(X, \G_K(\F^n))$ contains a
stratified-regular map 
\begin{equation*}
g \colon X \to \G_K(\F^n)
\end{equation*}
with $g|_A=f|_A$.

\item\label{th-4-10-b} The map $f$ is homotopic to a stratified-regular map $h \colon X
\to \G_K(\F^n)$ with $h|_A=f|_A$.

\item\label{th-4-10-c} The pullback $\F$-vector bundle $f^*\gamma_K(\F^n)$ on $X$ admits
a stratified-algebraic structure.
\end{conditions}
\end{theorem}

\begin{proof}
As in the case of Theorem~\ref{th-4-4}, it suffices to prove that
(\ref{th-4-10-c}) implies (\ref{th-4-10-a}).

Suppose that (\ref{th-4-10-c}) holds. Let $\TC$ be a stratification of
$X$ such that the $\F$-vector bundle $f^*\gamma_K(\F^n)$ admits a
$\TC$-algebraic structure, and let $\PC$ be a stratification of $A$ such
that the map $f|_A$ is $\PC$-regular. The collection
\begin{equation*}
\X \coloneqq \{ P \cap T \mid P \in \PC, T \in \TC \} \cup \{ T \cap (X
\setminus A) \mid T \in \TC \}
\end{equation*}
is a stratification of $X$, and $\A \coloneqq \{ S \in \X \mid S
\subseteq A \}$ is a stratification of $A$. By construction, $\X$ is a
refinement of $\TC$, and $\A$ is a refinement of $\PC$. Consequently,
the map $f|_A$ is $\A$-regular, and the $\F$-vector bundle
$f^*\gamma_K(\F^n)$ on $X$ admits an $\X$-algebraic structure. The proof
is complete in view of Corollary~\ref{cor-4-6}.
\end{proof}

In the results above, approximation by $\X$-regular or
stratified-regular maps is equiva\-lent to certain conditions on pullbacks
of the tautological vector bundle. It is often convenient (see
Sections~\ref{sec-5} and~\ref{sec-6}) to have these conditions expressed
in terms involving multiblowups.

\begin{proposition}\label{prop-4-11}
Let $f \colon X \to \G_K(\F^n)$ be a continuous map whose restriction
$f|_A$ is stratified-regular. Assume the existence of a multiblowup $\pi
\colon X' \to X$ over $A$ such that the pullback $\F$-vector bundle $(f
\circ \pi)^* \gamma_K(\F^n)$ on $X'$ admits an algebraic structure. Then
each neighborhood of $f$ in $\C(X, \G_K(\F^n))$ contains a
stratified-regular map $g \colon X \to \G_K(\F^n)$ such that $g|_A =
f|_A$ and the restriction $g|_{X \setminus A}$ is a regular map. In
particular, $f$ is homotopic to a stratified-regular map $h \colon X \to
\G_K(\F^n)$ such that $h|_A = f|_A$ and the restriction $h|_{X \setminus
A}$ is a regular map.
\end{proposition}

\begin{proof}
The map $f|_A$ is $\A$-regular for some stratification $\A$ of $A$. The
collection \linebreak $\A' \coloneqq \{ \pi^{-1}(S) \mid S \in \A \}$ is a
stratification of $A' \coloneqq \pi^{-1}(A)$, whereas $\X' \coloneqq \A'
\cup \{X' \setminus A'\}$ is a stratification of $X'$. Since the map $f
\circ \pi \colon X' \to \G_K(\F^n)$ is continuous and its restriction
$(f \circ \pi)|_{A'}$ is $\A'$-regular, according to
Corollary~\ref{cor-4-6} (with $X=X'$, $A=A'$, $\X=\X'$, $\A=\A'$,
$f=f \circ \pi$), there exists an $\X'$-regular map $\varphi \colon X'
\to \G_K(\F^n)$ that is arbitrarily close to $f \circ \pi$ and satisfies
$\varphi|_{A'} = (f \circ \pi)|_{A'}$. In particular, the restriction
$\varphi|_{X' \setminus A'}$ is a regular map. By assumption, the
restriction $\pi_A \colon X' \setminus A' \to X \setminus A$ of $\pi$ is
a biregular isomorphism. Since the map $\pi$ is proper, the map $g
\colon X \to \G_K(\F^n)$, defined by
\begin{equation*}
g(x)=
\begin{cases}
\varphi(\pi_A^{-1}(x)) & \textrm{for $x$ in $X \setminus A$} \\
f(x) & \textrm{for $x$ in $A$},
\end{cases}
\end{equation*}
is continuous. Moreover, $g$ is close to $f$, $g|_A = f|_A$, and the
restriction $g|_{X \setminus A}$ is a regular map. If $g$ is
sufficiently close to $f$, then $g$ is homotopic to $f$, and hence the
existence of $h$ follows.
\end{proof}

\begin{corollary}\label{cor-4-12}
Let $f \colon X \to \G_K(\F^n)$ be a continuous map whose restriction
$f|_A$ is stratified-regular. If the variety $X \setminus A$ is
nonsingular, then the following conditions are equivalent:
\begin{conditions}
\item\label{cor-4-12-a} Each neighborhood of $f$ in $\C(X, \G_K(\F^n))$
contains a stratified-regular map 
\begin{equation*}
g \colon X \to \G_K(\F^n)
\end{equation*}
such that
$g|_A=f|_A$ and the restriction $g|_{X \setminus A}$ is a regular map.

\item\label{cor-4-12-b} The map $f$ is homotopic to a stratified-regular
map $h \colon X \to \G_K(\F^n)$ such that $h|_A = f|_A$ and the
restriction $h|_{X \setminus A}$ is a regular map.

\item\label{cor-4-12-c} There exists a multiblowup $\pi \colon X' \to X$
over $A$ such that $X'$ is a nonsingular variety and the pullback
$\F$-vector bundle $(f \circ \pi)^* \gamma_K(\F^n)$ on $X'$ admits an
algebraic structure.
\end{conditions}
\end{corollary}

\begin{proof}
It is a standard fact that (\ref{cor-4-12-a}) implies
(\ref{cor-4-12-b}). In view of Proposition~\ref{prop-4-11},
(\ref{cor-4-12-c}) implies (\ref{cor-4-12-a}). It remains to prove that
(\ref{cor-4-12-b}) implies (\ref{cor-4-12-c}).

Suppose that (\ref{cor-4-12-b}) holds. Since the variety $X \setminus A$
is nonsingular, Hironaka's theorem on resolution of singularities
\cite{bib25, bib30} implies the existence of a multiblowup $\rho \colon
\tilde{X} \to X$ over $A$ such that $\tilde{X}$ is a nonsingular variety
and the set $\tilde{X} \setminus \rho^{-1}(A)$ is Zariski dense in
$\tilde{X}$. The map $h \circ \rho \colon \tilde{X} \to \G_K(\F^n)$ is
continuous and its restriction to $\tilde{X} \setminus \rho^{-1}(A)$ is
a regular map. Hence, according to Hironaka's theorem on resolution of
points of indeterminacy \cite{bib25, bib30}, there exists a multiblowup
$\sigma \colon X' \to \tilde{X}$ over $\rho^{-1}(A)$ such that the map
$h \circ \pi \colon X' \to \G_K(\F^n)$ is regular, where $\pi \coloneqq
\rho \circ \sigma$. Consequently, $(h \circ \pi)^* \gamma_K(\F^n)$ is an
algebraic $\F$-vector bundle on $X'$. The variety $X'$ is nonsingular,
and $\pi \colon X' \to X$ is a multiblowup over $A$. Since the maps $h
\circ \pi$ and $f \circ \pi$ are homotopic, the $\F$-vector bundles $(h
\circ \pi)^* \gamma_K(\F^n)$ and $(f \circ \pi)^*\gamma_K(\F^n)$ are
topologically isomorphic. Thus, (\ref{cor-4-12-b}) implies
(\ref{cor-4-12-c}).
\end{proof}

We also have the following modification of Proposition~\ref{prop-4-11}.

\begin{proposition}\label{prop-4-13}
Let $f \colon X \to \G_K(\F^n)$ be a continuous map whose restriction
$f|_A$ is stratified-regular. Assume the existence of a multiblowup $\pi
\colon X' \to X$ over $A$ such that the pullback $\F$-vector bundle $(f
\circ \pi)^* \gamma_K(\F^n)$ on $X'$ admits a stratified-algebraic
structure. Then each neighborhood of $f$ in $\C(X, \G_K(\F^n))$ contains
a stratified-regular map $g \colon X \to \G_K(\F^n)$ with $g|_A = f|_A$.
In particular, $f$ is homotopic to a stratified-regular map $h \colon X
\to \G_K(\F^n)$ with $h|_A = f|_A$.
\end{proposition}

\begin{proof}
The map $f|_A$ is $\A$-regular for some stratification $\A$ of $A$, and
the $\F$-vector bundle $(f \circ \pi)^* \gamma_K (\F^n)$ admits an
$\SC$-algebraic structure for some stratification $\SC$ of $X'$. The
collection $\PC \coloneqq \{ \pi^{-1}(T) \mid T \in \A \}$ is a
stratification of $A' \coloneqq \pi^{-1}(A)$, and the map $(f \circ
\pi)|_{A'}$ is $\PC$-regular. Moreover, the collection
\begin{equation*}
\SC' \coloneqq \{ P \cap S \mid P \in \PC, S \in \SC \} \cup \{ (X'
\setminus A') \cap S \mid S \in \SC \}
\end{equation*}
is a stratification of $X'$, and the collection
\begin{equation*}
\A' \coloneqq \{ S' \in \SC' \mid S' \subseteq A' \}
\end{equation*}
is a stratification of $A'$. By construction, $\A'$ is a refinement of
$\PC$, and $\SC'$ is a refinement of $\SC$. Consequently, the map $(f
\circ \pi)|_{A'}$ is $\A'$-regular, and the $\F$-vector bundle $(f \circ
\pi)^* \gamma_K(\F^n)$ on $X'$ admits an $\SC'$-algebraic structure.
According to Corollary~\ref{cor-4-6} (with $X=X'$, $A=A'$, $\X=\SC'$,
$\A=\A'$, $f=f \circ \pi$), there exists an $\SC'$-regular map $\varphi
\colon X' \to \G_K(\F^n)$ that is arbitrarily close to $f \circ \pi$ and
satisfies $\varphi|_{A'} = (f \circ \pi)|_{A'}$. The restriction $\pi_A
\colon X' \setminus A' \to X \setminus A$ of $\pi$ is a biregular
isomorphism. Since the map $\pi$ is proper, the map $g \colon X \to
\G_K(\F^n)$, defined by
\begin{equation*}
g(x)=
\begin{cases}
\varphi(\pi_A^{-1}(x)) & \textrm{for $x$ in $X\setminus A$}\\
f(x) & \textrm{for $x$ in $A$,}
\end{cases}
\end{equation*}
is continuous. By construction, $g$ is close to $f$ and $g|_A = f|_A$.
Moreover,
\begin{equation*}
\X \coloneqq \A \cup \{ \pi_A(S') \mid S' \in \SC', S' \subseteq X'
\setminus A' \}
\end{equation*}
is a stratification of $X$, and the map $g$ is $\X$-regular. If $g$ is
sufficiently close to $f$, then $g$ is homotopic to $f$, and hence the
existence of $h$ follows.
\end{proof}

\section[Topological versus stratified-algebraic vector
bundles]{\texorpdfstring{Topological versus stratified-algebraic\\ vector
bundles}{Topological versus stratified-algebraic vector
bundles}}\label{sec-5}

Throughout this section, $X$ denotes a compact real algebraic variety.
By making use of the notion of filtration, introduced in
Section~\ref{sec-2}, we demonstrate how the behavior of a vector bundle
on $X$ can be deduced from the behavior of its restrictions to Zariski
closed subvarieties of $X$. This is crucial for, in particular, the
proofs of Theorems~\ref{th-1-7}, \ref{th-1-8} and~\ref{th-1-9}, given in
the next section.

\begin{proposition}\label{prop-5-1}
Let $\xi$ be a topological $\F$-vector bundle on $X$ and let 
\begin{equation*}
\FC = (X_{-1}, X_0, \ldots, X_m)
\end{equation*}
be a filtration of $X$. Assume that for each
$i = 0, \ldots, m$, there exists a multiblowup $\pi_i \colon X'_i \to
X_i$ over $X_{i-1}$ such that the pullback $\F$-vector bundle $\pi_i^*
(\xi|_{X_i})$ on $X'_i$ admits an algebraic structure. Then $\xi$ admits
an $\overline{\FC}$-algebraic structure.
\end{proposition}

\begin{proof}
Since $X$ is compact, we may assume that $\xi$ is of the form $\xi =
f^*\gamma_K(\F^n)$ for some multi-Grassmannian $\G_K(\F^n)$ and
continuous map $f \colon X \to \G_K(\F^n)$. It suffices to prove that $f$
is homotopic to an $\overline{\FC}$-regular map, cf.
Corollary~\ref{cor-4-7}.

\begin{assertion}
For each $i=0,\ldots,m$, there exists a continuous map $g_i \colon X \to
\G_K(\F^n)$ homotopic to $f$ and such that the restriction of $g_i$ to
$X_j \setminus X_{j-1}$ is a regular map for $0 \leq j \leq i$.
\end{assertion}

If the Assertion holds, then the map $g_m$ is $\overline{\FC}$-regular
and the proof is complete. We prove the Assertion by induction on $i$.
Recall that $(X, X_i)$ is a polyhedral pair for \linebreak $0 \leq i \leq m$, cf.
\cite[Corollary~9.3.7]{bib9}. Since $X'_0=X_0$ and $\pi_0 \colon X'_0
\to X_0$ is the identity map, the $\F$-vector bundle $\xi|_{X_0}$ on
$X_0$ admits an algebraic structure. According to
Proposition~\ref{prop-4-11} (with $X=X_0$, $A=\varnothing$), the map
$f|_{X_0}$ is homotopic to a regular map $\varphi \colon X_0 \to
\G_K(\F^n)$. The homotopy extension theorem \cite[p.~118,
Corollary~5]{bib42} implies the existence of a continuous map $g_0
\colon X \to \G_K(\F^n)$ homotopic to $f$ and satisfying $g_0|_{X_0} =
\varphi$. We now suppose that $0 \leq i \leq m-1$ and the map $g_i$
satisfying the required conditions has already been constructed. In
particular, the map $g_i|_{X_i}$ is stratified-regular. Since the maps
$g_i|_{X_{i+1}}$ and $f|_{X_{i+1}}$ are homotopic, the $\F$-vector
bundles $(g_i|_{X_{i+1}})^* \gamma_K (\F^n)$ and $(f|_{X_{i+1}})^*
\gamma_K(\F^n) = \xi|_{X_{i+1}}$ are topologically isomorphic.
Consequently, the $\F$-vector bundle
\begin{equation*}
\pi_{i+1}^* ( (g_i|_{X_{i+1}})^* \gamma_K (\F^n) ) = ( (g_i|_{X_{i+1}})
\circ \pi_{i+1} )^* \gamma_K(\F^n)
\end{equation*}
on $X'_{i+1}$ admits an algebraic structure, being topologically
isomorphic to $\pi_{i+1}^*(\xi|_{X_{i+1}})$. According to
Proposition~\ref{prop-4-11} (with $X = X_{i+1}$, $A = X_i$), the map
$g_i|_{X_{i+1}}$ is homotopic to a continuous map $\psi \colon X_{i+1}
\to \G_K(\F^n)$ such that $\psi|_{X_i} = g_i|_{X_i}$ and the restriction
of $\psi$ to $X_{i+1} \setminus X_i$ is a regular map. By the homotopy
extension theorem, there exists a continuous map $g_{i+1} \colon X \to
\G_K(\F^n)$ homotopic to $g_i$ (hence homotopic to $f$) and satisfying
$g_{i+1}|_{X_{i+1}} = \psi$. This completes the proof of the Assertion.
\end{proof}

\begin{theorem}\label{th-5-2}
Let $\xi$ be a topological $\F$-vector bundle on $X$ and let $\FC =
(X_{-1}, X_0, \ldots, X_m)$ be a filtration of $X$. If the
stratification $\overline{\FC}$ is nonsingular, then the following
conditions are equivalent:
\begin{conditions}
\item\label{th-5-2-a} The $\F$-vector bundle $\xi$ admits an
$\overline{\FC}$-algebraic structure.

\item\label{th-5-2-b} For each $i=0,\ldots,m$, there exists a
multiblowup $\pi_i \colon X'_i \to X_i$ over $X_{i-1}$ such that $X'_i$
is a nonsingular variety and the pullback $\F$-vector bundle
$\pi_i^*(\xi|_{X_i})$ on $X'_i$ admits an algebraic structure.
\end{conditions}
\end{theorem}

\begin{proof}
Suppose that (\ref{th-5-2-a}) holds. In view of
Propositions~\ref{prop-2-2} and~\ref{prop-3-3}, we may assume that $\xi$ is of the form $\xi
= f^*\gamma_K(\F^n)$ for some multi-Grassmannian $\G_K(\F^n)$ and
$\overline{\FC}$-regular map $f \colon X \to \G_K(\F^n)$. Set $f_i
\coloneqq f|_{X_i}$ for $0 \leq i \leq m$. The map $f_i|_{X_{i-1}}$ is
stratified-regular, and the restriction of $f_i$ to $X_i \setminus
X_{i-1}$ is a regular map. The existence of $\pi_i$ as in
(\ref{th-5-2-b}) follows from Corollary~\ref{cor-4-12} (with $X=X_i$,
$A=X_{i-1}$). Thus, (\ref{th-5-2-a}) implies (\ref{th-5-2-b}). According
to Proposition~\ref{prop-5-1},  (\ref{th-5-2-b}) implies
(\ref{th-5-2-a}).
\end{proof}

Theorem~\ref{th-5-2} implies the following characterization of
topological $\F$-vector bundles on $X$ admitting a stratified-algebraic
structure.

\begin{corollary}\label{cor-5-3}
For a topological $\F$-vector bundle $\xi$ on $X$, the following
conditions are equivalent:
\begin{conditions}
\item The $\F$-vector bundle $\xi$ admits a stratified-algebraic
structure.

\item There exists a filtration $\FC = (X_{-1}, X_0, \ldots, X_m)$ of
$X$, with $\overline{\FC}$ a nonsingular strati\-fi\-ca\-tion, and for each $i
= 0, \ldots, m$, there exists a multiblowup $\pi_i \colon X'_i \to X_i$
over $X_{i-1}$ such that $X'_i$ is a nonsingular variety and the pullback
$\F$-vector bundle $\pi_i^*(\xi|_{X_i})$ on $X_i'$ admits an
algebraic structure.
\end{conditions}
\end{corollary}

\begin{proof}
It suffices to combine Proposition~\ref{prop-3-5} and
Theorem~\ref{th-5-2}.
\end{proof}

The following result will also prove to be useful.

\begin{theorem}\label{th-5-4}
For a topological $\F$-vector bundle $\xi$ on $X$, the following
conditions are equivalent:
\begin{conditions}
\item\label{th-5-4-a} The $\F$-vector bundle $\xi$ admits a
stratified-algebraic structure.

\item\label{th-5-4-b} There exists a filtration $\FC = (X_{-1}, X_0,
\ldots, X_m)$ of $X$, and for each $i=0,\ldots,m$, there exists a
multiblowup $\pi_i \colon X'_i \to X_i$ over $X_{i-1}$ such that the
pullback $\F$-vector bundle $\pi_i^*(\xi|_{X_i})$ on $X'_i$ admits a
stratified-algebraic structure.
\end{conditions}
\end{theorem}

\begin{proof}
By Corollary~\ref{cor-5-3}, (\ref{th-5-4-a}) implies (\ref{th-5-4-b}).
It suffices to prove that (\ref{th-5-4-b}) implies (\ref{th-5-4-a}). The
proof is similar to the proof of Proposition~\ref{prop-5-1}.

Suppose that (\ref{th-5-4-b}) holds. Since $X$ is compact, we may assume
that $\xi$ is of the form
\begin{equation*}
\xi = f^* \gamma_K(\F^n)
\end{equation*}
for some
multi-Grassmannian $\G_K(\F^n)$ and continuous map $f \colon X \to
\G_K(\F^n)$. It remains to prove that $f$ is homotopic to a
stratified-regular map, cf. Theorem~\ref{th-4-10} (with $A =
\varnothing$).

\begin{assertion}
For each $i=0, \ldots, m$, there exists a continuous map $g_i \colon X
\to \G_K(\F^n)$ homotopic to $f$ and such that the restriction
$g_i|_{X_i}$ is a stratified-regular map.
\end{assertion}

If the Assertion holds, then the map $g_m$ is stratified-regular and the
proof is complete. We prove the Assertion by induction on $i$. Since
$X'_0=X_0$ and $\pi_0 \colon X'_0 \to X_0$ is the identity map, the
$\F$-vector bundle $\xi|_{X_0}$  on $X_0$ admits a stratified-algebraic
structure. According to Proposition~\ref{prop-4-13} (with $X=X_0$ and
$A=\varnothing$), the map $f|_{X_0}$ is homotopic to a
stratified-regular map $\varphi \colon X_0 \to \G_K(\F^n)$. The homotopy
extension theorem implies the existence of a continuous map $g_0 \colon
X \to \G_K(\F^n)$ homotopic to $f$ and satisfying $g_0|_{X_0} = \varphi$.
We now suppose that $0 \leq i \leq m-1$ and the map $g_i$ satisfying the
required conditions has already been constructed. In particular, the map
$g_i|_{X_i}$ is
stratified-regular. Since the maps $g_i|_{X_{i+1}}$ and $f|_{X_{i+1}}$
are homotopic, the $\F$-vector bundles $(g_i|_{X_{i+1}})^* \gamma_K
(\F^n)$ and $(f|_{X_{i+1}})^* \gamma_K (\F^n) = \xi|_{X_{i+1}}$ are
topologically isomorphic. Consequently, the $\F$-vector bundle
\begin{equation*}
\pi_{i+1}^* ( ( g_i|_{X_{i+1}} )^* \gamma_K (\F^n) ) = ( (g_i|_{X_{i+1}})
\circ \pi_{i+1} )^* \gamma_K (\F^n)
\end{equation*}
on $X'_{i+1}$ admits a stratified-algebraic structure, being
topologically isomorphic to \linebreak $\pi_{i+1}^* (\xi|_{X_{i+1}})$. According to
Proposition~\ref{prop-4-13} (with $X = X_{i+1}$, $A=X_i$), the map
$g_i|_{X_{i+1}}$ is homotopic to a stratified-regular map $\psi \colon
X_{i+1} \to \G_K(\F^n)$ satisfying $\psi|_{X_i} = g_i|_{X_i}$. By the
homotopy extension theorem, there exists a continuous map $g_{i+1}
\colon X \to \G_K(\F^n)$ homotopic to $g_i$ (hence homotopic to $f$) and
satisfying $g_{i+1}|_{X_{i+1}} = \psi$. This completes the proof of the
Assertion.
\end{proof}

\section{Blowups and vector bundles}\label{sec-6}
We need certain constructions involving smooth (of class $\C^{\infty}$)
manifolds. All smooth manifolds are assumed to be paracompact and
without boundary. Submanifolds are supposed to be closed subsets of the
ambient manifold. For any smooth manifold $M$, the total space of the
tangent bundle to $M$ is denoted by $TM$. If $Z$ is a smooth submanifold
of $M$, then $N_ZM$ denotes the total space of the normal bundle to $Z$
in $M$. Thus for each point $x$ in $Z$, the fiber $(N_ZM)_x$ is equal to
$TM_x/TZ_x$.

Let $\theta$ be a smooth $\R$-vector bundle of rank $r$ on
$M$, and let $s \colon M \to \theta$ be a smooth section transverse to
the zero section. Then the zero locus $Z(s)$ of $s$ is a smooth
submanifold of $M$, which is either empty or of codimension $r$. In
order to have a convenient reference, we record the following well known
fact.

\begin{lemma}\label{lem-6-1}
With notation as above, the restriction $\theta|_{Z(s)}$ is smoothly
isomorphic to the normal bundle to $Z(s)$ in $M$.
\end{lemma}

\begin{proof}
Let $E$ denote the total space of $\theta$, and let $Z \coloneqq Z(s)$.
We regard $M$ as a smooth submanifold of $E$, identifying it with the
image by the zero section. We identify $\theta$ with the normal bundle
to $M$ in $E$. Since $Z= s^{-1}(M)$ and $s$ is transverse to $M$ in $E$,
for each point $z$ in $Z$, the differential $ds_z \colon TM_z \to TE_z$
induces an $\R$-linear isomorphism
\begin{equation*}
(N_ZM)_z \to (N_ME)_z = E_z.
\end{equation*}
The proof is complete.
\end{proof}

For any smooth submanifold $D$ of $M$ of codimension $1$, there exists a
smooth $\R$-line bundle $\lambda(D)$ on $M$ with a smooth section $s_D
\colon M \to \lambda(D)$ such that $Z(s_D) = D$ and $s_D$ is transverse
to the zero section. Both $\lambda(D)$ and $s_D$ are constructed as in
algebraic geometry, by regarding $D$ as a $\C^{\infty}$ divisor on $M$.
Explicitly, if $\{U_i\}$ is an open cover of $M$ and $\{ f_i \colon U_i
\to \R \}$ is a collection of smooth local equations for $D$, then
$\lambda(D)$ is determined by the transition functions $g_{ij} \coloneqq f_i
/ f_j \colon U_i \cap U_j \to \R \setminus \{0\}$, whereas $s_D$ is
determined by the functions $f_i$. On the other hand, if $\lambda(D)$ and
$s_D$ are already given, then $s_D$ gives rise to smooth local equations
for $D$, and hence $\lambda(D)$ is uniquely determined up to smooth
isomorphism. It is also convenient to set $\lambda(\varnothing)
\coloneqq \varepsilon^1_M(\R)$.

For any smooth submanifold $Z$ of $M$, we denote by
\begin{equation*}
\pi(M,Z) \colon B(M,Z) \to M
\end{equation*}
the blowup of $M$ with center $Z$ (cf. \cite{bib2} for basic properties
of this construction). Recall that as a point set, $B(M,Z)$ is the union
of $M \setminus Z$ and the total space $\PB(N_ZM)$ of the projective
bundle associated with the normal bundle to $Z$ in $M$. The map
$\pi(M,Z)$ is the identity on $M \setminus Z$ and the bundle projection
$\PB(N_ZM) \to Z$ on $\PB(N_ZM)$. On $B(M,Z)$ there is a natural smooth
manifold structure, and $\pi(M,Z)$ is a smooth map. If $Z \neq
\varnothing$ and $\codim_M Z \geq 1$, then $\pi(M,Z)^{-1}(Z) =
\PB(N_ZM)$ is a smooth submanifold of $B(M,Z)$ of codimension $1$.

\begin{proposition}\label{prop-6-2}
Let $M$ be a smooth manifold. Let $\theta$ be a smooth $\R$-vector
bundle of positive rank on $M$, and let $s \colon M \to \theta$ be a
smooth section transverse to the zero section. If $Z \coloneqq Z(s)$ and
$\pi(M,Z) \colon B(M,Z) \to M$ is the blowup of $M$ with center $Z$,
then the pullback $\R$-vector bundle $\pi(M,Z)^*\theta$ on $B(M,Z)$
contains a smooth $\R$-line subbundle smoothly isomorphic to
$\lambda(D)$, where $D \coloneqq \pi(M,Z)^{-1}(Z)$.
\end{proposition}

\begin{proof}
Let $E$ be the total space of $\xi$ and let $p \colon E \to
M$ be the bundle projection. We regard $M$ as a smooth submanifold of
$E$ and identify the normal bundle to $M$ in $E$ with $\theta$. Thus as
a point set $B(E,M)$ is the union of $E \setminus M$ and $\PB(E)$, while
\begin{equation*}
\pi(E,M) \colon B(E,M) \to E
\end{equation*}
is the identity on $E \setminus M$ and
the bundle projection $\PB(E) \to M$ on $\PB(E)$. Here $\PB(E)$ is the
total space of the projective bundle on $M$ associated with $\theta$.
The pullback smooth $\R$-vector bundle $(p \circ \pi(E,M))^*\theta$ on
$B(E,M)$ contains a smooth $\R$-line subbundle $\lambda$ defined as
follows. The fiber of $\lambda$ over a point $e$ in $E \setminus M$ is
the line $\{e\} \times (\R e)$, and the restriction $\lambda|_{\PB(E)}$
is the tautological $\R$-line bundle on $\PB(E)$.

Since $s$ is transverse to $M$ in $E$, for each point $z$ in $Z$, the
differential
\begin{equation*}
ds_z \colon TM_z \to (TE)_z
\end{equation*}
induces an $\R$-linear
isomorphism
\begin{equation*}
\bar{d}s_z \colon (N_ZM)_z \to (N_ME)_z = E_z
\end{equation*}
between the fibres over $z$ of the normal bundle to $Z$ in $M$ and the
normal bundle to $M$ in $E$. Define $\bar{s} \colon B(M,Z) \to B(E,M)$
by $\bar{s}(x)=x$ for all $x$ in $M \setminus Z$ and $\bar{s}(l) =
\bar{d}s_z(l)$ for all $l$ in $\PB(N_ZM)_z$ (thus, $\bar{s}(l)$ is in the
fiber $\PB(E)_z$). By construction, $\bar{s}$ is a smooth map satisfying
\begin{equation*}
p \circ \pi(E,M) \circ \bar{s} = \pi(M,Z).
\end{equation*}
Hence $\bar{s}^*\lambda$ is a smooth $\R$-line subbundle of
\begin{equation*}
\bar{s}^*( (p \circ \pi(E,M) )^* \theta ) = (p \circ \pi(E,M) \circ
\bar{s})^* \theta = \pi(M,Z)^*\theta.
\end{equation*}
It remains to prove that the $\R$-line bundles $\bar{s}^*\lambda$ and
$\lambda(D)$ are smoothly isomorphic. To this end, it suffices to
construct a smooth section $u \colon B(M,Z) \to \bar{s}^*\lambda$ that
is transverse to the zero section and satisfies $Z(u) = D$. Such a
section can be obtained as follows. The smooth section $v \colon B(E,M)
\to \lambda$, defined by $v(e)=(e,e)$ for all $e$ in $E \setminus M$ and
$v|_{\PB(E)}=0$, is transverse to the zero section and satisfies $Z(v) =
\PB(E)$. On the other hand, the smooth map $\bar{s} \colon B(M,Z) \to
B(E,M)$ is transverse to $\PB(E)$ in $B(E,M)$ and $\bar{s}^{-1}(\PB(E))
= \pi(M,Z)^{-1}(Z) = D$. Consequently, the smooth section 
\begin{equation*}
u \coloneqq \bar{s}^*v \colon B(M,Z) \to \bar{s}^*\lambda
\end{equation*}
satisfies the required conditions.
\end{proof}

Let $X$ be a nonsingular real algebraic variety. For any nonsingular
Zariski closed subvariety $D$ of $X$ of codimension $1$, there exists an
algebraic $\R$-line bundle $\lambda(D)$ on $X$ with an algebraic section
$s_D \colon X \to \lambda(D)$ such that $Z(s_D)=D$ and $s_D$ is
transverse to the zero section. The $\R$-line bundle $\lambda(D)$ is
uniquely determined up to algebraic isomorphism.

Recall that any algebraic $\R$-vector bundle $\theta$ on $X$ has an
algebraic section transverse to the zero section. Indeed, $\theta$ is
generated by finitely many global algebraic sections $s_1, \ldots, s_n$,
cf. \cite[Theorem~12.1.7]{bib9} or Proposition~\ref{prop-3-7}. According
to the transversality theorem, for a general point $(t_1, \ldots, t_n)$
in $\R^n$, the algebraic section $t_1 s_1 + \cdots + t_n s_n$ is
transverse to the zero section.

If $Z$ is a nonsingular Zariski closed subvariety of $X$, the
$\C^{\infty}$ blowup
\begin{equation*}
\pi(X,Z) \colon B(X,Z) \to X
\end{equation*}
can be identified
with the algebraic blowup, cf. \cite[Lemma~2.5.5]{bib2}.

The following consequence of Proposition~\ref{prop-6-2} will play an
important role.

\begin{corollary}\label{cor-6-3}
Let $X$ be a nonsingular real algebraic variety. Let $\theta$ be a
smooth $\R$-vector bundle of positive rank on $X$, and let $s \colon X
\to \theta$ be a smooth section transverse to the zero section. Assume
that $Z \coloneqq Z(s)$ is a nonsingular Zariski closed subvariety of
$X$. If $\pi \colon X' \to X$ is the blowup with center $Z$, then the
smooth $\R$-vector bundle $\pi^*\theta$ on $X'$ contains a smooth
$\R$-line subbundle $\lambda$ which is smoothly isomorphic to an
algebraic $\R$-line bundle on $X'$. In particular, $\lambda$ admits an
algebraic structure.
\end{corollary}

\begin{proof}
Let $D \coloneqq \pi^{-1}(Z)$. Either $D = \varnothing$ or $D$ is a
nonsingular Zariski closed subvariety of $X'$ of codimension $1$.
Consequently, $\lambda(D)$ is an algebraic $\R$-line bundle on $X'$. The
proof is complete in view of Proposition~\ref{prop-6-2}.
\end{proof}

There is also an algebraic-geometric counterpart of
Proposition~\ref{prop-6-2}.

\begin{proposition}\label{prop-6-4}
Let $X$ be a nonsingular real algebraic variety. Let $\theta$ be an
algebraic $\R$-vector bundle of positive rank on $X$, and let $s \colon
X \to \theta$ be an algebraic section transverse to the zero section. If
$\pi \colon X' \to X$ is the blowup with center $Z \coloneqq Z(s)$, then
the pullback $\R$-vector bundle $\pi^*\theta$ on $X'$ contains an
algebraic $\R$-line subbundle which is algebraically isomorphic to
$\lambda(D)$, where $D \coloneqq \pi^{-1}(Z)$.
\end{proposition}

\begin{proof}
The proof of Proposition~\ref{prop-6-2} can be carried over to the
algebraic-geometric setting.
\end{proof}

For any $\R$-vector bundle $\theta$, we denote by $\CB \otimes \theta$ and $\HB
\otimes \theta$ the complexification and the quaternionification of
$\theta$. In order to have the uniform notation, we also set $\R \otimes
\theta \coloneqq \theta$. Thus, $\F \otimes \theta$ is an $\F$-vector
bundle.

The following lemma leads directly to the proof of Theorem~\ref{th-1-7}.

\begin{lemma}\label{lem-6-5}
Let $X$ be a compact nonsingular real algebraic variety. If $\xi$ is a
topological $\F$-vector bundle on $X$ such that the $\R$-vector bundle
$\xi_{\R}$ admits an algebraic structure, then $\xi$ admits a
stratified-algebraic structure as an $\F$-vector bundle.
\end{lemma}

\begin{proof}
We may assume without loss of generality that the variety $X$ is
irreducible. Then $\xi$ is of constant rank since $\xi_{\R}$ admits an
algebraic structure. We use induction on $r \coloneqq \rank \xi$.
Obviously, the assertion holds if $r=0$. Suppose that $r \geq 1$ and the
assertion holds for all topological $\F$-vector bundles
of rank at most $r-1$, defined on compact nonsingular real algebraic
varieties. Let $\theta$ be an algebraic $\R$-vector bundle on $X$ that
is topologically isomorphic to $\xi_{\R}$. Let $s \colon X \to \theta$
be an algebraic section transverse to the zero section and let $Z
\coloneqq Z(s)$. Then either $Z = \varnothing$ or $Z$ is a nonsingular
Zariski closed subvariety of $X$ of codimension $d(\F)r$. Let $\pi
\colon X' \to X$ be the blowup of $X$ with center $Z$. According to
Proposition~\ref{prop-6-4}, the pullback $\R$-vector bundle
$\pi^*\theta$ on $X'$ contains an algebraic $\R$-line subbundle
$\lambda$. Hence, $\pi^*(\xi_{\R}) = (\pi^*\xi)_{\R}$ contains a
topological $\R$-line subbundle $\mu$ that is topologically isomorphic
to $\lambda$. Since $\pi^*\xi$ is an $\F$-vector bundle, it contains a
topological $\F$-line subbundle $\lambda'$ isomorphic to $\F \otimes
\mu$. By construction, $\lambda'$ is topologically isomorphic to the
algebraic $\F$-line bundle $\F \otimes \lambda$. In particular, the
$\F$-line bundle $\lambda'$ admits an algebraic structure, and the
$\F$-vector bundle $\pi^*\xi$ can be expressed as
\begin{equation*}
\pi^* \xi = \lambda' \oplus \xi',
\end{equation*}
where $\xi'$ is a topological $\F$-vector bundle on $X'$ of rank $r-1$.
Note that the $\R$-vector bundle
\begin{equation*}
\pi^*(\xi_{\R}) = (\pi^*\xi)_{\R} = \lambda'_{\R} \oplus \xi'_{\R}
\end{equation*}
admits an algebraic structure. Since $\lambda'_{\R}$ admits an algebraic
structure, there exists an algebraic $\R$-vector bundle $\eta$ on $X'$
such that the direct sum $\eta \oplus \lambda'_{\R}$ is topologically
isomorphic to a trivial algebraic $\R$-vector bundle $\varepsilon$ on
$X'$, cf. \cite[Theorem~12.1.7]{bib9} or Proposition~\ref{prop-3-7}.
Consequently, the $\R$-vector bundle $\varepsilon \oplus \xi'_{\R}$
admits an algebraic structure, being topo\-logi\-cal\-ly isomorphic to $\eta
\oplus \pi^*(\xi_{\R})$. According to \cite[Proposition~12.3.5]{bib9} or
Corollary~\ref{cor-3-13}, the $\R$-vector bundle $\xi'_{\R}$ admits an
algebraic structure. Now, by the induction hypothesis, the $\F$-vector
bundle $\xi'$ on $X'$ admits a stratified-algebraic structure.
Consequently, the $\F$-vector bundle $\pi^*\xi$ on $X'$ admits a
stratified-algebraic structure.

Consider the topological $\F$-vector bundle $\xi|_Z$ on $Z$. Since the
$\R$-vector bundle
\begin{equation*}
(\xi|_Z)_{\R} = (\xi_{\R})|_Z
\end{equation*}
is topologically
isomorphic to $\theta|_Z$, the construction above can be repeated with
$\xi|_Z$ substituted for $\xi$. By continuing this process, we obtain a
filtration $\FC = (X_{-1}, X_0, \ldots, X_m)$ of $X$ such that for each
$i=0, \ldots, m$, the following two conditions are satisfied:
\begin{starconditions}
\item $X_i$ is a nonsingular Zariski closed subvariety of $X$;

\item If $\pi_i \colon X'_i \to X_i$ is the blowup of $X_i$ with center
$X_{i-1}$, then the pullback $\F$-vector bundle $\pi_i^* (\xi|_{X_i})$
on $X'_i$ admits a stratified-algebraic structure.
\end{starconditions}
According to Theorem~\ref{th-5-4}, the topological $\F$-vector bundle
$\xi$ admits a stratified-algebraic structure, as required.
\end{proof}

\begin{proof}[Proof of Theorem~\ref{th-1-7}]
Let $\xi$ be a topological $\F$-vector bundle on $X$. Obviously, if
$\xi$ admits a stratified-algebraic structure, then so does $\xi_{\R}$.

Suppose now that $\xi_{\R}$ admits a stratified-algebraic structure. We
prove by induction on $\dim X$ that $\xi$ admits a stratified-algebraic
structure. In view of Theorem~\ref{th-1-6}, there exists a birational
multiblowup $\pi \colon X' \to X$ such that $X'$ is a nonsingular
variety and the $\R$-vector bundle $\pi^*(\xi_{\R}) = (\pi^*\xi)_{\R}$ on
$X'$ admits an algebraic structure. By Lemma~\ref{lem-6-5}, the
$\F$-vector bundle $\pi^*\xi$ on $X'$ admits a stratified-algebraic
structure. Let $A$ be a Zariski closed
subvariety such that $\dim A < \dim X$ and $\pi$ is a multiblowup over
$A$. The $\R$-vector bundle $(\xi|_A)_{\R} = (\xi_{\R})|_A$ on $A$ admits a
stratified-algebraic structure. Hence, by the induction hypothesis, the
$\F$-vector bundle $\xi|_A$ on $A$ admits a stratified-algebraic
structure. Note that $\FC = (\varnothing, A, X)$ is a filtration of $X$.
According to Theorem~\ref{th-5-4}, the $\F$-vector bundle $\xi$ admits a
stratified-algebraic structure.
\end{proof}

We show next that Theorem~\ref{th-1-7} can be restated as an
approximation result. If $K$ is a nonempty finite collection of
nonnegative integers, we set $d(\F)K \coloneqq \{ d(\F)k \mid k \in K
\}$. Any $\F$-vector subspace $V$ of $\F^n$ can be regarded as an
$\R$-vector subspace of $\R^{d(\F)n}$, which is indicated by $V_{\R}$.
The correspondence $V \to V_{\R}$ gives rise to a regular map
\begin{equation*}
i \colon \G_K(\F^n) \to \G_{d(\F)K}(\R^{d(\F)n}).
\end{equation*}

\begin{theorem}\label{th-6-6}
Let $X$ be a compact real algebraic variety. For a continuous map 
\begin{equation*}
f \colon X \to \G_K(\F^n),
\end{equation*}
the following conditions are equivalent:
\begin{conditions}
\item\label{th-6-6-a} The map $f \colon X \to \G_K(\F^n)$ can be
approximated by stratified-regular maps.

\item\label{th-6-6-b} The map $i \circ f \colon X \to \G_{d(\F)K}
(\R^{d(\F)n})$ can be approximated by stratified-regular maps.
\end{conditions}
\end{theorem}

\begin{proof}
Condition (\ref{th-6-6-a}) implies (\ref{th-6-6-b}), the map $i$ being
regular. Since
\begin{equation*}
(\gamma_K(\F^n))_{\R} = i^*( \gamma_{d(\F)K} (\R^{d(\F)n}) ),
\end{equation*}
we get
\begin{equation*}
( f^* \gamma_K (\F^n) )_{\R} = f^* ( (\gamma_K(\F^n))_{\R} ) = (i \circ
f)^* \gamma_{d(\F)K} (\R^{d(\F)n}).
\end{equation*}
Hence, according to Theorem~\ref{th-4-10} (with $\F=\R$, $f = i \circ
f$, $A = \varnothing$), condition (\ref{th-6-6-b}) implies that the
$\R$-vector bundle $(f^*\gamma_K(\F^n))_{\R}$ admits a
stratified-algebraic structure. In view of Theorem~\ref{th-1-7}, the
$\F$-vector bundle $f^*\gamma_K(\F^n)$ admits a stratified-algebraic
structure. Making again use of Theorem~\ref{th-4-10} (with $A =
\varnothing$), we conclude that (\ref{th-6-6-b}) implies
(\ref{th-6-6-a}).
\end{proof}

In view of Theorem~\ref{th-4-10}, one readily sees that
Theorem~\ref{th-6-6} implies Theorem~\ref{th-1-7}.

Proofs of Theorems~\ref{th-1-8} and~\ref{th-1-9} require further
preparation.

\begin{lemma}\label{lem-6-7}
Let $X$ be a compact nonsingular real algebraic variety and let $\xi$ be
an adapted smooth $\F$-vector bundle on $X$. Then there exists a smooth
section $s \colon X \to \xi$ transverse to the zero section and such
that $Z(s)$ is a nonsingular Zariski locally closed subvariety of $X$.
\end{lemma}

\begin{proof}
According to the definition of an adapted vector bundle, we can choose a
smooth section $u \colon X \to \xi$ transverse to the zero section and
such that its zero locus $Z(u)$ is smoothly isotopic to a nonsingular
Zariski locally closed subvariety $Z$ of $X$. Since any isotopy can be
extended to a diffeotopy \cite[p.~180, Theorem~1.3]{bib27}, there exists
a smooth diffeomorphism $h \colon X \to X$ that is homotopic to the
identity map $\one_X$ of $X$ and satisfies $h(Z) = Z(u)$. The
pullback section $h^*u \colon X \to h^*\xi$ is transverse to the zero
section and $Z(h^*u)=Z$. The proof is complete since the $\F$-vector
bundle $h^*\xi$ is smoothly isomorphic to $\xi$, the maps $h$ and
$\one_X$ being homotopic.
\end{proof}

\begin{lemma}\label{lem-6-8}
Let $X$ be a nonsingular real algebraic variety. Let $\xi$ be a smooth
$\F$-vector bundle on $X$, and let $s \colon X \to \xi$ be a smooth
section transverse to the zero section and such that $Z \coloneqq Z(s)$
is a nonsingular Zariski locally closed subvariety of $X$. Let $V$ be
the Zariski closure of $Z$ in $X$ and let $W \coloneqq V \setminus Z$.
Then there exists a multiblowup $\rho \colon \tilde{X} \to X$ over $W$
such that $\tilde{Z} \coloneqq \rho^{-1}(Z)$ is a nonsingular Zariski
closed subvariety of $\tilde{X}$, the pullback section $\rho^*s \colon
\tilde{X} \to \rho^*\xi$ is transverse to the zero section, and
$Z(\rho^*s) = \tilde{Z}$. Furthermore, if the variety $X$ is compact and
the $\F$-vector bundles $\xi|_W$ on $W$ and $\rho^*\xi$ on $\tilde{X}$
admit stratified-algebraic structures, then $\xi$ admits a
stratified-algebraic structure.
\end{lemma}

\begin{proof}
Since $Z$ is closed in the Euclidean topology, it readily follows that
$W$ is the singular locus of $V$. According to Hironaka's theorem on
resolution of singularities \cite{bib25, bib30}, there exists a
multiblowup $\rho \colon \tilde{X} \to X$ over $W$ such that the Zariski
closure $\tilde{Z}$ of $\rho^{-1}(Z)$ in $\tilde{X}$ is a nonsingular
subvariety. In the case under consideration, $\tilde{Z} = \rho^{-1}(Z)$,
the set $Z$ being closed in the Euclidean topology. Since the
restriction $\rho_W \colon \tilde{X} \setminus \rho^{-1}(W) \to X
\setminus W$ of $\rho$ is a biregular isomorphism, the pullback section
$\rho^*s \colon \tilde{X} \to \rho^*\xi$ is transverse to the zero
section. By construction, $Z(\rho^*s) = \tilde{Z}$. Note that $\FC =
(\varnothing, W, X)$ is a filtration of $X$. The last assertion in
Lemma~\ref{lem-6-8} follows from Theorem~\ref{th-5-4}.
\end{proof}

\begin{proof}[Proof of Theorem~\ref{th-1-8}]
Let $\xi$ be an adapted smooth $\F$-line bundle on $X$. According to
Lemma~\ref{lem-6-7}, there exists a smooth section $s \colon X \to \xi$
transverse to the zero section and such that $Z \coloneqq Z(s)$ is a
nonsingular Zariski locally closed subvariety of $X$. Let $V$ be the
Zariski closure of $Z$ in $X$ and let $W \coloneqq V \setminus Z$. Since
$\xi$ is of rank $1$ and $Z(s|_W) = \varnothing$, it follows that the
$\F$-vector bundle $\xi|_W$ on $W$ is topologically trivial. In
particular, $\xi|_W$ admits a stratified-algebraic structure.

In view of Lemma~\ref{lem-6-8}, we may assume without loss of generality
that $Z$ is a nonsingular and Zariski closed subvariety of $X$ (that is,
$W=\varnothing$). Let $\pi \colon X' \to X$ be the blowup with center
$Z$. According to Corollary~\ref{cor-6-3}, the smooth $\R$-vector bundle
$\pi^*(\xi_{\R}) = (\pi^*\xi)_{\R}$ on $X'$ contains a smooth $\R$-line
subbundle $\lambda$ admitting an algebraic structure. Since $\pi^*\xi$
is an $\F$-line bundle, it follows that it is topologically isomorphic
to $\F \otimes \lambda$. Consequently, $\pi^*\xi$ admits an algebraic
structure. The restriction $(\xi_{\R})|_Z$ is smoothly
isomorphic to the normal bundle to $Z$ in $X$, cf. Lemma~\ref{lem-6-1}.
In particular, the $\R$-vector bundle $(\xi|_Z)_{\R} = (\xi_{\R})|_Z$
admits an algebraic structure. By Theorem~\ref{th-1-7}, the $\F$-line
bundle $\xi|_Z$ admits a stratified-algebraic structure. Since $\FC =
(\varnothing, Z, X)$ is a filtration of $X$, Theorem~\ref{th-5-4}
implies that $\xi$ admits a stratified-algebraic structure.
\end{proof}

\begin{proof}[Proof of Theorem~\ref{th-1-9}]
Assume that both $\F$-vectors bundles $\xi$ and $\det\xi$ are adapted.
According to Theorem~\ref{th-1-8}, the $\F$-line bundle $\det\xi$ admits
a stratified algebraic structure.

By Lemma~\ref{lem-6-7}, there exists a smooth section $s \colon X \to
\xi$ transverse to the zero section and such that $Z \coloneqq Z(s)$ is
a nonsingular Zariski locally closed subvariety of $X$. Let $V$ be the
Zariski closure of $Z$ in $X$ and let $W \coloneqq V \setminus Z$. Since
$\xi$ is of rank $2$ and $Z(s|_W) = \varnothing$, it follows that
$\xi|_W$ is topologically isomorphic to $\varepsilon^1_W(\F)\oplus\mu$
for some topological $\F$-line bundle $\mu$ on $W$. We have
\begin{equation*}
(\det\xi)|_W = \det(\xi|_W) \cong \det(\varepsilon^1_W(\F) \oplus \mu)
\cong \varepsilon^1_W(\F) \otimes \mu \cong \mu.
\end{equation*}
Consequently, the $\F$-vector bundle $\xi|_W$ on $W$ admits a
stratified-algebraic structure.

In view of Lemma~\ref{lem-6-8}, we may assume without loss of generality
that $Z$ is a nonsingular and Zariski closed subvariety of $X$ (that is,
$W = \varnothing$). Let $\pi \colon X' \to X$ be the blowup with center
$Z$. According to Corollary~\ref{cor-6-3}, the smooth $\R$-vector bundle
$\pi^*(\xi_{\R}) = (\pi^*\xi)_{\R}$ contains a smooth $\R$-line
subbundle $\lambda$ admitting an algebraic structure. Since $\pi^*\xi$
is an $\F$-vector bundle, it contains a smooth $\F$-line subbundle
$\xi_1$ isomorphic to $\F \otimes \lambda$. Hence
\begin{equation*}
\pi^*\xi = \xi_1 \oplus \xi_2,
\end{equation*}
where $\xi_1$ and $\xi_2$ are smooth $\F$-line bundles, and $\xi_1$
admits an algebraic structure. We have
\begin{equation*}
\pi^*(\det\xi) = \det(\pi^*\xi) = \det(\xi_1\oplus\xi_2) \cong \xi_1
\otimes \xi_2,
\end{equation*}
which implies
\begin{equation*}
\xi_1^{\vee} \otimes \pi^* (\det\xi) \cong \xi_1^{\vee} \otimes \xi_1
\otimes \xi_2 \cong \xi_2.
\end{equation*}
Thus, according to Propositions~\ref{prop-2-6} and~\ref{prop-3-15},
$\xi_2$ admits a stratified-algebraic structure. Consequently,
$\pi^*\xi$ admits a stratified-algebraic structure. The restriction
$(\xi_{\R})|_Z$ is smooth\-ly isomorphic to the normal bundle to $Z$ in
$X$, cf. Lemma~\ref{lem-6-1}. In particular, the $\R$-vector bundle
$(\xi|_Z)_{\R} = (\xi_{\R})|_Z$ admits an algebraic structure. By
Theorem~\ref{th-1-7}, the $\F$-vector bundle $\xi|_Z$ admits a
stratified-algebraic structure. Since $\FC
= (\varnothing, Z, X)$ is a filtration of $X$, Theorem~\ref{th-5-4}
implies that $\xi$ admits a stratified-algebraic structure.
\end{proof}

There are natural generalizations of Theorems~\ref{th-1-8}
and~\ref{th-1-9} to vector bundles on an arbitrary compact real
algebraic variety (not necessarily nonsingular).

First we recall a topological fact. If $M$ is a smooth manifold, then
each topological $\F$-vector bundle $\eta$ on $M$ is topologically
isomorphic to a smooth $\F$-vector bundle $\eta^{\infty}$, which is
uniquely determined up to smooth isomorphism, cf. \cite[p.~101,
Theorem~3.5]{bib27}.

A topological $\F$-vector bundle $\xi$ on a compact nonsingular real
algebraic variety $Y$ is said to be \emph{adapted} if the smooth
$\F$-vector bundle $\xi^{\infty}$ is adapted. In particular, $\xi$ is
adapted if it admits an algebraic structure.

\begin{theorem}\label{th-6-9}
Let $X$ be a compact real algebraic variety and let $\xi$ be a
topological $\F$-line bundle on $X$. Then the following conditions are
equivalent:
\begin{conditions}
\item\label{th-6-9-a} The $\F$-line bundle $\xi$ admits a
stratified-algebraic structure.

\item\label{th-6-9-b} There exists a filtration $\FC = (X_{-1}, X_0,
\ldots, X_m)$ of
$X$ with $\overline{\FC}$ a nonsingular strati\-fi\-ca\-tion, and for each $i
= 0, \ldots, m$, there exists a multiblowup $\pi_i \colon X'_i \to X_i$
over $X_{i-1}$ such that $X'_i$ is a nonsingular variety and the
pullback $\F$-line bundle $\pi_i^*(\xi|_{X_i})$ on $X'_i$ is adapted.
\end{conditions}
\end{theorem}

\begin{proof}
According to Corollary~\ref{cor-5-3}, (\ref{th-6-9-a}) implies
(\ref{th-6-9-b}). In view of Theorems~\ref{th-5-4} and~\ref{th-1-8},
(\ref{th-6-9-b}) implies (\ref{th-6-9-a}).
\end{proof}

\begin{theorem}\label{th-6-10}
Let $X$ be a compact real algebraic variety and let $\xi$ be a
topological $\F$-vector bundle of rank $2$ on $X$, where $\F=\R$ of
$\F=\CB$. Then the following conditions are equivalent:
\begin{conditions}
\item\label{th-6-10-a} The $\F$-vector bundle $\xi$ admits a
stratified-algebraic structure.

\item\label{th-6-10-b} There exists a filtration $\FC = (X_{-1}, X_0,
\ldots, X_m)$ of $X$ with $\overline{\FC}$ a nonsingular
stratification, and for each $i = 0, \ldots, m$, there exists a
multiblowup $\pi_i \colon X'_i \to X_i$ over $X_{i-1}$ such that $X'_i$
is a nonsingular variety and the pullback $\F$-vector bundles
$\pi_i^*(\xi|_{X_i})$ and $\pi_i^*((\det\xi)|_{X_i})$ on $X'_i$ are
adapted.
\end{conditions}
\end{theorem}

\begin{proof}
According to Corollary~\ref{cor-5-3}, (\ref{th-6-10-a}) implies
(\ref{th-6-10-b}). In view of Theorems~\ref{th-5-4} and~\ref{th-1-9},
(\ref{th-6-10-b}) implies (\ref{th-6-10-a}).
\end{proof}

\section{Stratified-algebraic cohomology classes}\label{sec-7}
To begin with, we summarize basic properties of algebraic cohomology
classes. Let $V$ be a compact nonsingular real algebraic variety. A
cohomology class $v$ in $H^k(V;\Z/2)$ is said to be \emph{algebraic} if
the homology class Poincar\'e dual to it can be represented by a Zariski
closed subvariety of $V$ of codimension $k$. The set $\Halg^k
(V; \Z/2)$ of all algebraic cohomology classes in $H^k(V; \Z/2)$ forms a
subgroup. The direct sum $\Halg^*(V; \Z/2)$ of the groups
$\Halg^k (V; \Z/2)$, for $k \geq 0$, is a subring of the
cohomology ring $H^* (V; \Z/2)$. For any algebraic $\R$-vector bundle
$\xi$ on $V$, its $k$th Stiefel--Whitney class $w_k(\xi)$ belongs to
$\Halg^k(V; \Z/2)$ for every $k \geq 0$. If $h \colon V \to
W$ is a regular map between compact nonsingular real algebraic
varieties, then
\begin{equation*}
h^*(\Halg^*(W; \Z/2)) \subseteq \Halg^* (V; \Z/2).
\end{equation*}
For proofs of the facts listed above we refer to \cite{bib20} or
\cite{bib1, bib6, bib9, bib17}.

We now introduce the main notion of this section. Let $X$ be a real
algebraic variety. A cohomology class $u$ in $H^k(X; \Z/2)$ is said to
be \emph{stratified-algebraic} if there exists a stratified-regular map
$\varphi \colon X \to V$, into a compact nonsingular real algebraic
variety $V$, such that $u = \varphi^*(v)$ for some cohomology class $v$
in $\Halg^k (V; \Z/2)$.

\begin{proposition}\label{prop-7-1}
For any real algebraic variety $X$, the set $\Hstr^k (X;
\Z/2)$ of all stratified-algebraic cohomology classes in $H^k(X; \Z/2)$
forms a subgroup. Furthermore, the direct sum
\begin{equation*}
\Hstr^* (X; \Z/2) = \bigoplus_{k \geq 0} \Hstr^k
(X; \Z/2)
\end{equation*}
is a subring of the cohomology ring $H^*(X;\Z/2)$.
\end{proposition}

\begin{proof}
Let $\varphi_i \colon X \to V_i$ be a stratified-regular map into a compact
nonsingular real algebraic variety $V_i$ for $i=1,2$. The stratified-regular map
$(\varphi_1, \varphi_2) \colon X \to V_1 \times V_2$ satisfies $p_i
\circ (\varphi_1, \varphi_2) = \varphi_i$, where $p_i \colon V_1 \times
V_2 \to V_i$ is the canonical projection. If $v_i$ (resp. $w_i$) is a
cohomology class in $H^k(V_i; \Z/2)$ (resp. $H^{k_i} (V_i; \Z/2)$) for
$i=1,2$, then
\begin{align*}
\varphi_1^*(v_1) + \varphi_2^*(v_2) &= (\varphi_1, \varphi_2)^*
(p_1^*(v_1) + p_2^*(v_2)), \\
\varphi_1^*(w_1) \cupproduct \varphi_2^*(w_2) &= (\varphi_1, \varphi_2)^*
(p_1^*(w_1) \cupproduct p_2^*(w_2)),
\end{align*}
where $\cupproduct$ stands for the cup product. If $v_i$ and $w_i$ are
algebraic cohomology classes, then the cohomology classes $p_1^*(v_1) +
p_2^*(v_2)$ and $p_1^*(w_1) \cupproduct p_2^*(w_2)$ are algebraic too. The
proof is complete.
\end{proof}

\begin{proposition}\label{prop-7-2}
If $f \colon X \to Y$ is a stratified-regular map between real algebraic
varieties, then
\begin{equation*}
f^*(\Hstr^*(Y; \Z/2)) \subseteq \Hstr^* (X; \Z/2).
\end{equation*}
\end{proposition}

\begin{proof}
Let $\psi \colon Y \to W$ be a stratified-regular map into a compact nonsingular
real algebraic variety $W$. For any cohomology class $w$ in $H^k(W;
\Z/2)$,
\begin{equation*}
f^*(\psi^*(w)) = (\psi \circ f)^*(w).
\end{equation*}
The proof is complete since $\psi \circ f$ is a stratified-regular map.
\end{proof}

\begin{proposition}\label{prop-7-3}
Let $\xi$ be a stratified-algebraic $\R$-vector bundle on a real
algebraic variety $X$. Then the $k$th Stiefel--Whitney class $w_k(\xi)$
of $\xi$ belongs to $\Hstr^k(X; \Z/2)$ for every $k \geq 0$.
\end{proposition}

\begin{proof}
According to Proposition~\ref{prop-3-4}, we may assume that $\xi$ is of
the form $\xi = f^*\gamma_K(\R^n)$ for some multi-Grassmannian
$\G_K(\R^n)$ and stratified-regular map $f \colon X \to \G_K(\R^n)$.
Then $w_k(\xi) = f^*(w_k (\gamma_K (\R^n) ) )$. The proof is complete in
view of Proposition~\ref{prop-7-2} since 
\begin{equation*}
\Halg^* (\G_K(\R^n); \Z/2) = H^*(\G_K(\R^n); \Z/2),
\end{equation*}
cf. \cite[Proposition~11.3.3]{bib9}.
\end{proof}

For any real algebraic variety $X$, let $\VB^1_{\R}(X)$ denote the group
of isomorphism classes of topological $\R$-line bundles on $X$ (with
operation induced by tensor product). The first Stiefel--Whitney class
gives rise to an isomorphism
\begin{equation*}
w_1 \colon \VB^1_{\R}(X) \to H^1(X; \Z/2).
\end{equation*}
Similarly, denote by $\VBRstr^1(X)$ the group of isomorphism classes of
stratified-algebraic $\R$-line bundles on $X$. Now, the first
Stiefel--Whitney class gives rise to a homomorphism
\begin{equation*}
w_1 \colon \VBRstr^1(X) \to H^1(X; \Z/2).
\end{equation*}

\begin{proposition}\label{prop-7-4}
For any real algebraic variety $X$,
\begin{equation*}
w_1(\VBRstr^1(X)) = \Hstr^1(X; \Z/2).
\end{equation*}
Furthermore, if the variety $X$ is compact, then the homomorphism
\begin{equation*}
w_1 \colon \VBRstr^1(X) \to H^1(X; \Z/2)
\end{equation*}
is injective.
\end{proposition}

\begin{proof}
The inclusion
\begin{equation*}
w_1(\VBRstr^1(X)) \subseteq \Hstr^1(X; \Z/2)
\end{equation*}
follows from Proposition~\ref{prop-7-3}. If $u$ is a cohomology class in
$\Hstr^1(X; \Z/2)$, then there exists a stratified-regular map $\varphi
\colon X \to V$, into a compact nonsingular real algebraic variety $V$,
such that $u = \varphi^*(v)$ for some cohomology class $v$ in
$\Halg^1(V; \Z/2)$. By \cite[Theorem~12.4.6]{bib9}, $v$ is of the form
$v = w_1(\lambda)$ for some algebraic $\R$-line bundle $\lambda$ on $V$.
The pullback $\R$-line bundle $\varphi^*\lambda$ on $X$ is
stratified-algebraic, cf. Proposition~\ref{prop-2-6}. Since $u =
\varphi^*(w_1(\lambda)) = w_1(\varphi^*\lambda)$, we get
\begin{equation*}
\Hstr^1(X; \Z/2) \subseteq w_1(\VBRstr^1(X)).
\end{equation*}
If the variety $X$ is compact, then the canonical homomorphism
$\VBRstr^1(X) \to \VB_{\R}^1(X)$ is injective (cf.
Theorem~\ref{th-3-10}), and hence the last assertion in the proposition
follows.
\end{proof}

Proposition~\ref{prop-7-4} can be interpreted as an approximation
result for maps into real projective $n$-space $\PB^n(\R) =
\G_1(\R^{n+1})$ with $n \geq 1$.

\begin{proposition}\label{prop-7-5}
Let $X$ be a compact real algebraic variety. For a continuous map $f
\colon X \to \PB^n(\R)$, the following conditions are equivalent:
\begin{conditions}
\item The map $f$ can be approximated by stratified-regular maps.

\item The cohomology class $f^*(u_n)$ belongs to $\Hstr^1 (X; \Z/2)$,
where $u_n$ is a generator of the group $H^1(\PB^n(\R); \Z/2) \cong
\Z/2$.
\end{conditions}
\end{proposition}

\begin{proof}
Since $u_n = w_1(\gamma_1(\R^{n+1}))$, we have $f^*(u_n) =
w_1(f^*\gamma_1(\R^{n+1}))$. Thus, it suffices to combine
Theorem~\ref{th-4-10} (with $A=\varnothing$) and
Proposition~\ref{prop-7-4}.
\end{proof}

\begin{corollary}\label{cor-7-6}
Let $X$ be a compact real algebraic variety. For a continuous
map \linebreak
$f \colon X \to \SB^1$, the following conditions are equivalent:
\begin{conditions}
\item The map $f$ can be approximated by stratified-regular maps.

\item The cohomology class $f^*(u)$ belongs to $\Hstr^1(X; \Z/2)$, where
$u$ is a generator of the group $H^1(\SB^1; \Z/2) \cong \Z/2$.
\end{conditions}
\end{corollary}

\begin{proof}
It suffices to apply Proposition~\ref{prop-7-5} since $\SB^1$ is
biregularly isomorphic to $\PB^1(\R)$.
\end{proof}

We next look for relationships between the groups $\Hstr^k (-; \Z/2)$
and $\Halg^k(-; \Z/2)$.

\begin{proposition}\label{prop-7-7}
If $X$ is a compact nonsingular real algebraic variety, then 
\begin{equation*}
\Hstr^*(X; \Z/2) = \Halg^*(X; \Z/2).
\end{equation*}
\end{proposition}

\begin{proof}
The inclusion
\begin{equation*}
\Halg^*(X; \Z/2) \subseteq \Hstr^*(X; \Z/2)
\end{equation*}
is obvious. According to \cite[Propostion~1.3]{bib32}, if $\varphi
\colon X \to V$ is a continuous rational map into a compact nonsingular
real algebraic variety $V$, then
\begin{equation*}
\varphi^* ( \Halg^*(V; \Z/2) ) \subseteq \Halg^*(X; \Z/2).
\end{equation*}
Consequently, the inclusion
\begin{equation*}
\Hstr^* (X; \Z/2) \subseteq \Halg^*(X; \Z/2)
\end{equation*}
follows since each stratified-regular map from $X$ into $V$ is
continuous rational, cf. Proposition~\ref{prop-2-2}.
\end{proof}

For an arbitrary compact real algebraic variety $X$, the group
$\Hstr^1(X; \Z/2)$ can be described as follows.

\begin{proposition}\label{prop-7-8}
Let $X$ be a compact real algebraic variety. For a cohomology class $u$
in $H^1(X; \Z/2)$, the following conditions are equivalent:
\begin{conditions}
\item The cohomology class $u$ belongs to $\Hstr^1(X; \Z/2)$.

\item There exists a filtration $\FC = (X_{-1}, X_0, \ldots, X_m)$ of
$X$ with $\overline{\FC}$ a nonsingular strati\-fication, and for each $i
=0, \ldots, m$, there exists a multiblowup $\pi_i \colon X'_i \to X_i$
over $X_{i-1}$ such that $X'_i$ is a nonsingular variety and the
cohomology class $\pi_i^*(u|_{X_i})$ belongs to $\Halg^1(X'_i; \Z/2)$.
Here $u|_{X_i}$ is the image of $u$ under the homomorphism \linebreak $H^1(X; \Z/2)
\to H^1(X_i; \Z/2)$ induced by the inclusion map $X_i \hookrightarrow
X$.
\end{conditions}
\end{proposition}

\begin{proof}
Let $\FC = (X_{-1}, X_0, \ldots, X_m)$ be a filtration of $X$ with
$\overline{\FC}$ a nonsingular stratification, and for each
$i=0,\ldots,m$, let $\pi_i \colon X'_i \to X_i$ be a multiblowup over
$X_{i-1}$ such that $X'_i$ is a nonsingular variety. If $\xi$ is a
topological $\R$-line bundle on $X$ with $w_1(\xi)=u$, then
$w_1(\xi|_{X_i}) = u|_{X_i}$ and $w_1(\pi_i^*(\xi|_{X_i})) =
\pi_i^*(u|_{X_i})$ for $0 \leq i \leq m$. In view of
\cite[Theorem~12.4.6]{bib9}, the $\R$-line bundle $\pi_i^*(\xi|_{X_i})$
on $X'_i$ admits an algebraic structure if and only if the cohomology
class $\pi_i^*(u|_{X_i})$ belongs to $\Halg^1(X'_i; \Z/2)$.
Consequently, it suffices to combine Corollary~\ref{cor-5-3} (with
$\F=\R$) and Propositions~\ref{prop-7-4} and~\ref{prop-7-7}.
\end{proof}

In \cite{bib19, bib-star}, algebraic cohomology classes are interpreted
as obstructions to representing homotopy classes by regular maps. Some
results contained in these papers can be strengthened by transferring them
to the framework of stratified objects, cf. Theorem~\ref{th-7-9} below.
First it is convenient to introduce some notation.

For any $n$-dimensional compact smooth manifold $M$, let $\llbracket M
\rrbracket$ denote its fundamental class in $H_n(M; \Z/2)$. If $M$ is a
smooth submanifold of a smooth manifold $N$, let $\llbracket M
\rrbracket_N$ denote the homology class in $H_n(N; \Z/2)$ represented by
$M$.

\begin{theorem}\label{th-7-9}
Let $Y$ be a real algebraic variety with $\Hstr^k(Y; \Z/2) \neq 0$ for
some ${k \geq 1}$. Then there exist a compact connected nonsingular real
algebraic variety $X$ and a continuous map $f \colon X \to Y$ such that
$\dim X = k+1$ and $f$ is not homotopic to any stratified-regular map.
\end{theorem}

\begin{proof}
Let $w$ be a nonzero cohomology class in $\Hstr^k(Y; \Z/2)$. Choose a
homology class $\beta$ in $H_k(Y; \Z/2)$ which satisfies
\begin{equation*}
\langle w, \beta \rangle \neq 0
\end{equation*}
and is represented by a singular cycle with support contained in one
connected component of $Y$. By \cite[Theorem~2.1]{bib19}, $\beta$ can be
expressed as
\begin{equation*}
\beta = h_*(\llbracket K \rrbracket),
\end{equation*}
where $K$ is a $k$-dimensional compact connected smooth manifold that is
the boundary of a compact smooth manifold with boundary and $h \colon K
\to Y$ is a continuous map. Let $p_0$ be a
point in $\SB^1$. Since the normal bundle of $K \times \{p_0\}$ in $K
\times \SB^1$ is trivial, according to \cite[Theorem~2.6,
Proposition~2.5]{bib19}, there exist a nonsingular real algebraic
variety $X$ and a smooth diffeomorphism $\varphi \colon X \to K \times
\SB^1$ such that the homology class $\alpha \coloneqq \llbracket
\varphi^{-1} (K \times \{p_0\}) \rrbracket_X$ satisfies
\begin{equation*}
\langle u, \alpha \rangle = 0
\end{equation*}
for every cohomology class $u$ in $\Halg^k(X; \Z/2)$.

Let $\pi \colon K \times \SB^1 \to K$ be the canonical projection. It
suffices to prove that the continuous map
\begin{equation*}
f \coloneqq h \circ \pi \circ \varphi \colon X \to Y
\end{equation*}
is not homotopic to any regular map. This can be done as follows. We
have $w = \psi^*(v)$ for some stratified-regular map $\psi \colon Y \to
W$ into a compact nonsingular real algebraic variety $V$ and a
cohomology class $v$ in $\Halg^k(V; \Z/2)$. Since $(\pi \circ \varphi)_*
(\alpha) = \pi_* (\varphi_* (\alpha) ) = \llbracket K \rrbracket$, we
get
\begin{equation*}
(\psi \circ f)_* (\alpha) = (\psi \circ h \circ \pi \circ \varphi)_*
(\alpha) = \psi_* (h_* (\llbracket K \rrbracket) ) =\psi_*(\beta).
\end{equation*}
Consequently,
\begin{equation*}
\langle (\psi \circ f)^*(v), \alpha \rangle = \langle v, (\psi \circ f)_*
(\alpha) \rangle = \langle v, \psi_*(\beta) \rangle = \langle \psi^*(v),
\beta \rangle = \langle w, \beta \rangle \neq 0.
\end{equation*}
It follows that the cohomology class $(\psi \circ f)^*(v)$ does not
belong to $\Halg^k(X; \Z/2)$. However, if the map $f \colon X \to Y$
were homotopic to a stratified-regular map $g \colon X \to Y$, then
$\psi \circ f$ would be homotopic to a stratified-regular map $\psi
\circ g$, and hence the cohomology class $(\psi \circ f)^* (v) = (\psi
\circ g)^* (v)$ would be in $\Halg^k(X; \Z/2)$, cf.
Propositions~\ref{prop-7-2} and~\ref{prop-7-7}. This completes the
proof.
\end{proof}

We conclude this section by giving an example of an $\F$-line bundle not
admitting a stratified-algebraic structure. Recall that $d(\F) =
\dim_{\R} \F$ and $\T^n = \SB^1 \times \cdots \times \SB^1$ (the $n$-fold
product).

\begin{example}\label{ex-7-10}
For any integer $n > d(\F)$, there exist a nonsingular real algebraic
variety $X$ and a topological $\F$-line bundle
$\xi$ on $X$ such that $X$ is diffeomorphic to $\T^n$ and $\xi$ does not
admit a stratified-algebraic structure. This assertion can be proved as
follows. Let $d \coloneqq d(\F)$ and let $y_0$ be a point in $\T^{n-d}$.
Let $\alpha$ be the homology class in $H_d(\T^n; \Z/2)$ represented by
the smooth submanifold $K \coloneqq \T^d \times \{y_0\}$ of $\T^n$. Set
\begin{equation*}
A \coloneqq \{ u \in H^d(\T^n; \Z/2) \mid \langle u, \alpha \rangle = 0
\}.
\end{equation*}
Since the normal bundle of $K$ in $\T^n$ is trivial and $K$ is the
boundary of a compact smooth manifold with boundary, it follows from
\cite[Propositon~2.5, Theorem~2.6]{bib19} that there exist a nonsingular
real algebraic variety $X$ and a smooth diffeomorphism $\varphi \colon X
\to \T^n$ with
\begin{equation*}
\Halg^d(X; \Z/2) \subseteq \varphi^*(A).
\end{equation*}
For the $\F$-line bundle $\gamma_1(\F^2)$ on $\G_1(\F^2)$, one has
$w_d(\gamma_1(\F^2)_{\R}) \neq 0$ in $H^d(\G_1(\F^2); \Z/2)$. Choosing a
continuous map $h \colon \T^d \to \G_1(\F^2)$ for which the induced
homomorphism
\begin{equation*}
h^* \colon H^d(\G_1(\F^2); \Z/2) \to H^d(\T^d; \Z/2)
\end{equation*}
is an isomorphism, we obtain a topological $\F$-line bundle $\lambda
\coloneqq h^* \gamma_1 (\F^2)$ with $w_d(\lambda_{\R}) \neq 0$ in
$H^d(\T^d; \Z/2)$. If $p \colon \T^n = \T^d \times \T^{n-d} \to \T^d$ is
the canonical projection and $\eta \coloneqq p^*\lambda$, then
\begin{equation*}
w_d(\eta_{\R}) \notin A.
\end{equation*}
Consequently, for the $\F$-line bundle $\xi \coloneqq \varphi^* \eta$ on
$X$, we have
\begin{equation*}
w_d(\xi_{\R}) \notin \Halg^d(X; \Z/2).
\end{equation*}
In view of Propositions~\ref{prop-7-3} and~\ref{prop-7-7}, the
$\R$-vector bundle $\xi_{\R}$ cannot admit a stratified-algebraic
structure. Hence $\xi$ does not admit a stratified-algebraic structure.
\end{example}

\section{\texorpdfstring{Stratified-$\CB$-algebraic cohomology
classes}{Stratified-C-algebraic cohomology classes}}\label{sec-8}

We first recall the construction of $\CB$-algebraic cohomology classes.
Let $V$ be a compact nonsingular real algebraic variety. A
\emph{nonsingular projective complexification} of $V$ is a pair $(\V,
\iota)$, where $\V$ is a nonsingular projective scheme over $\R$ and
$\iota \colon V \to \V(\CB)$ is an injective map such that $\V(\R)$ is
Zariski dense in $\V$, $\iota(V) = \V(\R)$ and $\iota$ induces a
biregular isomorphism between $V$ and $\V(\R)$. Here the set $\V(\R)$ of
real points of $\V$ is regarded as a subset of the set $\V(\CB)$ of
complex points of $\V$. The existence of $(\V, \iota)$ follows from
Hironaka's theorem on resolution of singularities \cite{bib26, bib30}.
We identify $\V(\CB)$ with the set of complex points of the scheme
$\V_{\CB} \coloneqq \V \times_{\Spec \R} \Spec \CB$ over $\CB$. The cycle
map
\begin{equation*}
\cl_{\CB} \colon A^*(\V_{\CB}) = \bigoplus_{k \geq 0} A^k(\V_{\CB}) \to
\Heven (\V(\CB);\Z) = \bigoplus_{k \geq 0} H^{2k}(\V(\CB); \Z)
\end{equation*}
is a ring homomorphism defined on the Chow ring of $\V_{\CB}$, cf.
\cite{bib20} or \cite[Corollary~19.2]{bib24}. Hence
\begin{equation*}
\Halg^{2k}(\V(\CB); \Z) \coloneqq \cl_{\CB}(A^k(\V_{\CB}))
\end{equation*}
is the subgroup of $H^{2k}(\V(\CB);\Z)$ that consists of the cohomology
classes corresponding to algebraic cycles (defined over $\CB$) on
$\V_{\CB}$ of codimension $k$. By construction,
\begin{equation*}
\HCalg^{2k} (V; \Z) \coloneqq \iota^* (\Halg^{2k} (\V(\CB); \Z) )
\end{equation*}
is a subgroup of $H^{2k}(V;\Z)$, and
\begin{equation*}
\HCalgeven (V; \Z) \coloneqq \bigoplus_{k \geq 0} \HCalg^{2k} (V; \Z)
\end{equation*}
is a subring of $\Heven(V;\Z)$. The subring $\HCalgeven (V; \Z)$ does
not depend on the choice of $(\V, \iota)$. If the variety $V$ is
irreducible, then $\HCalg^0(V; \Z)$ is the subgroup generated by $1 \in
H^0(V;\Z)$. A cohomology class $v$ in $H^{2k}(V;\Z)$ is said to be
\emph{$\CB$-algebraic} if it belongs to $\HCalg^{2k}(V;\Z)$. For any
algebraic $\CB$-vector bundle $\xi$ on $V$, its $k$th Chern class
$c_k(\xi)$ is in $\HCalg^{2k}(V;\Z)$ for every $k \geq 0$. If $h \colon
V \to W$ is a regular map between compact nonsingular real algebraic
varieties, then 
\begin{equation*}
h^*(\HCalgeven(W;\Z)) \subseteq \HCalgeven(V;\Z).
\end{equation*}
Proofs of these properties of $\HCalgeven(-;\Z)$ are given in
\cite{bib8}. Numerous applications of $\CB$-algebraic cohomology classes
can be found in \cite{bib8, bib12, bib16, bib18, bib33}.

Let $X$ be a real algebraic variety. A cohomology class $u$ in
$H^{2k}(X;\Z)$ is said to be \emph{stratified-$\CB$-algebraic} if there
exists a stratified-regular map $\varphi \colon X \to V$, into a compact
nonsingular real algebraic variety $V$, such that $u = \varphi^*(v)$ for some
cohomology class $v$ in $\HCalg^{2k}(V;\Z)$.

\begin{proposition}\label{prop-8-1}
For any real algebraic variety $X$, the set $\HCstr^{2k}(X;\Z)$ of all
stratified-$\CB$-algebraic cohomology classes in $H^{2k}(V;\Z)$ forms a
subgroup. Furthermore, the direct sum
\begin{equation*}
\HCstreven (X;\Z) \coloneqq \bigoplus_{k \geq 0} \HCstr^{2k} (X;\Z)
\end{equation*}
is a subring of $\Heven(X;\Z)$.
\end{proposition}

\begin{proof}
It suffices to adapt the proof of Proposition~\ref{prop-7-1}.
\end{proof}

\begin{proposition}\label{prop-8-2}
If $f \colon X \to Y$ is a stratified-regular map between real algebraic
varieties, then
\begin{equation*}
f^*(\HCstreven(Y;\Z)) \subseteq \HCstreven(X;\Z).
\end{equation*}
\end{proposition}

\begin{proof}
The proof of Proposition~\ref{prop-7-2} also works in the case under
consideration here.
\end{proof}

\begin{proposition}\label{prop-8-3}
Let $\xi$ be a stratified-algebraic $\CB$-vector bundle on a real
algebraic variety $X$. Then the $k$th Chern class $c_k(\xi)$ of $\xi$
belongs to $\HCstr^{2k}(X; \Z)$ for every $k \geq 0$.
\end{proposition}

\begin{proof}
One can copy the proof of Proposition~\ref{prop-7-3}. It suffices to
observe that
\begin{equation*}
\HCalgeven (\G_K(\CB^n); \Z) = \Heven (\G_K(\CB^n); \Z),
\end{equation*}
cf. \cite[Example~5.5]{bib33}.
\end{proof}

\begin{corollary}\label{cor-8-4}
Let $\xi$ be a stratified-regular $\R$-vector bundle on a real algebraic
variety $X$. Then the $k$th Pontryagin class $p_k(\xi)$ of $\xi$ belongs
to $\HCstr^{4k}(X; \Z)$ for every $k \geq 0$.
\end{corollary}

\begin{proof}
Since $p_k(\xi) = (-1)^k c_{2k}(\CB \otimes \xi)$, is suffices to make
use of Proposition~\ref{prop-8-3}.
\end{proof}

For any topological $\F$-vector bundle $\xi$ on $X$, one can interpret
$\rank \xi$ as an element of $H^0(X;\Z)$. Then the following holds.

\begin{proposition}\label{prop-8-5}
Let $X$ be a real algebraic variety. If $\xi$ is a stratified-algebraic
$\F$-vector bundle on $X$, then $\rank \xi$ belongs to $\HCstr^0(X;\Z)$.
Conversely, each cohomology class in $\HCstr^0(X;\Z)$ is of the form
$\rank \eta$ for some stratified-algebraic $\F$-vector bundle $\eta$ on
$X$, whose restriction to each connected component of $X$ is
topologically trivial.
\end{proposition}

\begin{proof}
By Proposition~\ref{prop-3-4}, $\xi$ is of the form $\xi =
f^*\gamma_K(\F^n)$ for some multi-Grassmannian $\G_K(\F^n)$ and
stratified-regular map $f \colon X \to \G_K(\F^n)$. Since $\rank \xi =
f^*(\rank \gamma_K(\F^n))$ and
\begin{equation*}
\HCalg^0(\G_K(\F^n);\Z) = H^0(\G_K(\F^n);\Z),
\end{equation*}
it follows that $\rank\xi$ belongs to $\HCstr^0(X;\Z)$. The second part
of the proposition readily follows from the description of
$\HCstr^0(X;\Z)$.
\end{proof}

For any real algebraic variety $X$, let $\VB_{\CB}^1(X)$ denote the group
of isomorphism classes of topological $\CB$-line bundles on $X$ (with
operation induced by tensor product). The first Chern class gives rise
to an isomorphism
\begin{equation*}
c_1 \colon \VB_{\CB}^1(X) \to H^2(X;\Z).
\end{equation*}
Similarly, denote by $\VBCstr^1(X)$ the group of isomorphism classes of
stratified-algebraic $\CB$-line bundles on $X$. Now, the first Chern
class gives rise to a homomorphism
\begin{equation*}
c_1 \colon \VBCstr^1(X) \to H^2(X;\Z).
\end{equation*}

\begin{proposition}\label{prop-8-6}
For any real algebraic variety $X$,
\begin{equation*}
c_1(\VBCstr^1(X)) = \HCstr^2(X;\Z).
\end{equation*}
Furthermore, if the variety $X$ is compact, then the homomorphism
\begin{equation*}
c_1 \colon \VBCstr^1(X) \to H^2(X;\Z)
\end{equation*}
is injective.
\end{proposition}

\begin{proof}
The inclusion
\begin{equation*}
c_1(\VBCstr^1(X)) \subseteq \HCstr^2(X;\Z)
\end{equation*}
follows from Proposition~\ref{prop-8-3}. If $u$ is a cohomology class in
$\HCstr^2(X;\Z)$, then there exists a stratified-regular map $\varphi
\colon X \to V$, into a compact nonsingular real algebraic variety $V$,
such that $u = \varphi^*(v)$ for some cohomology class $v$ in
$\HCalg^2(V;\Z)$. By \cite[Remark~5.4]{bib8}, $v$ is of the form $v =
c_1(\lambda)$ for some algebraic $\CB$-line bundle $\lambda$ on $V$. The
pullback $\CB$-line bundle $\varphi^*\lambda$ on $X$ is
stratified-algebraic, cf. Proposition~\ref{prop-2-6}. Since $u =
\varphi^*(c_1(\lambda)) = c_1(\varphi^*\lambda)$, we get
\begin{equation*}
\HCstr^2(X;\Z) \subseteq c_1(\VBCstr^1(X)).
\end{equation*}
If the variety $X$ is compact, then the canonical homomorphism
$\VBCstr^1(X) \to \VB_{\CB}^1(X)$ is injective (cf.
Theorem~\ref{th-3-10}), and hence the last assertion in the proposition
follows.
\end{proof}

Proposition~\ref{prop-8-6} can be interpreted as an approximation result
for maps into complex projective $n$-space $\PB^n(\CB) =
\G_1(\CB^{n+1})$ for $n \geq 1$.

\begin{proposition}\label{prop-8-7}
Let $X$ be a compact real algebraic variety. For a continuous map $f
\colon X \to \PB^n(\CB)$, the following conditions are equivalent:
\begin{conditions}
\item The map $f$ can be approximated by stratified-regular maps.

\item The cohomology class $f^*(v_n)$ belongs to $\HCstr^2(X;\Z)$, where
$v_n$ is a generator of the group $H^2(\PB^n(\CB);\Z) \cong \Z$.
\end{conditions}
\end{proposition}

\begin{proof}
We may assume that $v_n = c_1(\gamma_1(\CB^{n+1}))$. Since $f^*(v_n) =
c_1(f^*\gamma_1(\CB^{n+1}))$, it suffices to make use of
Theorem~\ref{th-4-10} (with $A \neq \varnothing$) and
Proposition~\ref{prop-8-6}.
\end{proof}

\begin{corollary}\label{cor-8-8}
Let $X$ be a compact real algebraic variety. For a continuous
map\linebreak $f
\colon X \to \SB^2$, the following conditions are equivalent:
\begin{conditions}
\item The map $f$ can be approximated by stratified-regular maps.

\item The cohomology class $f^*(s_2)$ belongs to $\HCstr^2(X;\Z)$, where
$s_2$ is a generator of the group $H^2(\SB^2;\Z) \cong \Z$.
\end{conditions}
\end{corollary}

\begin{proof}
It suffices to apply Proposition~\ref{prop-8-7} since $\SB^2$ is
biregularly isomorphic to $\PB^1(\CB)$.
\end{proof}

Any real algebraic variety $X$ is homotopically equivalent to a compact
polyhedron \cite[Corollary~9.6.7]{bib9}, and hence the Chern character
\begin{equation*}
\ch \colon K_{\CB}(X) \to \Heven(X;\Q)
\end{equation*}
induces an isomorphism
\begin{equation*}
\ch \colon K_{\CB}(X) \otimes \Q \to \Heven(X;\Q),
\end{equation*}
cf. \cite{bib4} or \cite[p.~255, Theorem~A]{bib25}. On the other hand,
the Chern character
\begin{equation*}
\ch \colon \KCstr(X) \to \Heven(X;\Q)
\end{equation*}
induces a homomorphism
\begin{equation*}
\ch \colon \KCstr(X) \otimes \Q \to \Heven(X;\Q).
\end{equation*}
We next describe the image of the last homomorphism. Denote by
$\HCstreven(X;\Q)$ the image of $\HCstreven(X;\Z)\otimes\Q$ by the
canonical isomorphism $\Heven(X;\Z) \otimes\Q \to \Heven(X;\Q)$.

\begin{proposition}\label{prop-8-9}
For any real algebraic variety $X$,
\begin{equation*}
\ch(\KCstr(X) \otimes \Q) = \HCstreven(X;\Q).
\end{equation*}
Furthermore, if the variety $X$ is compact, then the homomorphism
\begin{equation*}
\ch \colon \KCstr(X) \otimes \Q \to \Heven(X;\Q)
\end{equation*}
is injective.
\end{proposition}

\begin{proof}
The inclusion
\begin{equation*}
\ch(\KCstr(X) \otimes \Q) \subseteq \HCstreven(X;\Q)
\end{equation*}
follows from Propositions~\ref{prop-8-3} and~\ref{prop-8-5}. If $u$ is a
cohomology class in $\HCstreven(X;\Q)$, then there exists a
stratified-regular map $\varphi \colon X \to V$, into a compact
nonsingular real algebraic variety $V$, such that $u = \varphi^*(v)$ for
some cohomology class $v$ in $\HCalgeven(V; \Q)$. By
\cite[Proposition~5.4]{bib8}, $v$ is of the form $v = \ch(\alpha)$ for
some element $\alpha$ in $\KCalg(X)\otimes\Q$. In view of
Proposition~\ref{prop-2-6}, $\varphi^*(\alpha)$ is in
$\KCstr(X)\otimes\Q$. Since $u = \varphi^*(\ch(\alpha)) =
\ch(\varphi^*(\alpha))$, we get
\begin{equation*}
\HCstreven(X;\Q) \subseteq \ch(\KCstr(X)\otimes\Q).
\end{equation*}
If the variety $X$ is compact, then the canonical homomorphism
$\KCstr(X) \to K_{\CB}(X)$ is injective (cf. Corollary~\ref{cor-3-11}),
and hence the last assertion in the proposition follows.
\end{proof}

One can also prove a sharper result than the first part of
Proposition~\ref{prop-8-9}.

\begin{proposition}\label{prop-8-10}
Let $X$ be a real algebraic variety and let $k$ be a positive integer.
For any cohomology class $u$ in $\HCstr^{2k}(X;\Z)$, there exists a
stratified-algebraic $\CB$-vector bundle $\xi$ on $X$ with $c_i(\xi)=0$
for $1 \leq i \leq k-1$ and $c_k(\xi) = (-1)^{k-1}(k-1)!u$.
\end{proposition}

\begin{proof}
We first consider $\CB$-algebraic cohomology classes.

\begin{assertion}
Let $V$ be a compact nonsingular real algebraic variety and let $v$ be a
cohomology class in $\HCalg^{2k}(V;\Z)$, where $k \geq 1$. Then there
exists an algebraic $\CB$-vector bundle $\eta$ on $V$ with $c_i(\eta)=0$
for $1 \leq i \leq k-1$ and $c_k(\eta) = (-1)^{k-1}(k-1)!v$.
\end{assertion}

The Assertion can be proved as follows. Let $(\V,\iota)$ be a
nonsingular projective complexification of $V$. We may assume that $V =
\V(\R)$ and $\iota \colon V \hookrightarrow \V(\CB)$ is the inclusion
map. By definition, $v = \iota^*(\cl_{\CB}(z))$ for some element $z$ in
the Chow group $A^k(\V_{\CB})$. Let $\U$ be an affine Zariski open
subscheme of $\V$ containing $V$. According to Grothendieck's formula
\cite[Example~15.3]{bib24} (see also \cite[Propostition~19.1.2]{bib24}),
there exists an algebraic vector bundle $\E$
on $\U_{\CB}$ with $c_i(\E)=0$ for $1 \leq i \leq k-1$ and $c_k(\E) =
(-1)^{k-1}(k-1)!j^*(\cl_{\CB}(z))$, where $j \colon \U(\CB) \to \V(\CB)$
is the inclusion map. The restriction $\eta \coloneqq \E|_V$ is an
algebraic $\CB$-vector bundle on $V$ satisfying all the conditions
stated in the Assertion.

We can now easily complete the proof. The cohomology class $u$ is of the
form ${u = \varphi^*(v)}$, where $\varphi \colon X \to V$ is a
stratified-regular map into a compact nonsingular real algebraic variety
$V$, and $v$ is a cohomology class in $\HCalg^{2k}(V;\Z)$. If $\eta$ is
an algebraic $\CB$-vector bundle on $V$ as in the Assertion, then the
pullback $\varphi^*\eta$ is a stratified-algebraic $\CB$-vector bundle
on $X$ (cf. Proposition~\ref{prop-2-6}) having all the required
properties.
\end{proof}

Under some assumptions we obtain a nice characterization of topological
$\CB$-vector bundles admitting a stratified-algebraic structure.

\begin{theorem}\label{th-8-11}
Let $X$ be a compact real algebraic variety. Assume that the group
$H^*(X;\Z)$ has no torsion and the quotient group $H^{2k}(X;\Z) /
\HCstr^{2k}(X;\Z)$ has no $(k-1)!$-torsion elements for every $k \geq
1$. For a topological $\CB$-vector bundle $\xi$ on $X$, the following
conditions are equivalent:
\begin{conditions}
\item\label{th-8-11-a} $\xi$ admits a stratified-algebraic structure.

\item\label{th-8-11-b} $\rank\xi$ belongs to $\HCstr^0(X;\Z)$, and
$c_k(\xi)$ belongs to $\HCstr^{2k}(X;\Z)$ for every $k \geq 1$.
\end{conditions}
\end{theorem}

\begin{proof}
According to Propositions~\ref{prop-8-3} and~\ref{prop-8-5},
(\ref{th-8-11-a}) implies (\ref{th-8-11-b}). For the proof of the
reversed implication, we first establish the following:

\begin{assertion}
Let $\theta$ be a topological $\CB$-vector bundle on $X$ and let $k$ be
a positive integer such that $c_i(\theta)=0$ for $1 \leq i \leq k-1$ and
$c_k(\theta)$ belongs to $\HCstr^{2k}(X;\Z)$. Then there exists a
stratified-algebraic $\CB$-vector bundle $\theta_k$ on $X$ with
$c_i(\theta_k)=0$ for $1 \leq i \leq k-1$ and $c_k(\theta \oplus
\theta_k) = 0$.
\end{assertion}

Recall that $X$ is a compact polyhedron, cf.
\cite[Corollary~9.6.7]{bib9}. Since the group $H^*(X;\Z)$ has no
torsion, the $k$th Chern class $c_k(\theta)$ is of the form
\begin{equation*}
c_k(\theta) = (-1)^{k-1}(k-1)!u
\end{equation*}
for some cohomology class $u$ in $H^{2k}(X;\Z)$, cf. \cite[p.~19]{bib4}. By
assumption, $(-1)^{k-1}(k-1)!u$ is in $\HCstr^{2k}(X;\Z)$, and hence $u$
is in $\HCstr^{2k}(X;\Z)$, the quotient group
\begin{equation*}
H^{2k}(X;\Z)/\HCstr^{2k}(X;\Z)
\end{equation*}
having no $(k-1)!$-torsion elements.
According to Proposition~\ref{prop-8-10}, there exists a
stratified-algebraic $\CB$-vector bundle $\theta_k$ on $X$ with
$c_i(\theta_k)=0$ for $1 \leq i \leq k-1$ and $c_k(\theta_k) =
(-1)^{k-1}(k-1)!(-u)$. Thus, $\theta_k$ satisfies all the conditions
stated in the Assertion.

If (\ref{th-8-11-b}) holds, then making use of the Assertion, we obtain
a stratified-algebraic $\CB$-vector bundle $\zeta$ on $X$ with $c_i(\xi
\oplus \zeta) = 0$ for every $i \geq 1$. If $\eta$ is a
stratified-algebraic $\CB$-vector bundle on $X$ such that the direct sum
$\zeta \oplus \eta$ is a trivial vector bundle (cf.
Proposition~\ref{prop-3-7}), then $c_i(\xi) = c_i(\eta)$ for every $i
\geq 1$. Now, in view of the second part of Proposition~\ref{prop-8-5},
there exists a stratified-algebraic $\CB$-vector bundle $\xi'$ on $X$
such that $\rank \xi - \rank \xi'$ is constant, and $c_i(\xi) =
c_i(\xi')$ for every $i \geq 1$. Consequently, the $\CB$-vector bundles
$\xi$ and $\xi'$ are topologically stably equivalent, the group
$\Heven(X;\Z)$ having no torsion, cf. \cite{bib40}. Hence, according to
Corollary~\ref{cor-3-14}, $\xi$ admits a stratified-algebraic structure.
\end{proof}

Any $\HB$-vector bundle $\xi$ can be regarded as a $\CB$-vector bundle,
which is indicated by $\xi_{\CB}$.

\begin{corollary}\label{cor-8-12}
Let $X$ be a compact real algebraic variety. Assume that the group
$H^*(X; \Z)$ has no torsion and the quotient group $H^{2k}(X;\Z) /
\HCstr^{2k}(X;\Z)$ has no $(k-1)!$-torsion elements for every $k \geq
1$. For a topological $\HB$-vector bundle $\xi$ on $X$, the following
conditions are equivalent:
\begin{conditions}
\item\label{cor-8-12-a} $\xi$ admits a stratified-algebraic structure.

\item\label{cor-8-12-b} $\rank \xi$ belongs to $\HCstr^0(X;\Z)$, and
$c_{2k}(\xi_{\CB})$ belongs to $\HCstr^{4k}(X;\Z)$ for every $k \geq
1$.
\end{conditions}
\end{corollary}

\begin{proof}
According to Propositions~\ref{prop-8-3} and~\ref{prop-8-5},
(\ref{cor-8-12-a}) implies (\ref{cor-8-12-b}). Since $\xi$ is an
$\HB$-vector bundle, it follows that $c_{2i+1}(\xi)=0$ for every $i \geq
0$. Hence, in view of Theorem~\ref{th-8-11}, if (\ref{cor-8-12-b})
holds, then the $\CB$-vector bundle $\xi_{\CB}$ admits a
stratified-algebraic structure. Consequently, the $\R$-vector bundle
$\xi_{\R} = (\xi_{\CB})_{\R}$ admits a stratified-algebraic structure.
Thus, by Theorem~\ref{th-1-7}, (\ref{cor-8-12-b}) implies
(\ref{cor-8-12-a}).
\end{proof}

For $\R$-vector bundles we can only obtain a weaker result.

\begin{corollary}\label{cor-8-13}
Let $X$ be a compact real algebraic variety and let $\xi$ be a
topological $\R$-vector bundle on $X$. Assume that the group $H^*(X;\Z)$
has no torsion and the quotient group $H^{2k}(X;\Z) / \HCstr^{2k}(X;\Z)$
has no $(k-1)!$-torsion elements for every $k \geq 1$. If $\rank \xi$
belongs to $\HCstr^0(X;\Z)$, and $p_k(\xi)$ belongs to
$\HCstr^{4k}(X;\Z)$ for every $k \geq 1$, then the direct sum
$\xi \oplus \xi$ admits a stratified-algebraic structure.
\end{corollary}

\begin{proof}
We have $c_{2i+1}(\CB\otimes\xi)=0$ for every $i \geq 0$ since
$c_{2i+1}(\CB\otimes\xi)$ is always an element of order at most $2$ (cf.
\cite[p.~174]{bib38}) and the group $\Heven(X;\Z)$ has no torsion.
Moreover, $c_{2k}(\CB\otimes\xi) = (-1)^kp_k(\xi)$ for every $k \geq 1$.
According to Theorem~\ref{th-8-11}, the $\CB$-vector bundle
$\CB\otimes\xi$ admits a stratified-algebraic structure. The proof is
complete since ${(\CB\otimes\xi)_{\R} = \xi \oplus \xi}$.
\end{proof}

We next identify some classes of real algebraic varieties which satisfy
the assumptions of the last three results.

\begin{lemma}\label{lem-8-14}
If $X$ is a compact real algebraic variety of even dimension $2k$, then
\begin{equation*}
\HCstr^{2k}(X;\Z) = H^{2k}(X;\Z).
\end{equation*}
\end{lemma}

\begin{proof}
Let $s_{2k}$ be a generator of the cohomology group $H^{2k}(\SB^{2k};\Z)
\cong \Z$. Each cohomology class $u$ in $H^{2k}(X;\Z)$ is of the form
$u=f^*(s_{2k})$ for some continuous map $f \colon X \to \SB^{2k}$.
According to Theorem~\ref{th-2-5}, we may assume that $f$ is
stratified-regular. The proof is complete since
$\HCalg^{2k}(\SB^{2k};\Z) = H^{2k}(\SB^{2k}; \Z)$, cf.
\cite[Propostion~4.8]{bib8}.
\end{proof}

\begin{proposition}\label{prop-8-15}
Let $X = X_1 \times \cdots \times X_n$, where each $X_i$ is a compact
real algebraic variety homotopically equivalent to the unit $d_i$-sphere
for $1 \leq i \leq n$. Then
\begin{equation*}
\HCstreven(X;\Z) = \Heven(X;\Z).
\end{equation*}
\end{proposition}

\begin{proof}
Let $p_i \colon X \to X_i$ be the canonical projection, $1 \leq i \leq
n$. For each pair $(j,l)$ of integers satisfying $1 \leq j \leq l \leq
n$, let $q_{jl} \colon X \to X_j \times X_l$ be the canonical
projection. The $\Z$-algebra $\Heven(X;\Z)$ is generated by all
$p_i^*(H^{d_i}(X_i; \Z))$ with $d_i$ even and all $q_{jl}^*( H^{d_j +
d_l} (X_j \times X_l; \Z) )$ with $d_j$ and $d_l$ odd. The proof is
complete in view of Lemma~\ref{lem-8-14} and Propositions~\ref{prop-8-1}
and~\ref{prop-8-2}.
\end{proof}

The reader may compare Proposition~\ref{prop-8-15} and results of
\cite{bib14, bib16, bib18} to see a sharp difference between $\HCstreven(-;\Z)$
and $\HCalgeven(-;\Z)$.

\begin{proof}[Proof of Theorem~\ref{th-1-10}]
It suffices to combine Theorem~\ref{th-8-11},
Corollaries~\ref{cor-8-12}, \ref{cor-8-13}, and
Proposition~\ref{prop-8-15}.
\end{proof}

Theorem~\ref{th-1-10} can be interpreted as an approximation result for
maps into Grassmannians.

\begin{theorem}\label{th-8-16}
Let $X = X_1 \times \cdots X_n$, where each $X_i$ is a compact real
algebraic variety homotopically equivalent to the unit $d_i$-sphere for
$1 \leq i \leq n$. If $\F = \CB$ or $\F = \HB$, then for any pair
$(k,m)$ of integers satisfying $1 \leq k \leq m$, each continuous map
from $X$ into the Grassmannian $\G_k(\F^m)$ can be approximated by
stratified-regular maps.
\end{theorem}

\begin{proof}
If suffices to make use of Theorem~\ref{th-1-10} and Theorem~\ref{th-4-10}
(with $A = \varnothing$).
\end{proof}

As an interesting special case, we obtain the following approximation
result for maps with values in the unit spheres $\SB^2$ or $\SB^4$.

\begin{corollary}\label{cor-8-17}
Let $X = X_1 \times \cdots \times X_n$, where each $X_i$ is a compact
real algebraic variety homotopically equivalent to the unit $d_i$-sphere
for $1 \leq i \leq n$. If $k=2$ or $k=4$, then each continuous map from
$X$ into the unit $k$-sphere $\SB^k$ can be approximated by
stratified-regular maps.
\end{corollary}

\begin{proof}
Recall that $d(\F) = \dim_{\R}\F$, and $\G_1(\F^2)$ is biregularly
isomorphic to $\SB^{d(\F)}$. Thus, Corollary~\ref{cor-8-17} is a special
case of Theorem~\ref{th-8-16}.
\end{proof}

We also have the following result analogous to
Proposition~\ref{prop-8-15}.

\begin{theorem}\label{th-8-18}
Let $X$ be a compact real algebraic variety with 
\begin{equation*}
\Hstr^1(X; \Z/2) = H^1(X; \Z/2).
\end{equation*}
Assume that for a positive integer $k$, the cohomology
group $H^{2k}(X;\Z)$ is generated by the cup products of cohomology
classes belonging to $H^1(X;\Z)$. Then
\begin{equation*}
\HCstr^{2k}(X;\Z) = H^{2k}(X;\Z).
\end{equation*}
\end{theorem}

\begin{proof}
Each cohomology class in $H^1(X;\Z)$ is of the form $f^*(s_1)$, where $f
\colon X \to \SB^1$ is a continuous map and $s_1$ is a generator of the
group $H^1(X; \Z) \cong \Z$, cf. \cite[pp.~425, 428]{bib42}.
Consequently, the group $H^{2k}(X; \Z)$ is generated by the cohomology
classes of the form $g^*(t_{2k})$, where $g$ is a continuous map from
$X$ into the $2k$-torus $\T^{2k} = \SB^1 \times \cdots \times \SB^1$,
and $t_{2k}$ is a generator of the group $H^{2k}(\T^{2k}; \Z) \cong \Z$.
In view of Corollary~\ref{cor-7-6} and the equality $\Hstr^1(X;\Z/2) =
H^1(X; \Z/2)$, we may assume that the map $g$ is stratified-regular. By
Proposition~\ref{prop-8-15}, $\HCstr^{2k}(\T^{2k}; \Z) =
H^{2k}(\T^{2k};\Z)$. In order to complete the proof it suffices to apply
Propositions~\ref{prop-8-1} and~\ref{prop-8-2}.
\end{proof}

\begin{corollary}\label{cor-8-19}
Let $X$ be a compact real algebraic variety with
\begin{equation*}
\Hstr^1(X;\Z/2) = H^1(X;\Z/2).
\end{equation*}
Assume that $X$ is homotopically equivalent to the
$n$-torus $\SB^1 \times \cdots \times \SB^1$ ($n$ factors). Then
\begin{equation*}
\HCstreven(X;\Z) = \Heven(X;\Z).
\end{equation*}
If $\F = \CB$ or $\F = \HB$, then each topological $\F$-vector bundle on
$X$ admits a stratified-algebraic structure. If $\xi$ is a topological
$\R$-vector bundle on $X$, then the direct sum $\xi \oplus \xi$ admits a
stratified-algebraic structure.
\end{corollary}

\begin{proof}
Obviously, $\HCstr^0(X; \Z) = H^0(X;\Z) \cong \Z$, the group $H^*(X;\Z)$
has no torsion, and the $\Z$-algebra $H^*(X; \Z)$ is generated by
$H^1(X;\Z)$. It suffices to apply Theorem~\ref{th-8-18},
Theorem~\ref{th-8-11} and Corollaries~\ref{cor-8-12} and~\ref{cor-8-13}.
\end{proof}

We conclude this section by giving a different description of the
cohomology group $\HCstr^2(X;\Z)$, which in turn leads to a new
interpretation of our results on $\CB$-line bundles. First some
topological facts will be recalled for the convenience of the reader.

Let $M$ be a smooth manifold and let $N$ be a smooth submanifold of $M$
of codimension~$k$. By convention, submanifolds are assumed to be closed
subsets of the ambient manifold. Assume that the normal bundle of $N$ in
$M$ is oriented and denote by $\tau_N^M$ the Thom class of $N$ in the
cohomology group $H^k(M, M \setminus N; \Z)$, cf. \cite[p.~118]{bib38}.
The image of $\tau_N^M$ by the restriction homomorphism $H^k(M, M
\setminus N; \Z) \to H^k(M;\Z)$, induced by the inclusion map $M
\hookrightarrow (M; M \setminus N)$, will be denoted by $[N]^M$ and
called the cohomology class represented by $N$. If $M$ is compact and
oriented, and $N$ is endowed with the compatible orientation, then
$[N]^M$ is up to sign Poincar\'e dual to the homology class in $H_*(M;
\Z)$ represented by $N$, cf. \cite[p.~136]{bib38}.

Let $P$ be a smooth manifold and let $Q$ be a smooth submanifold of $P$.
Let $f \colon M \to P$ be a smooth map transverse to $Q$. If the normal
bundle of $Q$ in $P$ is oriented and the normal bundle of the smooth
submanifold $N \coloneqq f^{-1}(Q)$ of $M$ is endowed with the
orientation induced by $f$, then $\tau_N^M = f^*(\tau_Q^P)$, where $f$
is regarded as a map from $(M, M \setminus N)$ into $(P, P \setminus Q)$
(this follows form \cite[p.~117, Theorem~6.7]{bib27}). In particular,
${[N]^M = f^*([Q]^P)}$.

Let $\xi$ be an oriented smooth $\R$-vector bundle of rank $k$ on $M$. If $s
\colon M \to \xi$ is a smooth section transverse to the zero section,
and the normal bundle of ${Z(s) = \{ x \in M \mid s(x) = 0 \}}$ is endowed
with the orientation induced by $s$ from the orientation of $\xi$ (cf.
the proof of Lemma~\ref{lem-6-1}), then
\begin{equation*}
e(\xi) = [Z(s)]^M,
\end{equation*}
where $e(\xi)$ stands for the Euler class of $\xi$. Indeed, let $E$ be
the total space of $\xi$ and $p \colon E \to M$ the bundle projection. 
Identify $M$
with the image of the zero section of $\xi$. The section $s$ is
transverse to $M$ and $Z(s) = s^{-1}(M)$. Consequently, $[Z(s)]^M =
s^*([M]^E)$. Hence
\begin{equation*}
p^*([Z(s)]^M) = p^*(s^*([M]^E)) = (s \circ p)^* ([M]^E) = [M]^E,
\end{equation*}
where the last equality holds since $s \circ p \colon E \to E$ is
homotopic to the identity map. On the other hand, $p^*(e(\xi)) = [M]^E$,
cf. \cite[p.~98]{bib38}. It follows that $e(\xi) = [Z(s)]^M$ since $p^*$
is an isomorphism.

If $\xi$ is a $\CB$-vector bundle of rank $k$, then
\begin{equation*}
e(\xi_{\R}) = c_k(\xi),
\end{equation*}
where $\xi_{\R}$ is endowed with the orientation determined by the
complex structure of $\xi$, cf. \cite[p.~158]{bib38}.

\begin{lemma}\label{lem-8-20}
Let $M$ be a smooth manifold, and let $N$ be a smooth codimension $2$
submanifold of $M$ with oriented normal bundle. Let $\lambda$ be a
smooth $\CB$-line bundle on $M$ with $c_1(\lambda) = [N]^M$. Then there
exists a smooth section $s \colon M \to \lambda$ transverse to the zero
section and satisfying $Z(s)=N$.
\end{lemma}

\begin{proof}
Let $\tau = (T, \rho, N)$ be a tubular neighborhood of $N$ in $M$. The
smooth section $v \colon T \to \rho^*\tau$, defined by $v(x) = (x,x)$
for all $x$ in $T$, is transverse to the zero section and satisfies
$Z(v)=N$. The differential of $v$ induces an isomorphism between the
normal bundle to $N$ in $M$ and $(\rho^*\tau)|_N \cong \tau$ (cf. the
proof of Lemma~\ref{lem-6-1}). Via this isomorphism, $\tau$ is endowed
with an orientation. Actually, $\tau$ can be regarded as a smooth
$\CB$-line bundle, being oriented of rank $2$ (recall that $SO(2) \cong
U(1)$). The restriction of $\rho^*\tau$ to $T \setminus N$ is a trivial
smooth $\CB$-line bundle. Consequently, the $\CB$-line bundle
$\rho^*\tau$ on $T$ and the standard trivial $\CB$-line bundle
$\varepsilon$ on $M \setminus N$ can be glued over $T \setminus N$. The
resulting smooth $\CB$-line bundle $\mu$ on $M$ has a smooth section $w
\colon M \to \mu$, obtained by gluing $v$ and a nowhere zero section of
$\varepsilon$, which is transverse to the zero section and satisfies
$Z(w)=N$. It follows from the facts recalled before Lemma~\ref{lem-8-20}
that $c_1(\mu) = [N]^M$. The proof is complete since the smooth
$\CB$-line bundles $\lambda$ and $\mu$ are isomorphic.
\end{proof}

We now return to real algebraic geometry. Let $X$ be a compact
nonsingular real algebraic variety. We say that a cohomology class $u$
in $H^2(X;\Z)$ is \emph{adapted} if it is of the form $u=[Z]^X$, where
$Z$ is a nonsingular Zariski locally closed subvariety of $X$ of
codimension $2$, which is closed in the Euclidean topology and whose
normal bundle in $X$ is oriented.

Recall that the definition of an adapted smooth vector bundle is given
in Section~\ref{sec-1}, and then extended to topological vector bundles
in Section~\ref{sec-6} (before Theorem~\ref{th-6-9}).

\begin{proposition}\label{prop-8-21}
For a topological $\CB$-line bundle $\xi$ on a compact nonsingular real
algebraic variety, the following conditions are equivalent:
\begin{conditions}
\item\label{prop-8-21-a} The $\CB$-line bundle $\xi$ is adapted.

\item\label{prop-8-21-b} The cohomology class $c_1(\xi)$ is adapted.
\end{conditions}
\end{proposition}

\begin{proof}
We may assume that the $\CB$-line bundle $\xi$ is smooth. It follows
form the definition of an adapted $\CB$-line bundle that
(\ref{prop-8-21-a}) implies (\ref{prop-8-21-b}). If (\ref{prop-8-21-b})
holds, then $\xi$ is adapted in view of Lemma~\ref{lem-8-20}.
\end{proof}

Denote by $G(X)$ the subgroup of $H^2(X;\Z)$ generated by all adapted
cohomology classes.

\begin{theorem}\label{th-8-22}
For any compact nonsingular real algebraic variety $X$,
\begin{equation*}
G(X) \subseteq \HCstr^2(X; \Z).
\end{equation*}
In particular, if $\xi$ is a topological $\CB$-line bundle on $X$ with
$c_1(\xi)$ in $G(X)$, then $\xi$ admits a stratified-algebraic
structure.
\end{theorem}

\begin{proof}
The inclusion $G(X) \subseteq \HCstr^2(X;\Z)$ follows form
Theorem~\ref{th-1-8} (with $\F=\CB$) and Propositions~\ref{prop-8-6}
and~\ref{prop-8-21}. Hence, the second assertion is a consequence of
Proposition~\ref{prop-8-6}.
\end{proof}

Theorem~\ref{th-8-22} is of interest since the group $G(X)$ is easier to
compute directly than the group $\HCstr^2(X;\Z)$. Adapted cohomology
classes can also be used to give a complete description of
$\HCstr^2(X;\Z)$.

\begin{theorem}\label{th-8-23}
Let $X$ be a compact real algebraic variety. For a cohomology class $u$
in $H^2(X; \Z)$, the following conditions are equivalent:
\begin{conditions}
\item The cohomology class $u$ belongs to $\HCstr^2(X;\Z)$.

\item There exists a filtration $\FC = (X_{-1}, X_0, \ldots, X_m)$ of
$X$ with $\overline{\FC}$ a nonsingular stratification, and for each $i
= 0, \ldots, m$, there exists a multiblowup $\pi_i \colon X'_i \to X_i$
over $X_{i-1}$ such that $X'_i$ is a nonsingular variety and the
cohomology class $\pi_i^*(u|_{X_i})$ in $H^2(X'_i;\Z)$ is adapted. Here
$u|_{X_i}$ is the image of $u$ under the homomorphism $H^2(X;\Z) \to
H^2(X_i;\Z)$ induced by the inclusion map $X_i \hookrightarrow X$.
\end{conditions}
\end{theorem}

\begin{proof}
Let $\FC = (X_{-1}, X_0, \ldots, X_m)$ be a filtration of $X$ with
$\overline{\FC}$ a nonsingular stratification, and for each $i = 0,
\ldots, m$, let $\pi_i \colon X'_i \to X_i$ be a multiblowup over
$X_{i-1}$ such that $X'_i$ is a nonsingular variety. If $\xi$ is a
topological $\CB$-line bundle on $X$ with $c_1(\xi) = u$, then
${c_1(\xi|_{X_i}) = u|_{X_i}}$ and $c_1(\pi_i^*(\xi|_{X_i})) =
\pi_i^*(u|_{X_i})$ for $0 \leq i \leq m$. In view of
Proposition~\ref{prop-8-21}, the $\CB$-line bundle $\pi_i^*(\xi|_{X_i})$
on $X'_i$ is adapted if and only if the cohomology class
$\pi_i^*(u|_{X_i})$ is adapted. Consequently, it suffices to combine
Theorem~\ref{th-6-9} (with $\F=\CB$) and Proposition~\ref{prop-8-6}.
\end{proof}

\cleardoublepage
\phantomsection
\addcontentsline{toc}{section}{\refname}

\end{document}